\documentclass[10pt]{article}

\pdfoutput=1

\usepackage{amsthm}
\usepackage{mystyle}
\usepackage{mynotations}
\usepackage{lipsum}
\usepackage{epstopdf}
\usepackage{url}

\newtheorem{lemma}{Lemma}
\DeclareGraphicsExtensions{.eps,.pdf,.png,.jpg}

\author{
\begin{tabular}{c}
Paul Catala \\
DMA, ENS \\
{\small \url{pcatala@ens.fr}}
\end{tabular}
\begin{tabular}{c}
Vincent Duval\\
MOKAPLAN, Inria Paris \\
{\small \url{vincent.duval@inria.fr}}
\end{tabular} 
\begin{tabular}{c}
Gabriel Peyr\'e \\
DMA, ENS \\
{\small \url{gabriel.peyre@ens.fr}}
\end{tabular}
}

\title{A Low-Rank Approach \\ to Off-The-Grid Sparse Super-Resolution}

\begin{document}

\maketitle

\begin{abstract}
We propose a new solver for the sparse spikes super-resolution problem over the space of Radon measures. A common approach to off-the-grid deconvolution considers semidefinite (SDP) relaxations of the total variation (the total mass of the absolute value of the measure) minimization problem. The direct resolution of this SDP is however intractable for large scale settings, since the problem size grows as $f_c^{2d}$ where $f_c$ is the cutoff frequency of the filter and $d$ the ambient dimension. Our first contribution is a Fourier approximation scheme of the forward operator, making the TV-minimization problem expressible as a SDP. Our second contribution introduces a penalized formulation of this semidefinite lifting, which we prove to have low-rank solutions. Our last contribution is the FFW algorithm, a Fourier-based Frank-Wolfe scheme with non-convex updates. FFW leverages both the low-rank and the Fourier structure of the problem, resulting in an $O(f_c^d \log f_c)$ complexity per iteration. Numerical simulations are promising and show that the algorithm converges in exactly $r$ steps, $r$ being the number of Diracs composing the solution.
\end{abstract}

\section{Introduction}
\label{Introduction}


Sparse super-resolution problems consist in recovering pointwise sources from low-resolution and possibly noisy measurements, a typical example being deconvolution. Such issues arise naturally in fields like astronomical imaging \cite{puschmann05}, fluorescence microscopy \cite{hess06, rust06} or seismic imaging \cite{khaidukov04}, where it may be crucial to correct the physical blur introduced by sensing devices, due to diffraction or photonic noise for instance. They are also related to several statistical problems, for instance compressive statistical learning \cite{gribonval17}, or Gaussian mixture estimation \cite{poon17}.

The tasks we target in this work are framed as inverse problems over the Banach space of bounded Radon measures $\Radon{\Torus^d}$ on the $d$-dimensional torus $\Torus^d = \RR^d/\ZZ^d$. We aim at retrieving a discrete Radon measure $\mu_0 = \sum_{k=1}^r a_{0,k} \de_{x_{0,k}}$ ($a_{0,k} \in \CC$, and $x_{0,k} \in \TT^d$), given linear measurements $y$ in some separable Hilbert space $\Hilbert$, with
\eql{\label{eq:obs} 		
	y = \Phi \mu_0 + w = y_0 + w \in \Hilbert
}%
where $\Phi: \Radon{\Torus^d} \rightarrow \Hilbert$ is a known linear operator, and $w \in \Hilbert$ some unknown noise.

\subsection{Beurling LASSO}

Although this problem is severly ill-posed, sparse estimates are obtained by solving the following optimization program, known as the BLASSO \cite{castro12}
\eq{\label{eq:blasso}\tag{$\Pp_\la(y)$}		
	\uargmin{\mu \in \Radon{\Torus^d}} \frac{1}{2\la}\normH{\Phi \mu - y}^2 + \normtv{\mu}{\Torus^d}
}%
where the total variation norm $\normtv{\mu}{\Torus^d}$ of a measure $\mu \in \Radon{\Torus^d}$ is defined as
\eq{\label{eq:tvnorm}		
	\normtv{\mu}{\Torus^d} \eqdef \sup\enscond{\Re\Big(\int_{\Torus^d} \overline{\eta} \diff \mu\Big)}{\eta \in \Cont{\Torus^d}, \quad \normi{\eta} \leq 1}.
}
and extends the $\lun$-norm to the infinite-dimensional space of measures, favoring in particular the emergence of Dirac masses in the solution. 
The parameter $\la$ should be adapted to the noise level; in this work we focus on the case where $\la > 0$.

The main asset here is that no assumption is made about the support of the measure, contrary to early superresolution models where it was assumed to live on some discrete grid \cite{donoho92, dossal05}. \textcolor{black}{This grid-free approach has been at the core of several recent works \cite{bredies13,candes14,castro17}, and has offered beneficial mathematical insight on the problem, leading to a better understanding of the impact of minimal separation distance \cite{tang15} or of the signs of the spikes \cite{candes14,castro12}, and to sharp criteria for stable spikes recovery \cite{denoyelle15, duval17}}.
%
%

However,  the infinite-dimensionality of \eqref{eq:blasso} poses a major numerical challenge, and off-the-grid super-resolution, although analytically better, remains difficult to perform in practice. This paper introduces a new mehod to solve this optimization problem in a scalable way.

\subsection{Related works}
Several approaches to solve \eqref{eq:blasso} have been proposed in the literature.
\paragraph{Greedy approaches} In \cite{bredies13}, the authors propose to iteratively estimate the support of the initial measure using a conditional gradient (also known as Frank-Wolfe) scheme, which consists in greedy stepwise additions of new spikes. Frank-Wolfe algorithm is well fitted to operate directly over measures and is relatively inexpensive, but converges slowly. To remedy the matter, one can add non-convex corrective updates as a final step \cite{bredies13, boyd15}. This achieves state-of-the-art result in several applications, among which fluorescence microscopy \cite{boyd15}. 

\paragraph{Semidefinite approaches} \eqref{eq:blasso} may also be solved by lifting it to a semidefinite program. Semidefinite approaches for the total variation minimization problem were originally introduced in \cite{candes14,tang13}, but for unidimensional measures (\ie $d=1$) only. The multivariate case on the other hand raises a far more challenging problem, that may be solved using the \textcolor{black}{so-called} Lasserre's hierarchy \cite{lasserre01}, consisting in a sequence of increasingly better semidefinite approximations of \eqref{eq:blasso} or its dual \cite{castro17,Dumitrescu}.

\textcolor{black}{The numerical interest of semidefinite programs holds in that} they benefit from efficient solvers; however, these are limited to matrices of size a few hundreds. In our case this represents a major impediment, since the matrices involved in the semidefinite relaxations of \eqref{eq:blasso} are typically of size $f_c^d \times f_c^d$, where $f_c$ is the cutoff frequency of the filter. Up to now, semidefinite relaxations have essentially been successful for combinatorial problems like the MaxCut problem \cite{goemans95}, in which the number of variables is high and the degree of the polynomials involved in the constraints is low (it is equal to 2 in the MaxCut case). Imaging problems on the contrary typically involve polynomials with low numbers of variables (essentially the same as the dimension $d$), but high degrees ($f_c$), and to our knowledge, no efficient way of solving the hierarchy in those settings has been proposed yet. 

\paragraph{Non variational approaches} Although we focus here on $\lun$-regularization techniques, there is also a vast literature on non-convex or non-varational superresolution schemes. Many of these methods derive from the same idea proposed by Prony to encode the positions of the spikes as zeros of some polynomial, see \cite{krim96} for a review. This is the case for MUSIC \cite{schmidt86}, ESPRIT \cite{roy89}, or finite rate of innovation \cite{wei16}, among others. They remain difficult to extend to multivariate settings, see for instance \cite{kunis16}.

\subsection{Contributions}

Our first contribution, detailed in Section~\ref{sec:fourier-approximation}, is a spectral approximation scheme of the forward operator. This is a cornerstone of our approach. The approximation result is formulated in Prop.~\ref{prop:F-approximation}.

Next, in Section~\ref{sec:sdp-hierarchies}, we propose a framework of Lasserre-inspired semidefinite liftings for \eqref{eq:blasso}, that generalizes the approach of \cite{tang13} to multi-dimensional settings. More specifically, we derive a hierarchy of semidefinite programs approaching \eqref{eq:blasso} in arbitrary dimension (Prop.~\ref{prop:relaxation}). Then, following works by Curto and Fialkow \cite{curto96} and Laurent \cite{laurent10}, Theorem~\ref{thm:flatness} provides a practical criterion to detect the collapsing of the hierarchy.

Our third and main contribution is the FFW algorithm, detailed in Section~\ref{sec:ffw}, a Fourier-based Frank-Wolfe approach for these liftings, which requires only $O(f_c^d \log f_c)$ elementary computations per iterations, making it scalable for super-resolution problems in dimension greater than one. 

The code to reproduce the numerical results is available online\footnote{\url{https://github.com/Paulcat/Super-Resolution-SDP}}.


\subsection{Notations}

\paragraph{Measures} We consider measures defined over the torus $\Torus^d = (\RR/\ZZ)^d$. Note that in practice, in order to avoid periodization artifacts, one can use zero-padding or symmetrization at the boundary of the signal and still consider a periodic setting. We denote by $\Radon{\Torus^d}$ the space of bounded Radon measures on $\Torus^d$, endowed with its weak-* topology. It is the topological dual of the space $\Cont{\Torus^d}$ of continuous functions on $\TT^d$. 

\paragraph{Linear operators} 
We consider a linear operator $\Phi: \Radon{\Torus^d} \rightarrow \Hilbert$ of the form
\eql{\label{eq:fwdop} 
	\Phi: \mu \in \Radon{\Torus^d} \mapsto \int_{\Torus^d} \phi(x)\diff\mu(x), 
}
where $\phi: \Torus^d \rightarrow \Hh$ is assumed to be smooth. We denote by $\Ll$ the set of operators of this form. A typical instance is a \emph{convolution} operator, where $\phi(x) = \tilde{\phi}(\cdot - x)$, with $\tilde{\phi} \in \Ldeux(\Torus^d)$ and $\Hilbert = \Ldeux(\Torus^d)$. An important example is \emph{ideal low-pass filtering}, as considered in \cite{candes14}, which is a convolution with the Dirichlet kernel $\tilde{\phi}_D(x) \eqdef \sum_{k \in \Om_c} e^{2i\pi\dotp{k}{x}}$, where
\eq{
	\Om_c = \segi{-f_c}{f_c}^d,
}
for some cutoff frequency $f_c \in \NN^*$. This is equivalently modeled in the Fourier domain by taking $\phi(x) = (e^{-2i\pi\dotp{k}{x}})_{k \in \Om_c}$, and $\Hh = \CC^{\abs{\Om_c}}$. $\Phi$ is then simply the Fourier operator $\Ff_c$, defined as
\eql{\label{eq:Fop} 
	\Ff_c: \mu \mapsto \left( c_k(\mu) \right)_{k \in \Om_c}, \qwhereq c_k(\mu) = \int_{\Torus^d} e^{-2i\pi\dotp{k}{x}}\diff\mu(x).
}
Other commonly encountered imaging problems involve non translation-invariant operators, such as \emph{subsampled convolution}. Given a sampling domain $S \subset \Torus^d$, this is modelled as $\phi(x) = (\psi(s - x))_{s \in S}$. In that case, $\Hilbert = \CC^{\abs{S}}$. Lastly, we also consider in this paper more difficult settings, where the kernel is defined as $\phi(x) = (\psi(s,x))_{s \in S}$, with $\Hilbert = \CC^{\abs{S}}$. An example of non translation-invariant operator, studied in Section~\ref{sec:fourier-approx:examples}, are foveated measurements \cite{chang00}.


\paragraph{Trigonometric polynomials, Laurent polynomials} 

A multivariate Laurent polynomial of degree $\ell$ is defined as
\eq{
	z \in \CC^d \mapsto \sum_{k \in \Om_\ell} p_k z^k,
}
where $z^k$ must be understood as $z_1^{k_1}z_2^{k_2} \ldots z_d^{k_d}$, and $\Om_{\ell} = \segi{-\ell}{\ell}^d$. We write $\CC_\ell[\conj{\Z}^\pm]$ the set of Laurent polynomials of degree $\ell$. The restriction to the unit circle of a Laurent polynomial is a trigonometric polynomial.

When working with multivariate polynomials, one need to choose some ordering on the monomials. Typical examples of monomial orderings are the lexicographic order, or the graded lexicographic order. In our case, we use the \emph{colexicographic order}.

\begin{defn}[Colexicographical ordering] Given two multi-indices $\al, \be \in \ZZ^d$, one has $\al \prec_{colex} \be$ if $\al_d < \be_d$, or $\al_d = \be_d$ and $(\al_1, \ldots, \al_{d-1}) \prec_{colex} (\be_1, \ldots,\be_d)$.
\end{defn}

\section{Fourier approximation of operators}
\label{sec:fourier-approximation}
In this section we detail a spectral approximation method for the forward operator $\Phi$, that is consistent with semidefinite approaches (see Section~\ref{sec:sdp-hierarchies}) to solve \eqref{eq:blasso}. Our model encompasses deconvolution problems, as well as non-convolution ones -- in particular, we focus on the case of subsampled convolution and foveation (see Section~\ref{sec:fourier-approx:examples}), which are common in imaging problems.

\subsection{Spectral approximation operator}
Let $\Phi \in \Ll$ be an operator of the form~\eqref{eq:fwdop}, with kernel $\varphi$ (in the following, we assume that $\varphi$ is sufficiently smooth, namely $\varphi\in \Cder{j}(\Torus^d;\Hh)$, with $j\geq \lfloor\frac{d}{2}\rfloor+1$). It is possible to define a Fourier series expansion of $\varphi$, using the coefficients
\begin{align}
  \forall k \in \ZZ^d,\quad  c_k(\phi) &\eqdef \int_{\Torus^d} e^{-2i\pi\dotp{k}{x}}\phi(x) \diff x \in \Hh.
\end{align}
We refer to the paper of Kandil~\cite{kandil83} for the theory of the Fourier series of Hilbert-valued functions. In particular, as $\Hh$ is separable, it is shown that the Parseval formula holds, \ie for all $\psi\in \Ltwo{\Torus^d;\Hh}$,
\begin{align}\label{eq:parseval}
  \int_{\Torus^d}\normH{\psi(x)}^2\diff x = \sum_{k\in\ZZ^d} \normH{c_k(\psi)}^2.
\end{align}
As a consequence, the following equality holds strongly in $\Ltwo{\Torus^d;\Hh}$,
\begin{align}
  \psi  = \sum_{k\in \ZZ^d} c_k(\psi) e_k, \qwhereq e_k: x\mapsto e^{2i\pi\dotp{k}{x}}.
\end{align}

Now, given $f_c\in \NN$, we are interested in approximating $\Phi$ with some operator $\Phi_c$ whose kernel $\varphi_c$ has spectrum supported on $\Om_c$, \ie
\eq{
\Phi_c\in	\Ll_c \eqdef \enscond{\Psi \in \Ll}{c_k(\psi) = 0 \quad \forall\, k \in \ZZ^d \setminus \Om_c}.
} 
\begin{defn}[Spectral approximation operator] The spectral appro\-xi\-ma\-tion of $\Phi$ is the operator $\Phi_c \in \Ll_c$ defined by
\eq{
	c_k(\phi_c) = \left\{
		\begin{array}{ll}
			 c_k(\phi) & \forall\, k \in \Om_c\\
			 0 & \forall\, k \in \ZZ^d \setminus \Om_c
		\end{array}
	\right..
}
\end{defn}
From Parseval's equality~\eqref{eq:parseval}, we see that $\Phi_c$ is the best approximation of $\Phi$ in $\Ll_c$ in terms of the Hilbert-Schmidt operator norm, $\Psi\mapsto \|\Psi\|\eqdef\|\psi\|_{\Ltwo{\Torus^d;\Hh}}$, \ie
\begin{align}
  \Phi_c=\argmin_{\Psi\in \Ll_c}\|\Phi-\Psi \|^2= \argmin_{\Psi\in \Ll_c} \int_{\Torus^d}\normH{\varphi(x)-\psi(x)}^2\diff x.
\end{align}

The next result shows that the solution of the BLASSO when replacing $\Phi$ with $\Phi_c$ approximate the solutions of~\eqref{eq:blasso}.


\begin{prop} Let $\Phi \in \Ll$, with $\phi\in \Cder{j}(\Torus^d;\Hh)$, $j\geq \lfloor\frac{d}{2}\rfloor+1$. For each $f_c \in \NN$, let $\Phi_c \in \Ll_c$ be its spectral approximation operator, and let $\mu_{f_c}$ be a minimizer of $\Ee_{f_c}(\mu) \eqdef \frac{1}{2\la}\norm{\Phi_c\mu - y}^2 + \normtv{\mu}{\Torus^d}$ over $\Radon{\Torus^d}$. 

Then, the sequence $(\mu_{f_c})_{f_c\in \NN}$ has accumulation points in the weak-* topology and each of them is a solution to \eqref{eq:blasso}.
\label{prop:F-approximation}
\end{prop}
\begin{proof} By definition of $\Ee_{f_c}$ and $\mu_{f_c}$,
\eq{
	\normtv{\mu_{f_c}}{\Torus^d} \leq \Ee_{f_c}(\mu_{f_c}) \leq \Ee_{f_c}(0) = \frac{1}{2\la}\normH{y}^2.
}
As the total variation ball (of radius $\frac{1}{2\la}\norm{y}^2$) is compact and metrizable for the weak-* topology, the sequence $(\mu_{f_c})_{f_c\in \NN}$ has accumulation points. Let $\mu^\star$ be any accumulation point and let us denote by $(\mu_n)_{n\in \NN}$ any subsequence which converges towards $\mu^\star$ in the weak-$*$ sense.  
We denote by $\Phi_n$ (resp. $\varphi_n$) the corresponding operator (resp.\ kernel). Let $\mu\in\Radon{\Torus^d}$ be any Radon measure. 

We note that, as $\varphi\in \Cder{j}(\Torus^d;\Hh)$ with $j\geq \lfloor\frac{d}{2}\rfloor+1$, 
\begin{align*}
  2 \!\!\sum_{k\in \ZZ^d\setminus\{0\}}\normH{c_k(\varphi)}&\leq 
  \!\sum_{k \in \ZZ^d\setminus\{0\}}\normH{c_k(\varphi)}^2(k_1^j + \cdots + k_d^j)^2  +  \!\!\sum_{k \in \ZZ^d\setminus\{0\}} \frac{1}{(k_1^j + \cdots + k_d^j)^2} \\
  &<+\infty,
\end{align*}
hence the series  $\sum_{k\in \ZZ^d} c_k(\varphi) e_k$ converges uniformly on $\Torus^d$ towards $\varphi$.
As a result, $\varphi_n$ converges uniformly towards $\varphi$, and 
\begin{align*}
  \lim_{n\to+\infty}\normH{\Phi_n \mu - y}^2= \normH{\Phi {\mu} - y}^2.
\end{align*}
Moreover, for any $h\in \Hh$,
\begin{align*}
  \left|\dotp{h}{(\Phi_n \mu_n - \Phi\mu^{\star})} \right|& \leq \normH{h}\int_{\Torus^d} \normH{\varphi_n(x)-\varphi(x)}\diff \abs{\mu_n}(x)\\
                                                          & \quad \quad +\left|\int_{\Torus^d}\dotp{h}{\varphi(x)}\diff \mu_n(x) -\int_{\Torus^d}\dotp{h}{\varphi(x)}\diff \mu^\star(x)\right|.
\end{align*}
The first term vanishes by uniform convergence of $\varphi_n$ and boundedness of $(\mu_n)_{n\in\NN}$, whereas the second one vanishes by the weak-* convergence of $\mu_n$ towards $\mu^\star$. Hence, $\Phi_n \mu_n\rightharpoonup \Phi\mu^{\star}$ weakly in $\Hh$, and 
\begin{align}
  \normH{\Phi \mu^\star-y}^2 \leq \liminf_{n\to +\infty}  \normH{\Phi_n \mu_n-y}^2.
\end{align}

To complete the proof, we note that by definition of $\mu_n$, we have
\eql{\label{eq:ineq}
  \Ee_n(\mu_n) \leq \Ee_n(\mu), 
}
and passing to the inferior limit, we get,
\begin{align*}
  \la \normtv{\mu^\star}{\Torus^d} + \frac{1}{2}\normH{\Phi \mu^\star-y}^2 & \leq \liminf_{n\to +\infty}  \Ee_n(\mu_n) \\
  & \leq \liminf_{n\to +\infty}  \Ee_n(\mu) = \la \normtv{\mu}{\Torus^d}+  \frac{1}{2}\normH{\Phi \mu-y}^2.
\end{align*}
As this is true for any $\mu\in \Radon{\Torus^d}$, we deduce that $\mu^\star$ is a solution to~\eqref{eq:blasso}.
\end{proof}


By construction, $\Phi_c\mu$ only depends on the Fourier coefficients $c_k(\mu)$ for $k \in \Om_c$. Indeed, for any $\mu \in \Radon{\Torus^d}$,
\eq{
	\begin{aligned}
		\Phi_c \mu &= \int_{\Torus^d} \phi_c(x)\diff\mu(x) = \sum_{k\in\Om_c} c_k(\phi) \int_{\Torus^d} e^{2i\pi\dotp{k}{x}}\diff\mu(x) = \sum_{k\in\Om_c} c_{-k}(\phi) (\Ff_c\mu)_k.
	\end{aligned}
}
This leads to the following definition.
\begin{defn}[Spectral approximation factorization] The spectral approximation operator may be factorized as
  \eq{
	\Phi_c = \Aa(\phi)\Ff_c,
}
where $\Aa(\phi): \CC^{\abs{\Om_c}}\rightarrow \Hh$ is defined by $\Aa(\phi)_k\eqdef c_{-k}(\varphi)$. 
In particular, when $\Hilbert$ is of finite dimension $N$, the matrix of $\Aa(\phi) \in \Mm_{N,\abs{\Om_c}}(\CC)$ is given by
\eq{
	\Aa(\phi)_{j,k} = c_{-k}(\phi_j),  \quad \forall\, 1\leq j \leq N, \quad \forall\, k \in \Om_c.
}
We call this matrix the spectral approximation matrix of $\Phi$.
\end{defn}

In the rest of the paper, we therefore focus on solving the problem
\eq{
	\umin{\mu \in \Radon{\Torus^d}} \frac{1}{2\la} \normH{\A \Ff_c \mu - y}^2 + \normtv{\mu}{\Torus^d}.
	\tag{$\Pp_\la(y)$}
}

\subsection{Examples}
\label{sec:fourier-approx:examples}
We give a few instances of problems for which the matrix $\A$ may be computed. Each case discussed below is illustrated in Fig.~\ref{fig:example-measurements}.

\begin{figure}
	\captionsetup[subfigure]{labelformat=empty}
	\subfloat[Dirichlet]{\includegraphics[width=0.25\textwidth]{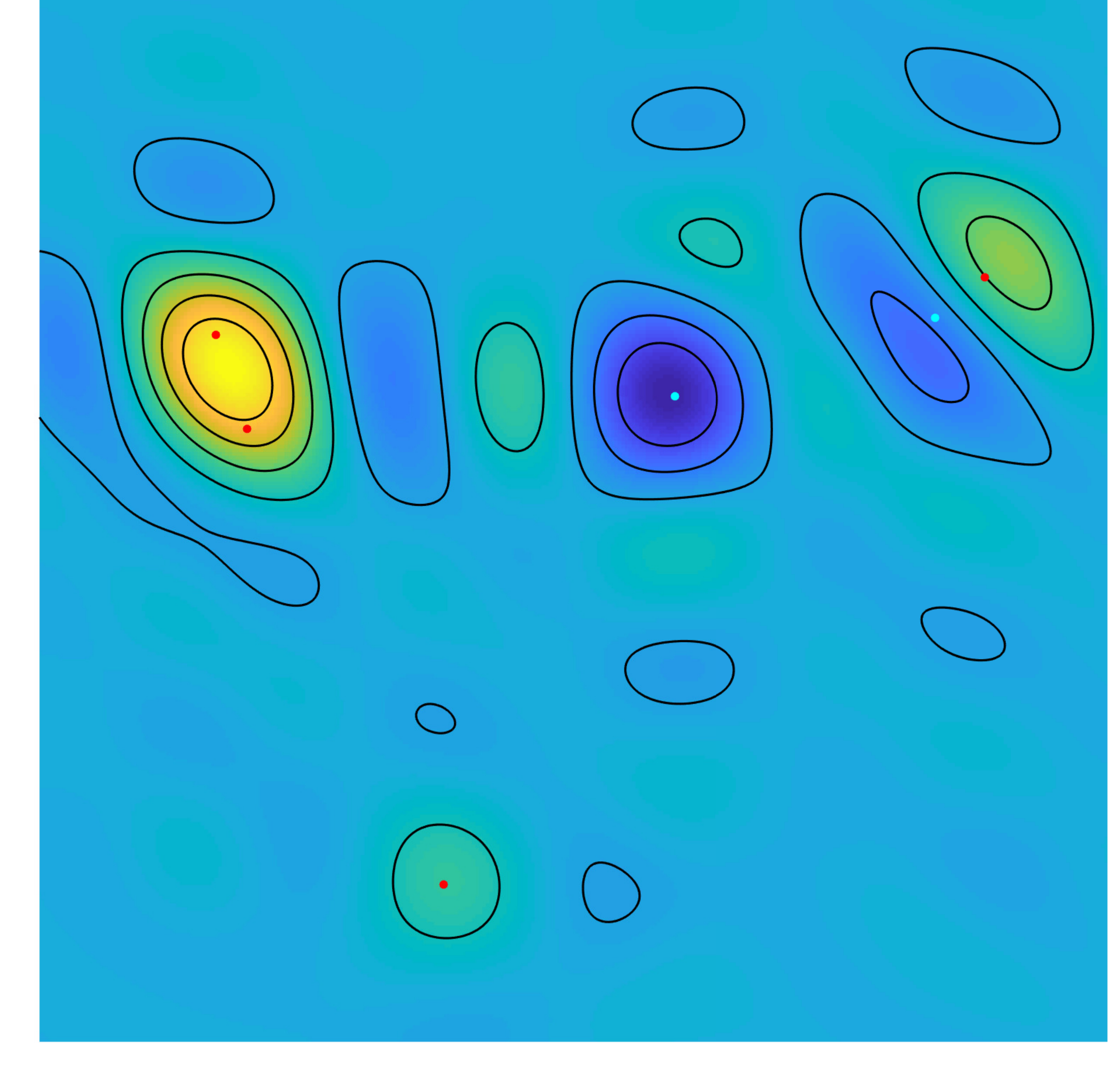}}
	\subfloat[Gaussian]{\includegraphics[width=0.25\textwidth]{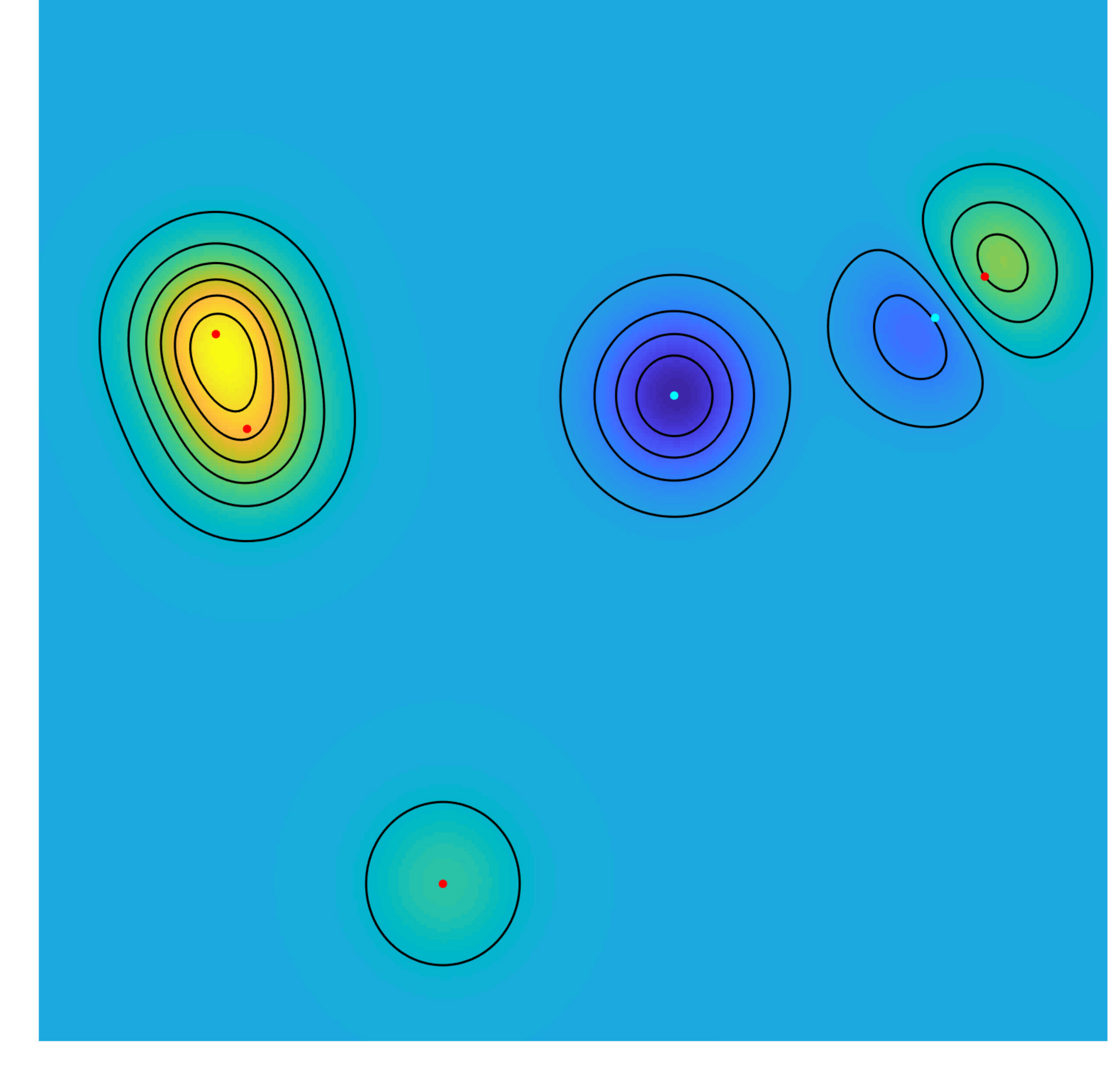}}
	\subfloat[Gaussian foveation]{\includegraphics[width=0.25\textwidth]{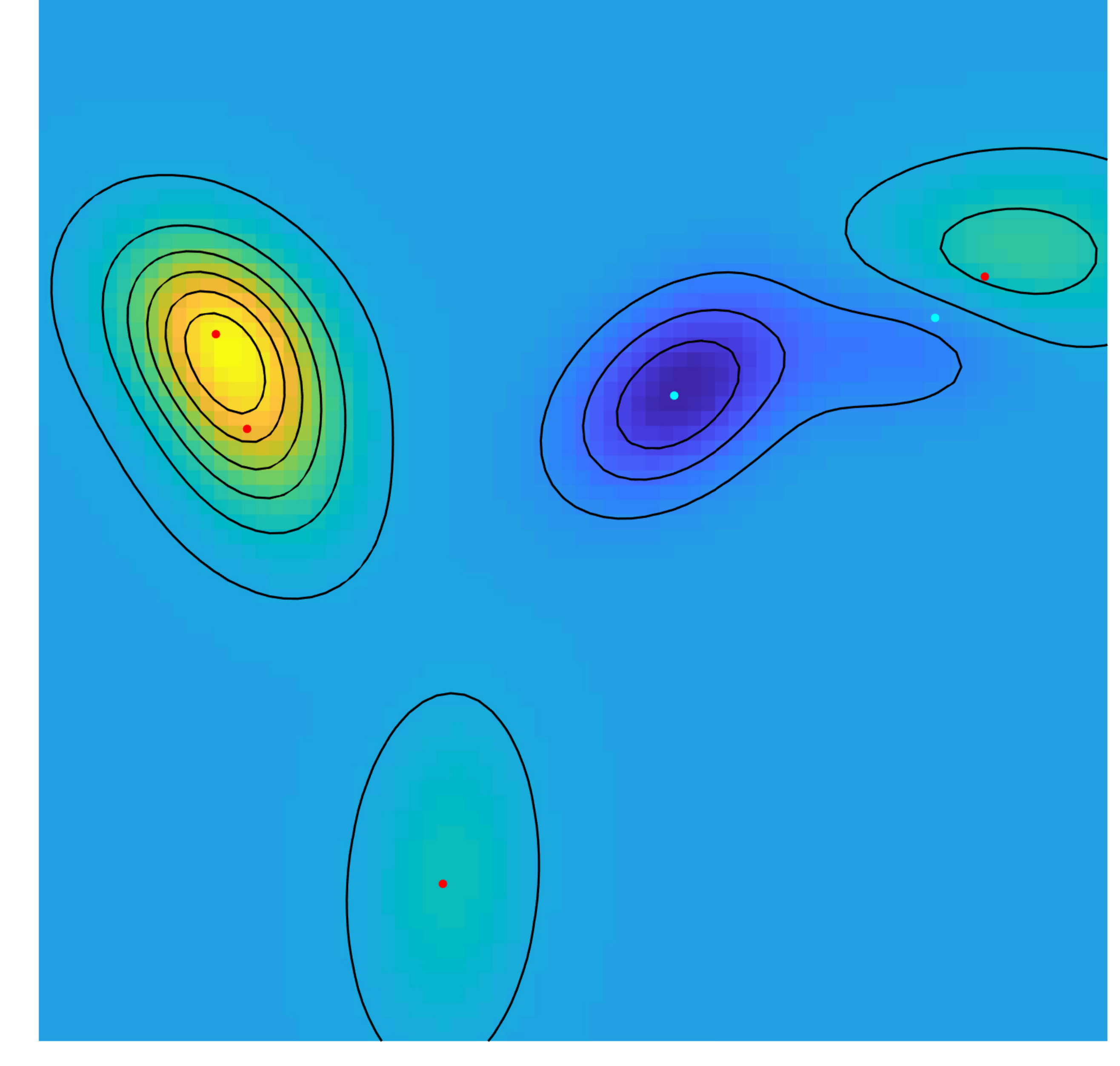}}
	\subfloat[Subsampled Gaussian]{\includegraphics[width=0.25\textwidth]{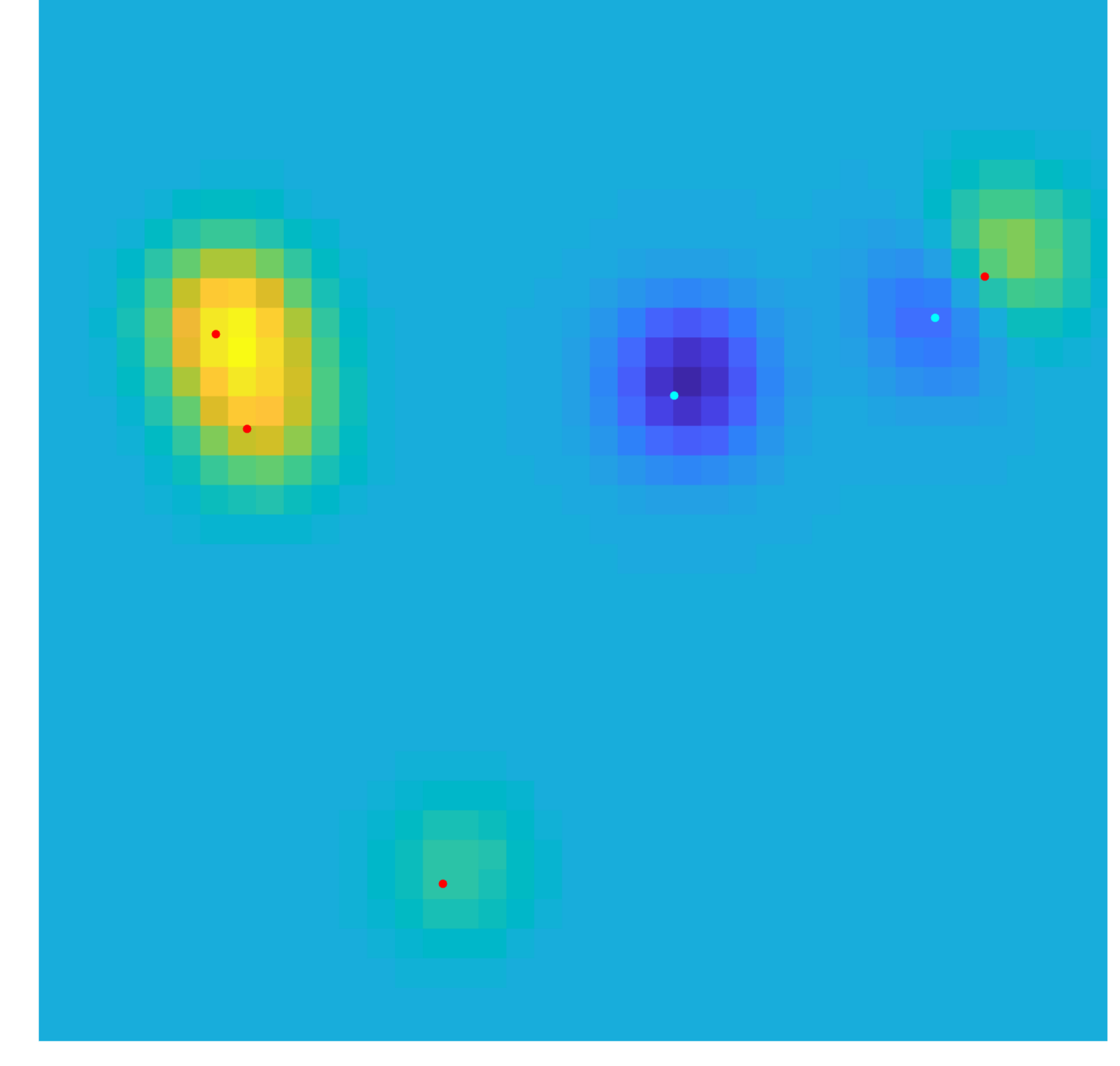}}
	\caption{Examples of measurements. In the first two cases (left), the observations $y$ live in the Fourier domain, and we plot $\Ff_c^*y$. In the last two cases (right), $y$ lives on a grid $\Gg$, and can be plotted directly.}
	\label{fig:example-measurements}
\end{figure}

\paragraph{Deconvolution} The convolution with some kernel $\tilde{\phi} \in \Ldeux(\Torus^d)$, obtained with $\phi(x) = \tilde{\phi}(\cdot - x)$, is equivalently obtained by multiplying the Fourier coefficients of the measure with those of $\tilde{\varphi}$. In other words, one might equivalently choose $\Hh=\ell^2(\ZZ^d)$, and 
\eq{
	\phi(x) = \left(c_k(\tilde{\phi})e^{-2i\pi\dotp{k}{x}}\right)_{k \in \ZZ^d}.
}
The spectral approximation matrix of a convolution operator is thus simply the diagonal matrix
\eq{
	\Aa(\phi) = \Diag~(c_k(\tilde{\phi}))_{k \in \Om_c}.
}
\begin{exmp}
Typical examples of convolution operators include the \emph{ideal low-pass filter} or the (periodized) \emph{Gaussian filter}. Ideal low-pass filtering is the case where
\eq{
	\tilde{\phi}(x) = \sum_{k=-f_c}^{f_c} e^{2i\pi\dotp{k}{x}},
}
and therefore $\phi(x) = (e^{-2i\pi\dotp{k}{x}})_{k \in \Om_c}$ and $\A = \Id$. For Gaussian filtering (with covariance $\Sigma$), the adequate definition of the kernel over the torus is 
\eq{
	\tilde{\phi}(x) = \sum_{k\in\ZZ^d} g(x+k), \qwithq g(t) \eqdef e^{-\frac{1}{2}\dotp{t}{\Sigma^{-1}t}} \quad \forall t \in \RR^d
}
which is also, by Poisson summation formula, $\tilde{\phi}(x) = \sum_{k\in\ZZ^d} \hat{g}(k)e^{-2i\pi\dotp{k}{x}}$, where $\hat{g}$ denotes the continuous Fourier transform of $g$. Hence the approximation matrix reads $\A = \Diag (\hat{g}(k))_{k\in\Om_c}$, with $\hat{g}(k) = (2\pi)^{\frac{d}{2}}(\det{\Sigma})^{\frac{1}{2}}e^{-2\pi^2\dotp{k}{\Sigma k}}$. 
The error induced by the spectral approximation in the Gaussian case is described in Fig.~\ref{fig:gauss-approx}. It can be made negligible with $f_c$ sufficiently large.
\end{exmp}

\begin{figure}
	\captionsetup[subfigure]{labelformat=empty}
	\subfloat[Gaussian]{\includegraphics[width=0.191\textwidth]{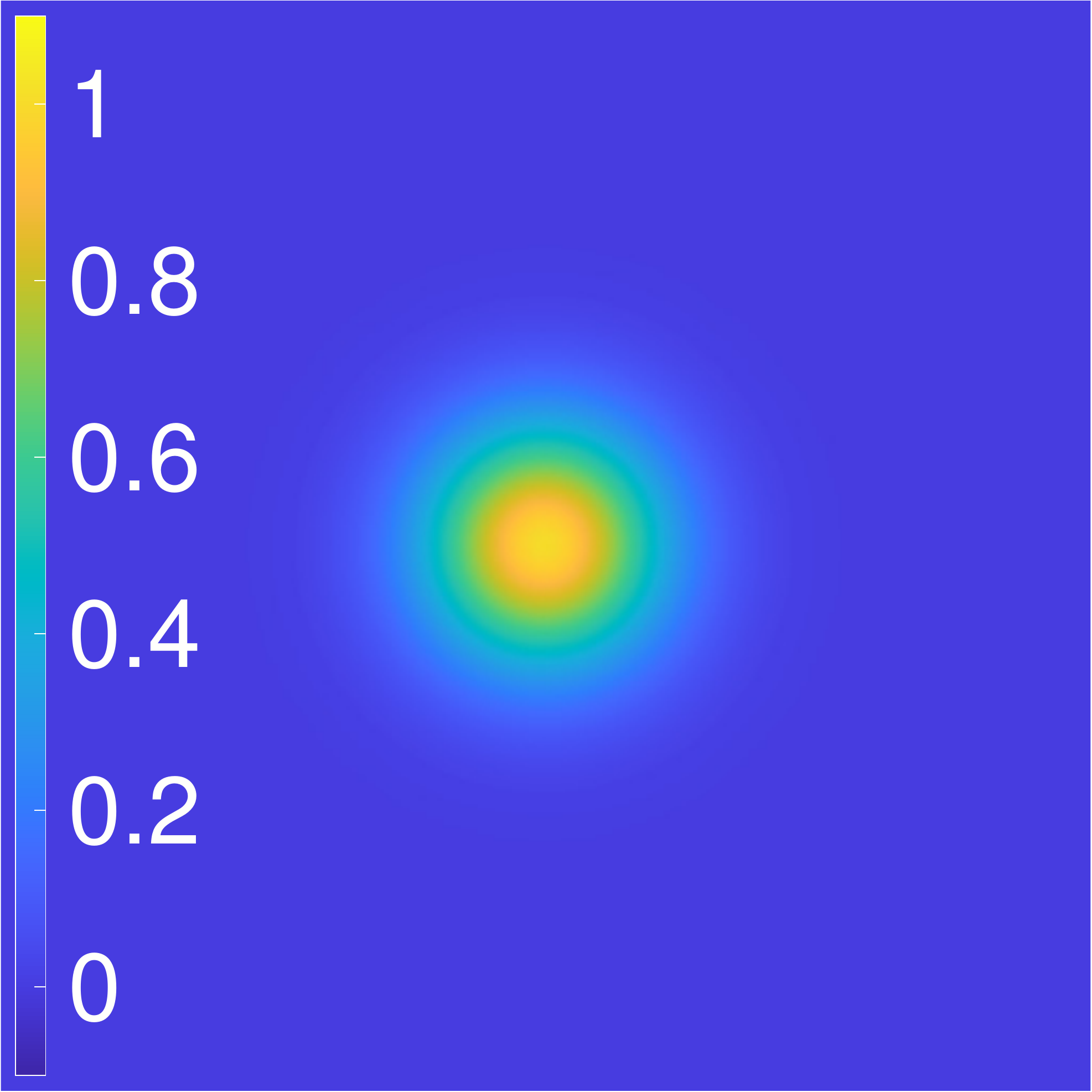}}
	\hspace{2mm}
	\subfloat[$f_c=5$]{\includegraphics[width=0.191\textwidth]{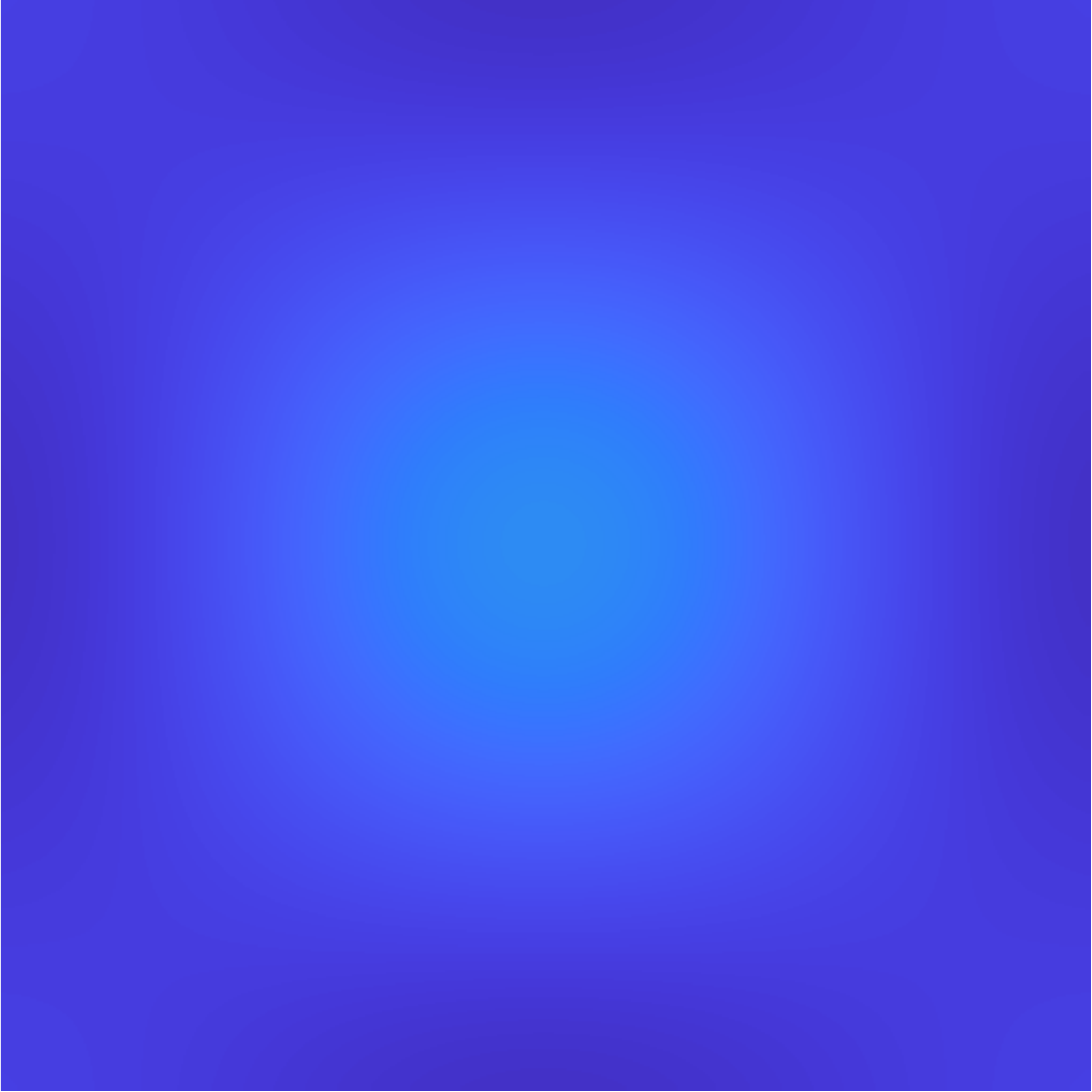}}
	\hspace{0.1mm}
	\subfloat[$f_c=10$]{\includegraphics[width=0.191\textwidth]{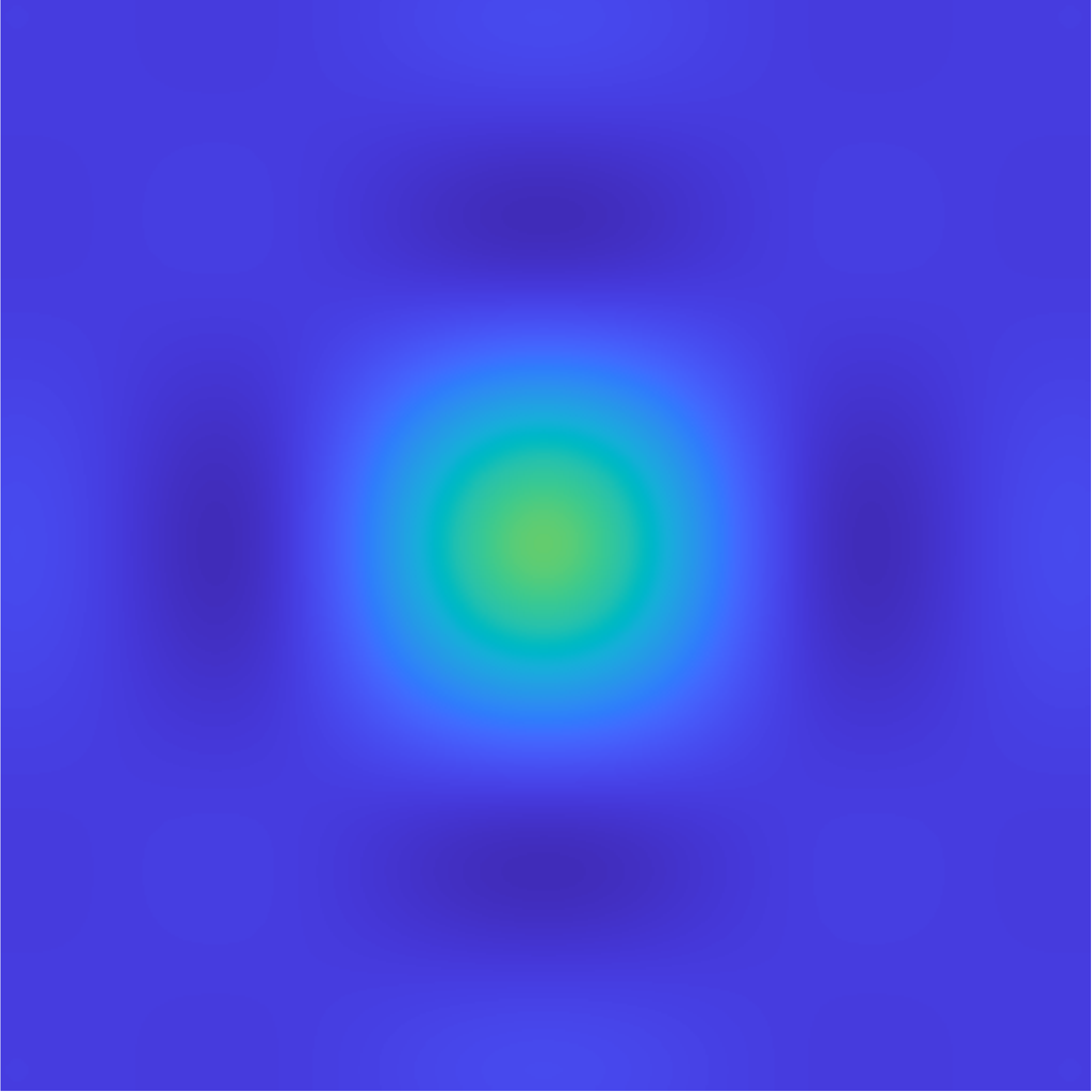}}
	\hspace{0.1mm}
	\subfloat[$f_c=15$]{\includegraphics[width=0.191\textwidth]{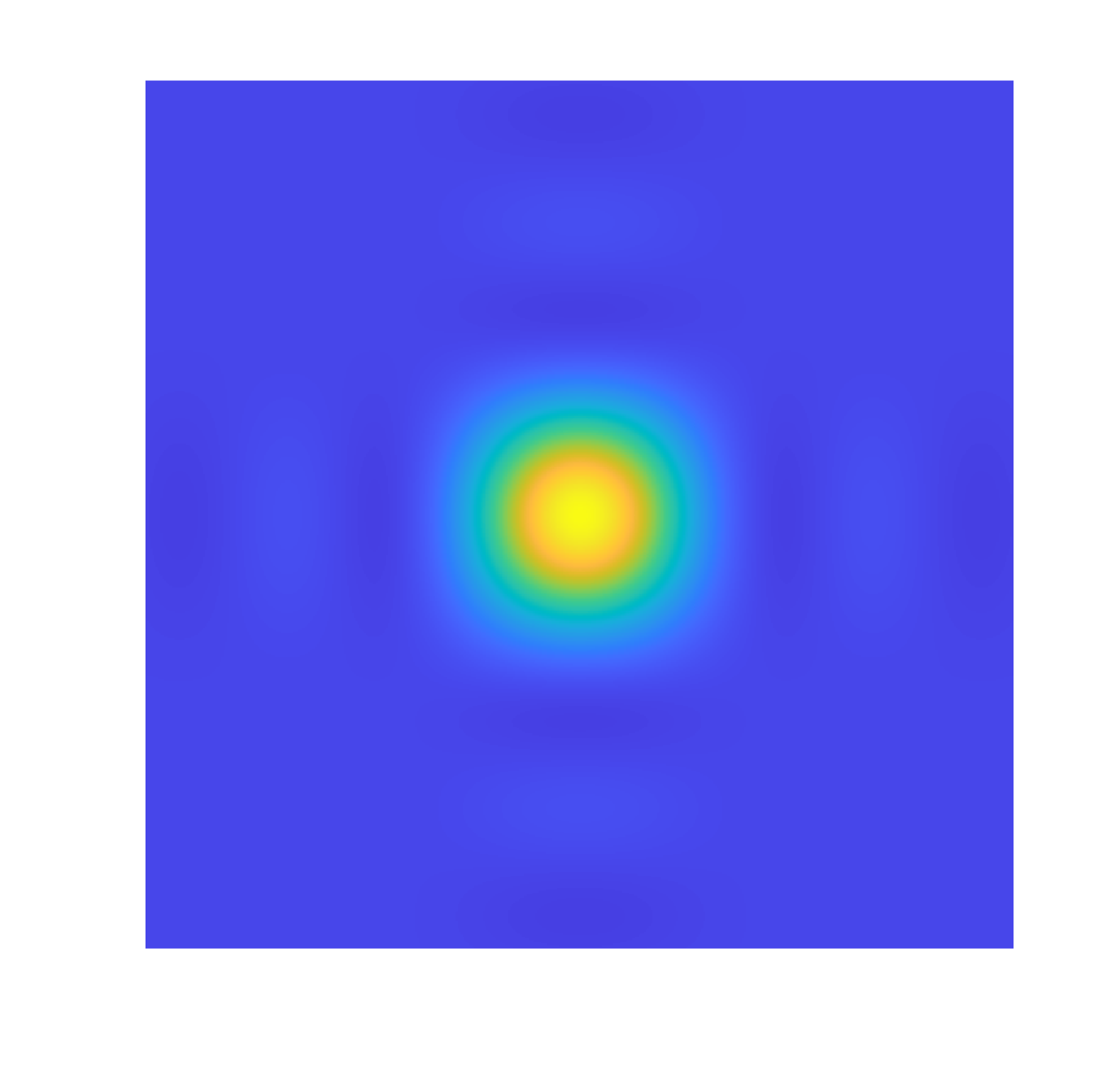}}
	\hspace{0.1mm}
	\subfloat[$f_c=30$]{\includegraphics[width=0.191\textwidth]{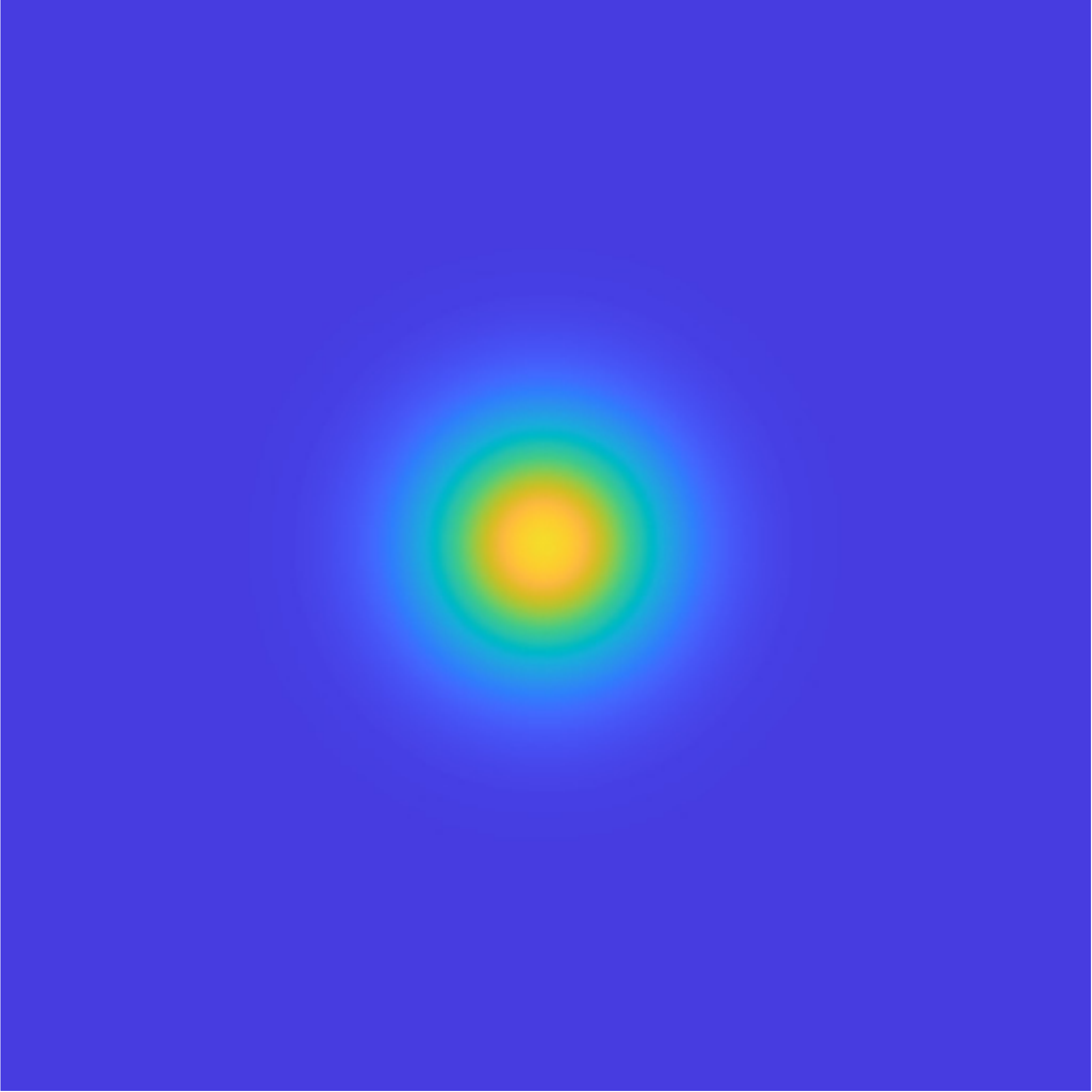}}
	\caption{Evolution of $\Phi_c \de_{x_0}$, $x_0 \in \Torus^2$, for different values of $f_c$, in the Gaussian case. Left image is a true Gaussian convolution (over $\RR^2$).}
	\label{fig:gauss-approx}
\end{figure}


\paragraph{Microscopy} In practical cases, one often only has access to convolution measurements over some sampling grid $\Gg$. In fluorescence microscopy for instance \cite{smlm}, the observations are accurately described as subsampled Gaussian measurements. In that case, given some convolution kernel $\tilde{\phi} \in \Ldeux(\Torus^d)$ (typically a Gaussian), $\phi$ may be defined as
\eql{ \label{eq:kernel-grid}
	\phi(x) = (\tilde{\phi}(s-x))_{s\in\Gg}
}
which leads to
\eq{
	\A = \left(c_k(\tilde{\phi}) e^{2i\pi \dotp{k}{t}}\right)_{t \in \Gg, k \in \Om_c}.
}

\begin{rem}
When the grid $\Gg$ is regular, multiplication by $\A$ or $\A^*$ may be computed efficiently using fast Fourier transforms, see Section~\ref{sec:fft-implemented-A}.
\end{rem}

\paragraph{Foveation} In more general cases, $\phi$ may be defined as
\eq{
	\phi(x) = (\tilde{\phi}(s,x))_{s \in \Gg},
}
in which case the lines of $\A$ consist in the Fourier coefficients of $x \mapsto \tilde{\phi}(s,x)$ at frequencies taken in $\Om_c$. Section~\ref{sec:numerics} studies foveation operators \cite{chang00}, where $\tilde{\phi}$ takes the form
\eql{\label{eq:kernel-fov}
	\tilde{\phi}(s,x) = g( \sigma^{-1}(x) (s - x)), \quad \forall s \in \Gg, \quad \forall x \in \Torus^d
}
for some smoothing function $g$, typically a Gaussian, and some (positive) covariance function $\sigma$. Fig.~\ref{fig:example-measurements} shows an example of Gaussian foveated measurements, for which the matrix $\A$ has no closed-form but may be approximated numerically using the discrete Fourier transform. The approximated foveation kernel is displayed in Fig.~\ref{fig:foveation-kernel}.

\begin{figure}
	\captionsetup[subfigure]{labelformat=empty}
	\subfloat[Foveation]{\includegraphics[width=0.191\textwidth]{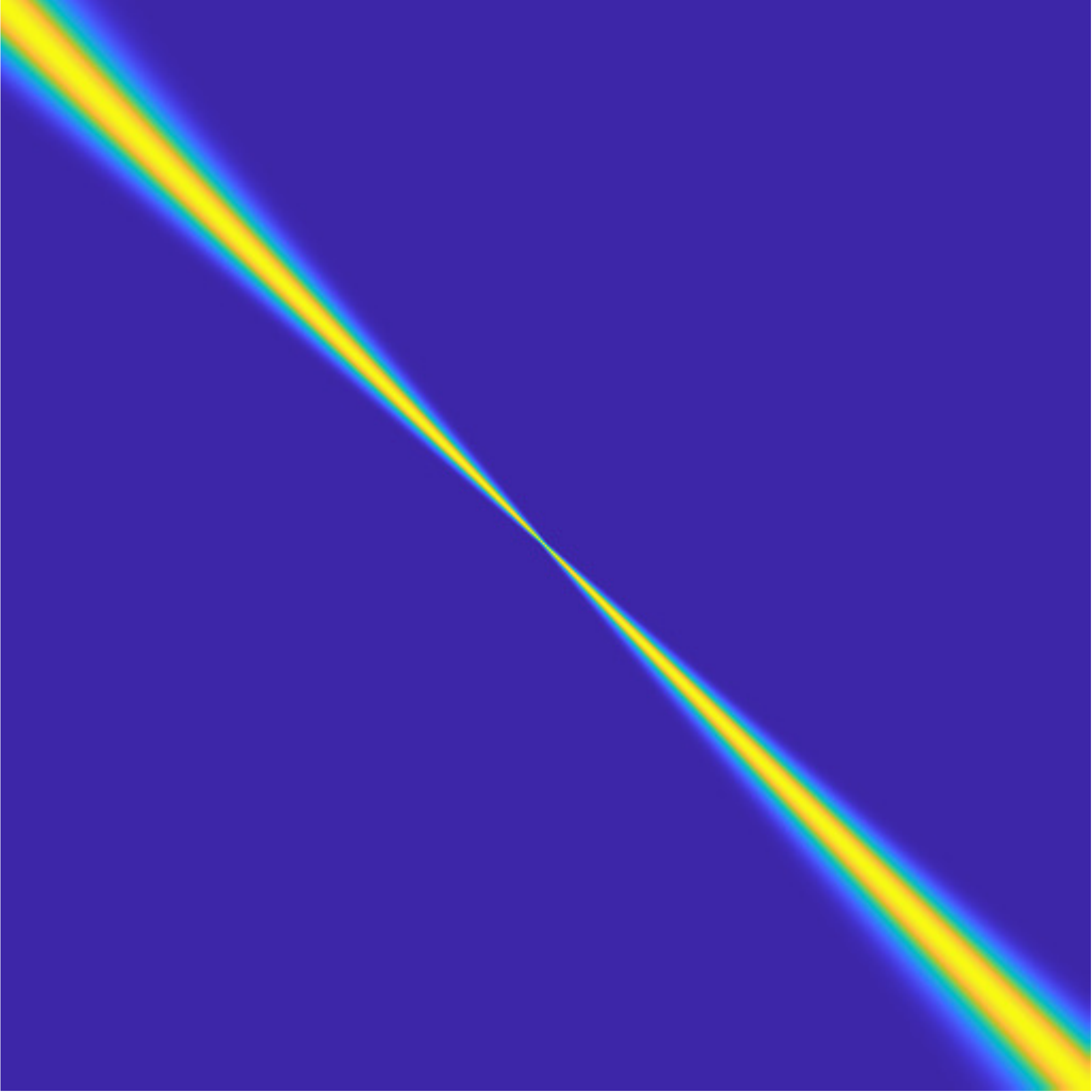}}
	\hspace{2mm}
	\subfloat[$f_c=10$]{\includegraphics[width=0.19\textwidth]{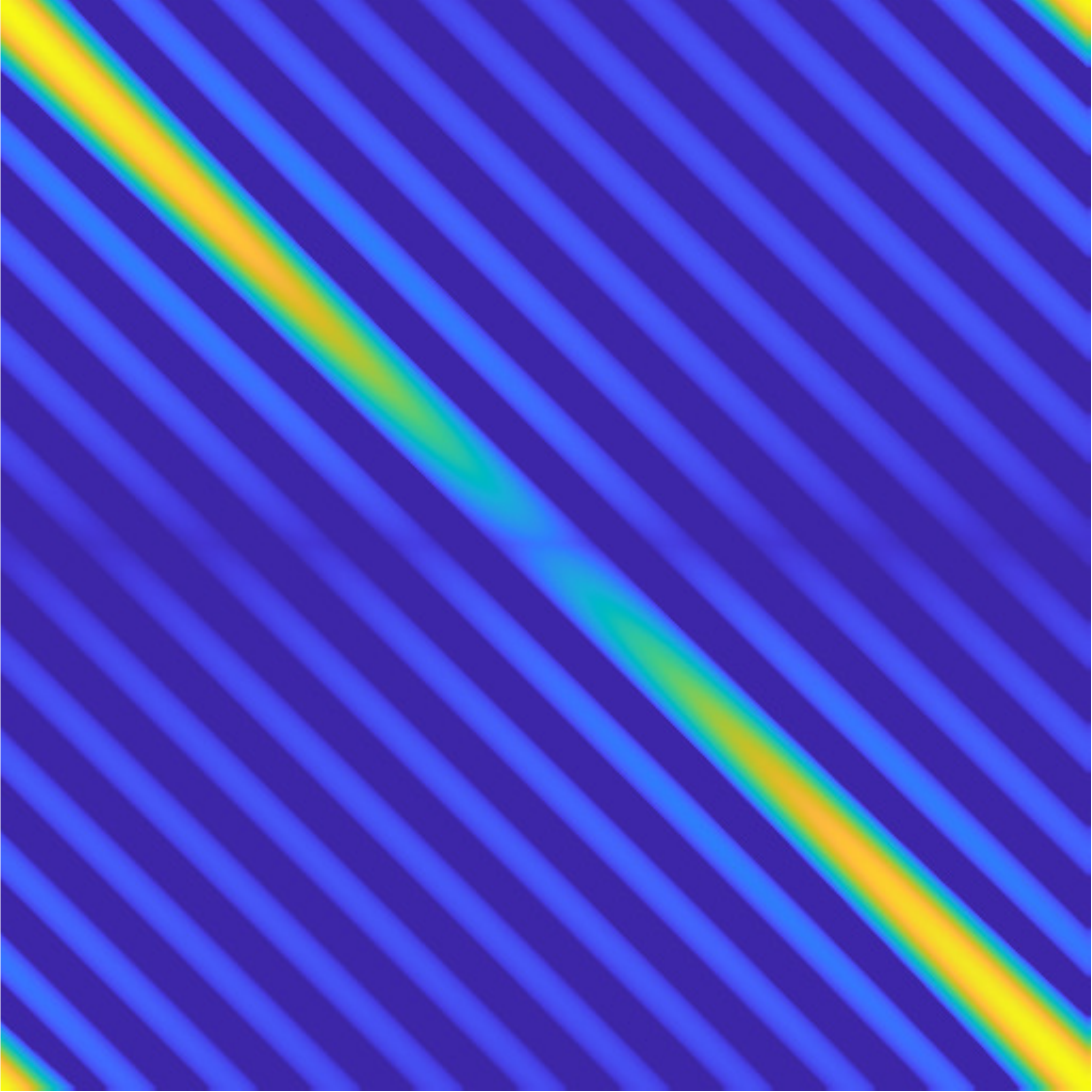}}
	\hspace{0.1mm}
	\subfloat[$f_c=30$]{\includegraphics[width=0.191\textwidth]{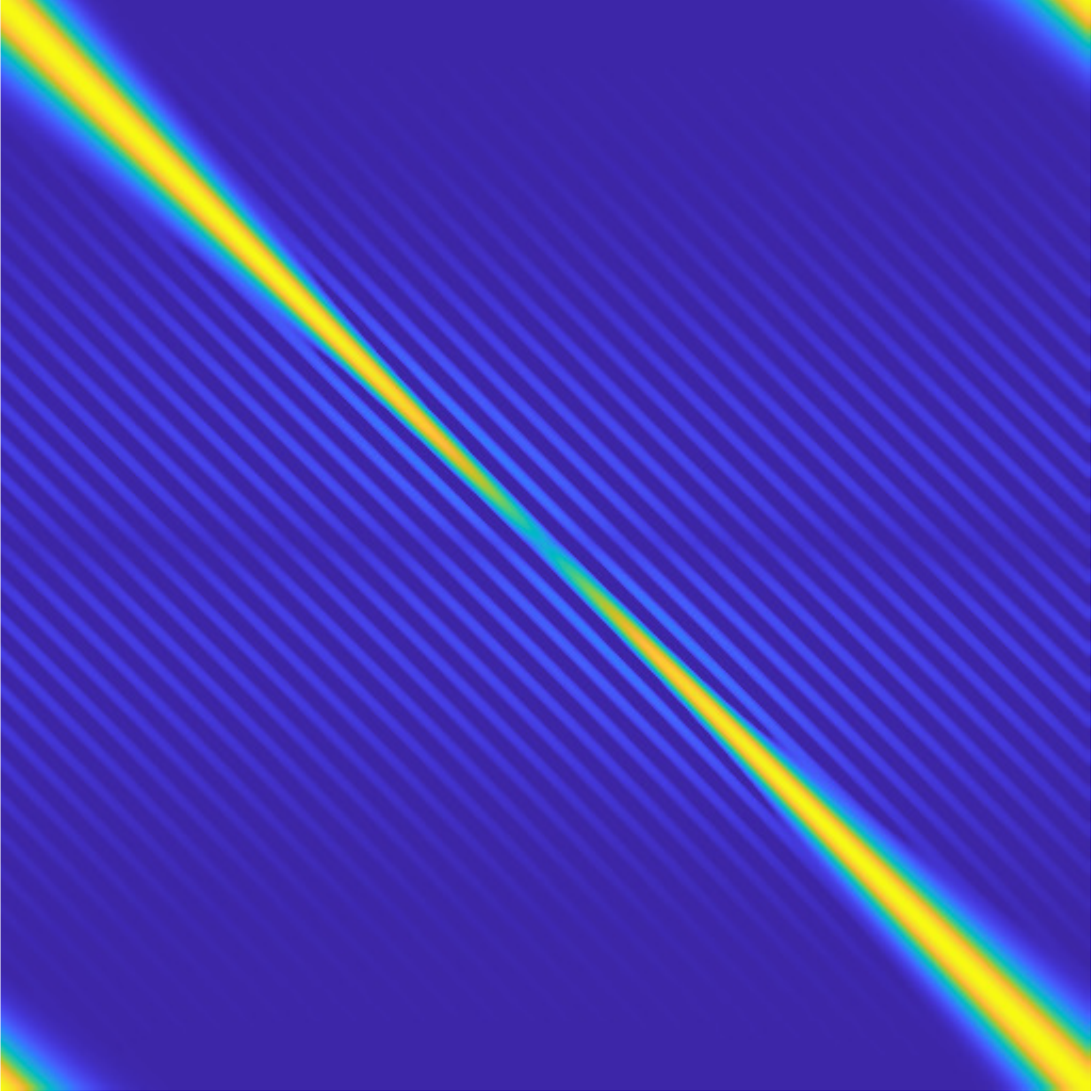}}
	\hspace{0.1mm}
	\subfloat[$f_c=50$]{\includegraphics[width=0.191\textwidth]{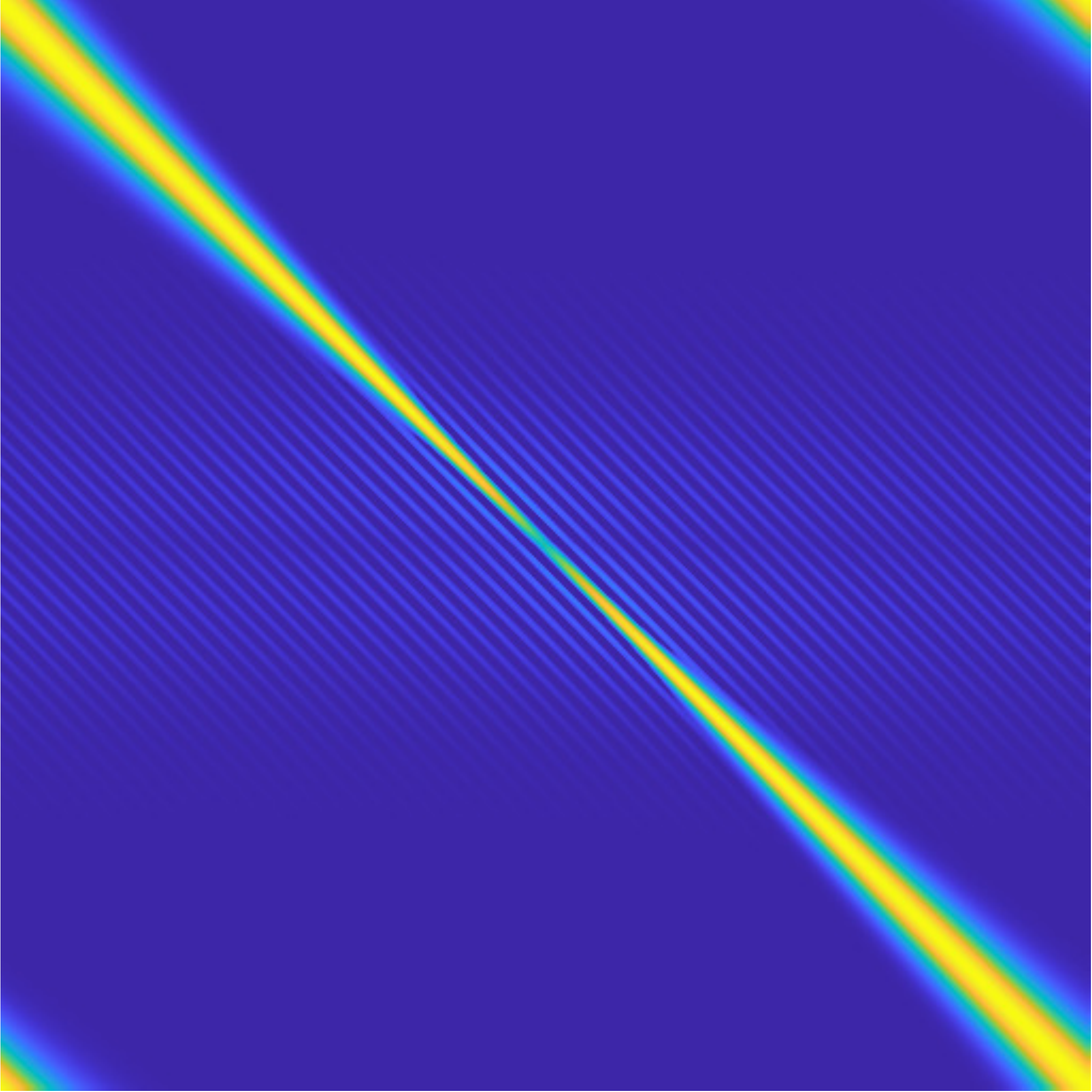}}
	\hspace{0.1mm}
	\subfloat[$f_c=100$]{\includegraphics[width=0.191\textwidth]{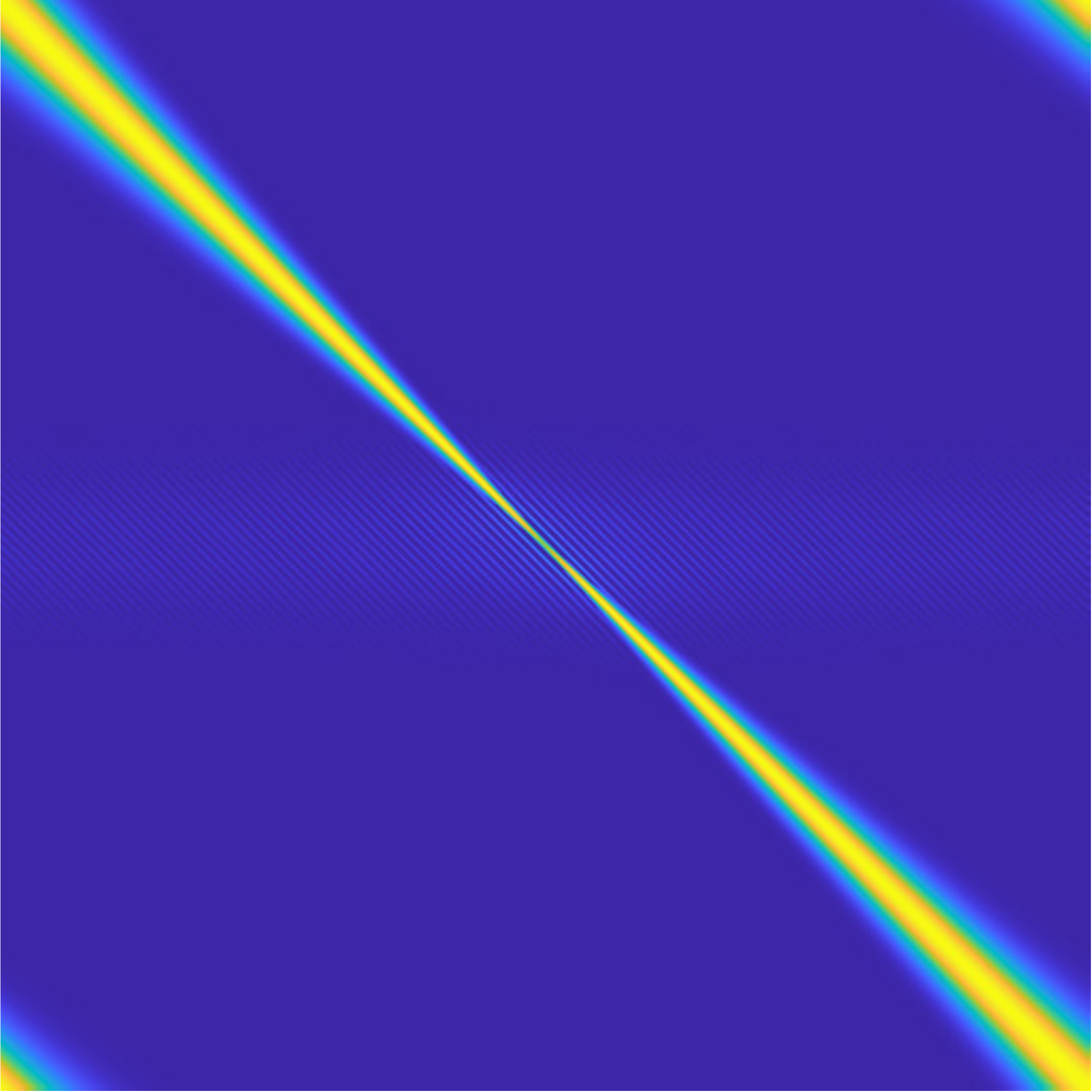}}
	 \caption{\textit{Left}: we display the kernel matrix $(\tilde{\phi}(s,x))_{s,x \in \Gg}$, for the $1$D foveation kernel $\tilde{\phi}$ of the form \eqref{eq:kernel-fov}, with $\Gg$ a regular grid over $\segc{0}{1}$. \textit{Right}: approximated kernel matrices $\A \Ff_c = (\sum_{k\in\Om_c} c_{-k}(\tilde{\phi}_s)e^{-2i\pi kx})_{s,x \in \Gg}$, for different values of $f_c$.}
	\label{fig:foveation-kernel}
\end{figure}


\section{Semidefinite hierarchies}
\label{sec:sdp-hierarchies}

In this section, we generalize the semidefinite programming formulation of the atomic norm used in~\cite{tang13} to the multidimensional case.
As in polynomial optimization~\cite{Lasserre,castro17}, the multidimensional case is much more involved than the one-dimensional one, and needs the introduction of a hierarchy of semidefinite programs.
Loosely speaking, these semidefinite approaches consist in replacing measures with (infinite) moment sequences, assuming they are compactly supported. The so-called \textit{semidefinite hierarchies} then result from truncating these moments, and, in the case where the domain is furthermore \textit{semi-algebraic}, invoking semidefinite characterizations of moment sequences,  see e.g. \cite{curto96}.

First, to make the connection between \eqref{eq:blasso} and the atomic norm minimization problem of~\cite{tang13} explicit, one can see that \eqref{eq:blasso} is actually equivalent to
\eql{
	\umin{z \in \CC^{(2f_c+1)^d}} \frac{1}{2}\normH{y - \A z}^2 + \la \left( \umin{\mu \in \Radon{\Torus^d}} \normtv{\mu}{\Torus^d}\qstq (\Ff\mu)_k = z_k \quad \forall k \in \Om_c \right).
\label{rem:noisy-noiseless}
}
Therefore, given $z \in \CC^{(2f_c+1)^d}$, we focus in this section on the constrained problem

\eq{\label{eq:noiseless}\tag{$\Qq_0(z)$}
	\umin{\mu \in \Radon{\Torus^d}} \normtv{\mu}{\Torus^d} \qstq (\Ff\mu)_k = z_k \quad \forall k \in \Om_c.
}

Note that, by compactness and lower semi-continuity, the above problem has indeed a minimum. Moreover, its value is the so-called \emph{atomic norm} of $z$ introduced in \cite{chandrasekaran12}. The purpose of the present section is to approximate \eqref{eq:noiseless} with problems involving only a finite number of moments of an optimal measure $\mu$ and its absolute value $\abs{\mu}$.

\subsection{Generalized T\oe{}plitz matrices, moment matrices.} 
\label{sec:moment-matrices}
Let $\ell\geq f_c$, and  $\mell\eqdef (2\ell+1)^d$. We assume that some ordering on multi-indices (\ie{} elements of $\Il \eqdef \llbracket -\ell,\ell \rrbracket^d$) has been chosen (for instance the colexicographic order). 

\begin{defn}[Generalized T\oe{}plitz matrix] We say that $R\in \CC^{\mell\times \mell}$ is a \emph{generalized T\oe{}plitz matrix} (also called T\oe{}plitz-block T\oe{}plitz, or mulitlevel T\oe{}plitz), denoted by $R\in \Tomat$, if for every multi-indices $i,j,k\in \llbracket -\ell,\ell \rrbracket^d$  such that $\normi{i+k}\leq \ell$ and $\normi{j+k}\leq \ell$,
\eql{\label{eq:rtoep}%
  R_{i+k,j+k}=R_{i,j}.
}
\end{defn}
If $R$ is the trigonometric moment matrix of some measure $\mu$, it obviously satisfies $R\in \Tomat$, as
\eq{ R_{i+k,j+k}=\int_{\Torus^d}e^{-2\imath\pi \dotp{i+k}{x}}e^{2\imath\pi \dotp{j+k}{x}}\text{d}\mu(x)=\int_{\Torus^d}e^{-2\imath\pi \dotp{i}{x}}e^{2\imath\pi \dotp{j}{x}}\text{d}\mu(x)=R_{i,j}.}

If the ordering on the multi-indices is colexicographical, the generalized Toeplitz property rewrites
\eql{\label{eq:gen-toep-colex}
  R = \sum_{k \in \llbracket -2\ell,2\ell \rrbracket^d} u_k \Theta_k
}%
where $\Theta_k = \theta_{k_d} \otimes \ldots \otimes \theta_{k_1}  \in \CC^{\mell \times \mell}$. Here $\theta_{k_j}$ denotes the $(2\ell+1)\times (2\ell+1)$ T\oe{}plitz matrix with ones on its $k_j$-th diagonal and zeros everywhere else, and $\otimes$ stands for the Kronecker product. For instance, with $2\times 2$ matrices, for $d=2$ and $k = (-1, 0)$, one has
\eq{ 		
	\Theta_k = \begin{bmatrix} 1 & 0 \\ 0 & 1\\ \end{bmatrix} \otimes \begin{bmatrix} 0 & 0\\ 1 & 0 \end{bmatrix} = \begin{bmatrix} 									0 & 0 & 0 & 0\\ 1 & 0 & 0 & 0\\
								0 & 0 & 0 & 0 \\ 0 & 0 & 1 & 0\\
				 \end{bmatrix}.
       }


\subsection{Studied relaxation}
\label{sec:noiseless-relaxation}
       Now, we consider for $\ell\geq f_c$, and  $\mell\eqdef (2\ell+1)^d$,
\eq{\label{eq:relaxnoiseless}\tag{$\Qq_0^{(\ell)}(z)$}
	\begin{aligned}
    \min_{\substack{R\in \Herm_{\mell}^+,\\ \tilde{z} \in \CC^{\mell}, \, \tau \in \RR}}
    &\frac{1}{2}\left(\frac{1}{\mell}\Tr(R) +\tau\right) \qstq
			&\left\{
					\begin{array}{lll}
						(a) & \begin{bmatrix} R & \tilde{z} \\ \tilde{z}^* & \tau \end{bmatrix} \succeq 0 \\
            (b)	&\tilde{z}_k = z_k, \quad \forall k \in \Om_c\\
            (c) & R \in \Tomat
					\end{array}
				\right.
	\end{aligned}.
}%
In the rest of the paper, we write
\eql{\label{eq:block-matrix}
	\Rr \eqdef \begin{bmatrix} R & \tilde{z} \\ \tilde{z}^* & \tau \end{bmatrix}.
}

\begin{rem}[Alternative form]
  In fact, as noted in \cite{polisano17}, it is possible to show by an homogeneity argument that the term $\tau$ must be chosen equal to $\frac{1}{\mell}\Tr(R)$. Indeed the positive semi-definiteness constraint in (a) is equivalent to the following three conditions
 (see for instance~\cite{curto96})
  \begin{enumerate}
    \item   $R\succeq 0$,
    \item there exists some $\alpha\in \CC^{\mell}$ such that $R\alpha = \tilde{z}$,
    \item $\tau \geq \dotp{\alpha}{R\alpha}$.
  \end{enumerate}
  Therefore, at optimality, $\tau=\dotp{\alpha}{R\alpha}$, and replacing $\alpha$, $R$ with $t\alpha$, $1/tR$ for $t>0$ yields another feasible point with energy $\frac{1}{2}\left(\frac{1}{t \mell}\Tr(R) +t\tau\right)$. Minimizing that quantity over $t$ yields the equality of the two terms.
Therefore, \eqref{eq:relaxnoiseless} is equivalent to \eq{\tag{$\tilde{\Qq}_0^{(\ell)}(z)$}
	\begin{aligned}
    \min_{\substack{R\in \Hh_{m}^+,\\ z \in \CC^{\mell}}}\quad
    &\frac{1}{\mell}\Tr(R) \qstq
			&\left\{
					\begin{array}{lll}
          (a') & \begin{bmatrix} R & \tilde{z} \\ \tilde{z}^h & \frac{1}{\mell}\Tr(R) \end{bmatrix} \succeq 0\\
            (b)	&\tilde{z}_k = z_k, \quad \forall k \in \Om_c\\
            (c) & R \in \Tomat
					\end{array}
				\right.
	\end{aligned}.
}%

Incidentally, notice that at optimality the rank of the large matrix in (a) and (a') is equal to $\rank(R)$.  
\end{rem}

The following result explains that \eqref{eq:relaxnoiseless} defines a relaxation of \eqref{eq:noiseless}.
\begin{prop}\label{prop:relaxation}
  Let $z\in \CC^{(2f_c+1)^d}$. For any $\ell\geq f_c$,
  \eql{%
    \min \eqref{eq:relaxnoiseless} \leq  \min (\Qq_0^{(\ell+1)}(z)) \leq  \min \eqref{eq:noiseless}.
  }
  Moreover, $\lim_{\ell\rightarrow +\infty}\min \eqref{eq:relaxnoiseless} = \min \eqref{eq:noiseless}$. 
\end{prop}
  The proof is a straightforward adaptation of the approach used in real polynomial optimization using Lasserre hierarchies~\cite{lasserre01}. We include it for the sake of completeness.
\begin{proof}
  First, we note that if $\tau$, $R'$ and $z'$ are feasible for $(\Qq_0^{(\ell+1)}(z))$, then $\tau$, $R$  and $\tilde{z}$ are feasible  for \eqref{eq:relaxnoiseless}, where $R$  and $z$ respectively denote the restrictions of $R'$ and $z'$ to $\Il$.  Since $R'$ is constant on its diagonals, we get $\frac{1}{(2\ell+3)^d}\Tr(R') =\frac{1}{\mell}\Tr(R)$. That yields the first inequality.

  Now, for the second inequality, let $\mu\in \Mm(\Torus^d)$ such that $\cfour{\mu}{k} = z_k$, and $\xi(x)\eqdef \frac{\text{d}\mu}{\text{d}\abs{\mu}}(x)$ be its sign (defined $\abs{\mu}$ almost everywhere). For any $\ell\geq f_c$, consider $R\in \CC^{\mell\times \mell}$, $z\in \CC^{\mell}$ and $\tau$  defined by
  \begin{align}\label{eq:Rmoment}
    R_{i,j}= \int_{\Torus^d}e^{-2\imath\pi \dotp{i-j}{x}}\text{d}\abs{\mu}(x), \quad   z_j= \int_{\Torus^d}e^{-2\imath\pi \dotp{j}{x}}\text{d}\mu(x)  \qandq \tau=\abs{\mu}(\Torus^d).
  \end{align}
  It is immediate that $(b)$ and $(c)$ are satisfied. Moreover, for all $p\in \CC^{\mell}$, $q\in \CC$,
  \begin{align*}
&  \begin{bmatrix}
    p\\q
  \end{bmatrix}^* \begin{bmatrix} R & z \\ z^h & \tau \end{bmatrix}\begin{bmatrix}
     p\\ q
   \end{bmatrix}=\sum_{i,j} p_i^*R_{i,j}p_j+ 2\sum_{j}\Re(q^* \dotp{\tilde{z}_j}{p_j}) + \tau \abs{q}^2\\
                &\quad= \int_{\Torus^d} \left(\abs{\sum_j p_j e^{2\imath\pi \dotp{j}{x}}}^2 +2\Re(q^*\xi^*(x)\sum_j p_je^{2\imath\pi\dotp{j}{x}}) +\abs{q}^2\right) \diff\abs{\mu}(x)\\
                &\quad= \int_{\Torus^d} \abs{\sum_j p_j e^{2\imath\pi \dotp{j}{x}} +q\xi(x)}^2  \diff\abs{\mu}(x)\geq 0,
  \end{align*}
  which yields $(a)$.
	As a result, $R$, $z$, $\tau$ is admissible for \eqref{eq:relaxnoiseless} with energy 
	$$\frac{1}{2}\left(\frac{1}{\mell}\Tr(R)+\tau\right) =\abs{\mu}(\Torus^d),$$ hence $\inf \eqref{eq:relaxnoiseless} \leq \inf \eqref{eq:noiseless}$.

To prove that the limit of the sequence is indeed $(\min\eqref{eq:noiseless})$, let us consider the dual problem to \eqref{eq:noiseless},
\eql{\label{eq:dual-noiseless}\tag{$\Dd_0(z)$}%
  \sup_{p\in \CC^{(2f_c+1)^d}} \Re \dotp{p}{z} \quad \text{s.t.}\quad \normi{\Ff_c^* p}\leq 1 
}
It is possible to check that \eqref{eq:dual-noiseless} always has a solution~\cite{duval15} and that strong duality holds (see for instance~\cite{candes14}), $\max \eqref{eq:dual-noiseless} = \min \eqref{eq:noiseless}$.

On the other hand, one may show that a dual problem to~\eqref{eq:relaxnoiseless} is given by
\eql{\label{eq:dualrelax}%
  \sup_{\substack{Q\in \Hh_{\mell}^+,\\ p\in \CC^{n_c}}} \Re \dotp{p}{z} \quad 		\text{s.t.} \quad \left\{
					\begin{array}{lll}
						(a) & \begin{bmatrix} Q & \tilde{p} \\ \tilde{p}^* & 1 \end{bmatrix} \succeq 0,\\
            					(b) &\tilde{p}_k = \left\{
								\begin{array}{ll}
									p_k	& \text{if} \quad k \in \Om_c \\
									0 	& \text{if} \quad k \in \Om_\ell \setminus\Om_c
								\end{array}
							\right. \\
            					(c) & Q-\frac{1}{\mell}\mathrm{I}_\mell \in \Tomat^\perp
					\end{array}
				\right.
        \tag{$\Dd_0^{(\ell)}(z)$}
}
where $\Tomat^\perp$ is the orthogonal complement to $\Tomat$, \ie{} 

\eql{%
  Q\in \Tomat^\perp \quad\Longleftrightarrow\quad \forall k\in \llbracket -2\ell,2\ell \rrbracket^d, \quad  \sum_{\substack{i,j\in \llbracket-\ell,\ell\rrbracket^d\\ i+j=k}} Q_{i,j}=0.
}
As before, there exists a solution to \eqref{eq:dualrelax} and $\max \eqref{eq:dualrelax} = \min \eqref{eq:relaxnoiseless}$.

Now, let $\varepsilon>0$, let $p$ be a solution to \eqref{eq:dual-noiseless} and let $p_\varepsilon\eqdef (1-\varepsilon)p$. Since  $\normi{\Ff_c^* p_\varepsilon}< 1$, the bounded real lemma~\cite[Corollary 4.25]{Dumitrescu} ensures that there exists $\ell\geq f_c$, a matrix $Q\in \Hh_{\mell}^+$ with $\mell=(2\ell+1)^d$ such that $Q-\frac{1}{\mell}\mathrm{I}_\mell \in \Tomat^\perp$ and
\eq{%
\begin{bmatrix} Q & \tilde{p}_\varepsilon \\ \tilde{p}_\varepsilon^h & 1 \end{bmatrix} \succeq 0,}
  where $\tilde{p}_\varepsilon$ extends $p_\varepsilon$ in the sense that $\tilde{p}_{\varepsilon,k} =p_{\varepsilon,k}$ for all $k \in \Om_c$, $0$ otherwise.
  As a result, $Q$ and $\tilde{p}_\varepsilon$ are admissible for~\eqref{eq:dualrelax}, hence
  \eq{%
    \min \eqref{eq:relaxnoiseless} = \max \eqref{eq:dualrelax} \geq (1-\varepsilon)\max \eqref{eq:dual-noiseless} = (1-\varepsilon) \min \eqref{eq:noiseless},
  }
  which yields the claimed convergence.
\end{proof}

\subsection{Tightness of the relaxation and low rank property}
\label{sec:noiseless-collapsing}
\label{sec:collapse-lowrank}
The next proposition discusses the equality case between \eqref{eq:relaxnoiseless} and \eqref{eq:noiseless}, referred to as \textit{collapsing} of the hierarchy, by interpreting $R$ as a moment matrix.
\begin{prop}\label{prop:tight}
Let $\ell\geq f_c$. Then, $\min \eqref{eq:relaxnoiseless} = \min \eqref{eq:noiseless}$ if and only if there exists $(R,z,\tau)$ solution to \eqref{eq:relaxnoiseless}  and $\mu$ solution to \eqref{eq:noiseless} such that
\eql{\label{eq:momentR}%
	\tau=\abs{\mu}(\Torus^d) \qandq R_{i,j}= \int_{\Torus^d}e^{-2\imath\pi\dotp{i-j}{x}}\d\abs{\mu}(x)
}
for all $i,j\in \Il$. In particular, if $\mu$ is a discrete measure with cardinal $r$, then $\rank R\leq r$.
\end{prop}

\begin{proof}
  Assume that $\min \eqref{eq:relaxnoiseless} = \min \eqref{eq:noiseless}$, and let $\mu$ be a solution to \eqref{eq:noiseless}. Define $R$, $z$ and $\tau$ by~\eqref{eq:Rmoment}. As in the proof of Proposition~\ref{prop:relaxation}, we see that 
  $(R,z,\tau)$ is admissible for \eqref{eq:relaxnoiseless}, with energy
  \eq{%
\frac{1}{2}\left(\frac{1}{\mell}\Tr(R)+\tau\right) =\abs{\mu}(\Torus^d)= \min\eqref{eq:relaxnoiseless}.
}
Hence $(R,z,\tau)$ is a solution to \eqref{eq:relaxnoiseless}.

The converse implication is straightforward: if $(R,z,\tau)$ (resp.  $\mu$) is a solution to  \eqref{eq:relaxnoiseless} (resp. \eqref{eq:noiseless}) such that \eqref{eq:momentR} holds, then $\frac{1}{2}\left(\frac{1}{\mell}\Tr(R)+\tau\right) =\abs{\mu}(\Torus^d)$ and we obtain $\min \eqref{eq:relaxnoiseless} = \min \eqref{eq:noiseless}$.
  
If $R$ satisfies~\eqref{eq:momentR} and $\mu$ has cardinal $r$, \ie{} $\mu=\sum_{i=1}^r a_i \delta_{x_i}$ with $x_i\neq x_j$ for $i\neq j$, we note that this matrix $R$ is of the form
\eq{%
  R= \sum_{i=1}^r a_i v_\ell(x_i)v_\ell(x_i)^*, \qwhereq v_\ell(x)\eqdef (e^{-2\imath\pi \dotp{k}{x}})_{k \in \Il}
}
Thus $R$ is a sum of at most $r$ rank one matrices, and $\rank R\leq r$.
\end{proof}

\begin{rem}[Low rank solutions]
It is important to note that, as proved in~\cite[Prop. 2.1]{tang13}, the equality $\min\,(\Qq_0^{(f_c)}(z))=\min\eqref{eq:noiseless}$ always holds for $d=1$. Therefore, if one seeks to recover a sparse measure $\mu=\sum_{i=1}^r a_i \delta_{x_i}$ solution to~\eqref{eq:noiseless}, then some solution to $\min\,(\Qq_0^{(f_c)}(z))$ has rank (at most) $r$. In dimension $d = 2$, it is known that the Lasserre hierarchy collapses at some order $\ell$ that may be arbitrarily large \cite[Section 4]{castro17}. In our experiments however, we have always observed $\min\,(\Qq_0^{(f_c)}(z))=\min\eqref{eq:noiseless}$. This is why our approach, exposed in Section~\ref{sec:ffw}, focuses on capturing a low-rank solution to $(\Qq_0^{(f_c)}(z))$.
\end{rem}

To conclude this section, we give a practical criterion to detect collapsing. As observed by Curto and Fialkow \cite{curto96}, the flatness property (see definition below) is essential when trying to determine whether some matrix is the moment matrix of a measure.

Given a matrix $R \in \Herm_\mell$, we write the block decomposition
\eql{
	R = \begin{blockarray}{ccc} & \Ilun & \Il \setminus \Ilun\\
		\begin{block}{c(cc)}
			\Ilun 		        & A     & B \\
			\Il \setminus \Ilun & B^h & C \\
		\end{block} 
	\end{blockarray}
	\label{eq:flatmat}
}

\begin{defn}[Flat matrix] Let $R \in \Herm_\mell$. In the decomposition \eqref{eq:flatmat}, we say that $R$ is flat (or that $R$ is a flat extension of $A$) if $\rank R = \rank A$.
\label{def:flat-matrix}
\end{defn}
The following result adapts Theorem~7.7 in \cite{curto96} to our setting. As it is not exactly the one used in \cite{curto96} (especially in terms of multi-dimensional degree), we provide a proof in Appendix~\ref{app:flatness}.

\begin{thm}
\label{thm:flatness}
Let $\ell \geq 2$, and $R \in \Herm^+_\mell \cap \Tomat$ be a positive semi-definite generalized T\oe{}plitz matrix. If $R$ is flat, then there exists a positive ($\rank R$)-sparse Borel measure $\nu$ such that $R$ is the moment matrix of $\nu$. 
\end{thm}

\subsection{Support reconstruction}
\label{sec:support-extraction}

%
%
%

We now discuss the problem of recovering the support $\{x_1, \ldots, x_r\}$ of the underlying positive measure $\nu = \sum_{j=1}^r a_j \de_{x_j}$ from its moment matrix $R_\ell(\nu)$.  A first approach, as in \cite{josz17}, is to use the method proposed in \cite{harmouch18}. In this section, we present a slightly different procedure, leveraging the low-rank factorization $R_\ell(\nu) = UU^*$ of the moment matrix that our algorithm provides, see Section~\ref{sec:ffw}. It follows the method exposed in \cite[Section 4.3]{Lasserre}, and is summarized in Algorithm~\ref{algo:extraction}.

By construction of moment matrices, one has $R_\ell(\nu) = V(\xm)DV(\xm)^*$, where $V(\xm) = \begin{bmatrix} e^{-2i\pi\dotp{k}{x_1}}, \ldots, e^{-2i\pi\dotp{k}{x_r}}\end{bmatrix}_{k\in\Om_\ell}$, and $D = \diag (a_1, \ldots, a_r)$. In the following, we write $v(x_1), \ldots, v(x_r)$ the columns of $V(\xm)$. Let $\tilde{U}$ be the reduced column echelon form of $U$, \ie $\tilde{U}$ has the form
\eq{
	\tilde{U} = \begin{bmatrix}
				1 	 	&		&		&		& 		\\
				\star 	&		&		&		& 		\\
				0 	 	& 1		&		&		& 		\\
				0 	 	& 0		& 1		&		& 		\\
				\star 	& \star 	& \star	&		& 		\\
				\vdots 	& 		& 		& \ddots 	& 		\\
				0 		& 0 		& 0 		& \ldots 	& 1		\\
				\star 	& \star 	& \star 	& \ldots 	& \star 	\\
				\vdots 	& 		& 		& 		& \vdots 	\\
				\star 	& \star 	& \star 	& \ldots 	& \star 	\\
			   \end{bmatrix}.
}
It is obtained by Gaussian elimination with column pivoting. The matrices $V(\xm)$ and $\tilde{U}$ span the same linear space. In particular, for any $1 \leq j \leq r$, $v(x_j) \in \Im \tilde{U}$. Therefore, there exists $w(x_j) \in \CC^r$ such that
\eql{\label{eq:poly-eq}
	v(x_j) = \tilde{U} w(x_j),
}
and one can verify that
\eq{
	w(x_j) = \begin{bmatrix} e^{-2i\pi\dotp{\ga_1}{x_j}} \\ \vdots \\ e^{-2i\pi\dotp{\ga_r}{x_j}} \end{bmatrix}, 
}
where $\ga_1, \ldots, \ga_r$ are the indices of the lines containing the pivots elements in $\tilde{U}$.

The recovery procedure thus reduces to solving the system of (trigonometric) polynomial equations
\eql{\label{eq:poly-sys}
	v(x) = \tilde{U} w(x),
}
for $x \in \Torus^d$. A popular method to solve such problems is the \emph{Stetter-M\"oller method}~\cite{moller95}, also known as the \emph{eigenvalue method}, which relates the solutions of \eqref{eq:poly-sys} to the eigenvalues of so-called \emph{multiplication matrices}, that can be build directly from the matrix $\tilde{U}$.
The multiplication matrices in the basis $\Bb \eqdef \{e^{-2i\pi\dotp{\ga_1}{\cdot}}, \ldots, e^{-2i\pi\dotp{\ga_r}{\cdot}}\}$ are the matrices $N_n$, $1 \leq n \leq d$, that satisfy
\eql{\label{eq:m-matrix}
	N_n w(x_j) = e^{-2i\pi\dotp{e_n}{x_j}} w(x_j), \quad \forall 1 \leq j \leq r.
}

In practice, the matrices $N_n$ can be read straightforwardly using \eqref{eq:poly-eq} by selecting the rows indexed by the monomials $e^{-2i\pi\dotp{\ga_j + e_n}{x}}$ in $\tilde{U}$, for $j=1, \ldots, r$. Then, using \eqref{eq:m-matrix}, one can retrieve the coordinates of each point of the support by computing the eigenvalues of each matrix $N_n$, $n=1, \ldots, d$. However, it is not clear how to put these coordinates together for recovering the positions $x_1, \ldots, x_r \in \Torus^d$. To this end, it is better to consider the eigenvectors $w(x_j)$, $j=1, \ldots, r$. Indeed, these eigenvectors are common to each matrices $N_n$. Therefore, if one consider the matrix $N = \sum_{n=1}^d \la_n N_n$, where $\la_n$ are random real numbers such that $\sum_n \la_n = 1$, then all the eigenspaces of $N$ are $1$-dimensional with probability $1$, and spanned by the vectors $w(x_j)$, $j=1, \ldots, r$. Thus, writing the Schur decomposition $N = Q T Q^*$, where $Q = \begin{bmatrix}q_1, \ldots, q_r \end{bmatrix}$ is orthogonal and $T$ is upper triangular, yields
\eq{
	e^{-2i\pi \dotp{e_n}{x_j}} = q_j^* N_n q_j,
}
from where we deduce the positions $x_j$, for $1\leq j \leq r$.


\begin{algorithm}
\caption{Support recovery}
\label{algo:extraction}
\begin{algorithmic}
\STATE{\textbf{input}: $U \in \Mm_{m,r}(\CC)$ of \emph{rank $r$} s.t. $R_\ell(\nu) = UU^*$}
\STATE{1. \emph{Compute multiplication matrices}:}
\STATE{\qquad \begin{tabular}{ll}
				(1.a) & compute the reduced column echelon form $\tilde{U}$ of $U$, and store the \\
					& (multi-)indices $\ga_1, \ldots, \ga_r \in \Om_\ell$ of pivot elements\\
				(1.b) & \textbf{for} $n=1, \ldots, d$\\
					& \quad - $I_n \eqdef \enscond{\ga_j + e_n}{j = 1, \ldots, r}$, with $e_n = (0, \ldots, 1, 0, \ldots, 0) \in \RR^d$\\
					& \quad - compute $N_n \eqdef \left(\tilde{U}_{i,j}\right)_{i \in I_n, j \in \segi{1}{r}}$\\
					& \textbf{endfor}
				\end{tabular}}
\STATE{2. Compute a random combination $N = \sum_{n=1}^d \la_n N_n$, where $\la_n \in \RR$ are random and satisfy $\sum_n \la_n = 1$}
\STATE{3. Compute the Schur decomposition $N = QTQ'$, with $Q \eqdef [q_1, \ldots, q_r]$}
\STATE{4. Compute $z_{j,n} \eqdef q_j' N_n q_j$, \quad $j=1, \ldots, r$, \quad $n=1, \ldots, d$}
\STATE{\textbf{return} $x_{j,n} = -\frac{1}{2\pi}\arg{z_{j,n}} \mod 1$, \quad $j=1, \ldots, r$, \quad $n=1, \ldots, d$}
\end{algorithmic}
\end{algorithm} 

\subsection{Semidefinite relaxation for the BLASSO}
\label{sec:sdp-blasso}
In view of Remark 2, concatenating the minimization in $z$ and in $(R,\tau)$, we are led to solve the following problem
\eql{\label{eq:sdp-blasso}\tag{$\Pp_\la^{(\ell)}(y)$}
		\umin{\substack{R \in \Herm_\mell^+, \\ \tilde{z} \in \CC^\mell,\\ \tau \in \RR}}    
	\frac{1}{2}\left(\frac{\Tr(R)}{\mell} + \tau\right) + \frac{1}{2\la}\normH{y-\Aa z}^2 \;\; \text{s.t.} \; \left\{
					\begin{array}{lll}
						(a) & \begin{bmatrix} R & \tilde{z} \\ \tilde{z}^* & \tau \end{bmatrix} \succeq 0\\
						(b) & z_k = \tilde{z}_k \quad \forall k \in \Om_c\\
						(c) & R \in \Tomat
					\end{array}
				\right.,
}
where $\mell = (2\ell+1)^d$. The problem \eqref{eq:sdp-blasso} is the semidefinite relaxation of \eqref{eq:blasso} at order $\ell$.

As it appears, the size of the above semidefinite program is $\mell^2 \geq (2f_c+1)^{2d}$. Therefore, usual \textcolor{black}{interior points methods} are limited when $d>1$, or even when $d=1$ for large values of $f_c$. In the rest of this paper, we introduce a method which scales well with the dimension $d$.

Thanks to the low-rank property highlighted in Section~\ref{sec:collapse-lowrank}, the search space of \eqref{eq:sdp-blasso} may be restricted to rank-deficient matrices. Such geometry is well exploited by conditional gradient algorithms. However, these methods are not able to handle the SDP contraint $(a)$ together with the linear constraint $(c)$. Furthermore, the intersection between manifolds of fixed rank matrices and the linear space defined by $(c)$ quickly becomes unanalyzable for matrices larger than $2\times 2$, and non-convex optimization schemes on this search space would likely be difficult to implement. Instead, we propose to smooth the geometry of the problem.

\section{T{\oe}plitz penalization}
\label{toeplitz_relaxation}

%
To overcome the difficulty induced by constraint $(c)$ in \eqref{eq:sdp-blasso}, we introduce a penalized version of \eqref{eq:sdp-blasso} -- which can also be seen as a perturbation of the atomic norm regularizer used in \cite{tang13, tang-recht15}. 

Let $P_{\Tomat}$ be the projector on the set $\Tomat$; when working with the colexicographical order, as we do in our numerical simulations, $P_{\Tomat}$ takes the form
\eq{
	\projtoep: R \mapsto \sum_{\normi{k} \leq 2\ell} \frac{\dotp{R}{\Theta_k}}{\norm{\Theta_k}^2} \Theta_k
}
where the matrix $\Theta_k$ are introduced in Section~\ref{sec:moment-matrices}. In dimension one, the operator $\projtoep$ replaces each entry of $R$ by the mean of the corresponding diagonal. We consider the following program:

\eq{		
	\begin{aligned}
		\umin{\tau,z,R} &\frac{1}{2}\left(\frac{1}{\mell}\Tr(R) + \tau \right) + \frac{1}{2\la}\normH{y - \Aa z}^2 + \frac{1}{2\rho}\norm{R - \projtoep (R)}^2\\
		\text{s.t.} &\left\{
					\begin{array}{ll}
						\begin{bmatrix}R & \tilde{z} \\ \tilde{z}^h & \tau \end{bmatrix} \succeq 0\\
						\tilde{z}_k = z_k, \quad \forall k \in \Om_c
					\end{array}
				\right.
	\end{aligned}
	\label{eq:sdp-blasso-pen}
	\tag{$\Pp_{\la,\rho}^{(\ell)}(y)$}
}%
where the parameter $\rho$ controls the penalization of the T{\oe}plitz constraint. We write $f_{\la,\rho}$ the objective of this problem, and $\Rr_{\la,\rho}$ a solution. We study numerically the validity of this approach. For simplicity, the numerical experiments of this section are all in dimension one.

\subsection{Sensitivity analysis}


Although $\Rr_{\la,\rho}$ obviously differs from the true solution $\Rr_\la$, it is possible to show, following an approach similar to \cite{vaiter15}, that under some mild non-degeneracy hypothesis on $\eta_{\la,\rho}$, and for small enough values of $\rho$, $\Rr_{\la,\rho}$ is sufficiently close to $\Rr_\la$ to allow accurate support reconstruction. In particular, both matrices have the same rank. Numerical observations confirm that this regime exists: Figure~\ref{rank-stability} shows the evolution of $\rank\Rr_{\la, \rho}$ with respect to $\rho$. We see that the rank of $\Rr_{\la,\rho}$ remains stable for low values of $\rho$, and equal to the sparsity of the initial measure. 

\begin{figure}
	\includegraphics[width=0.5\textwidth]{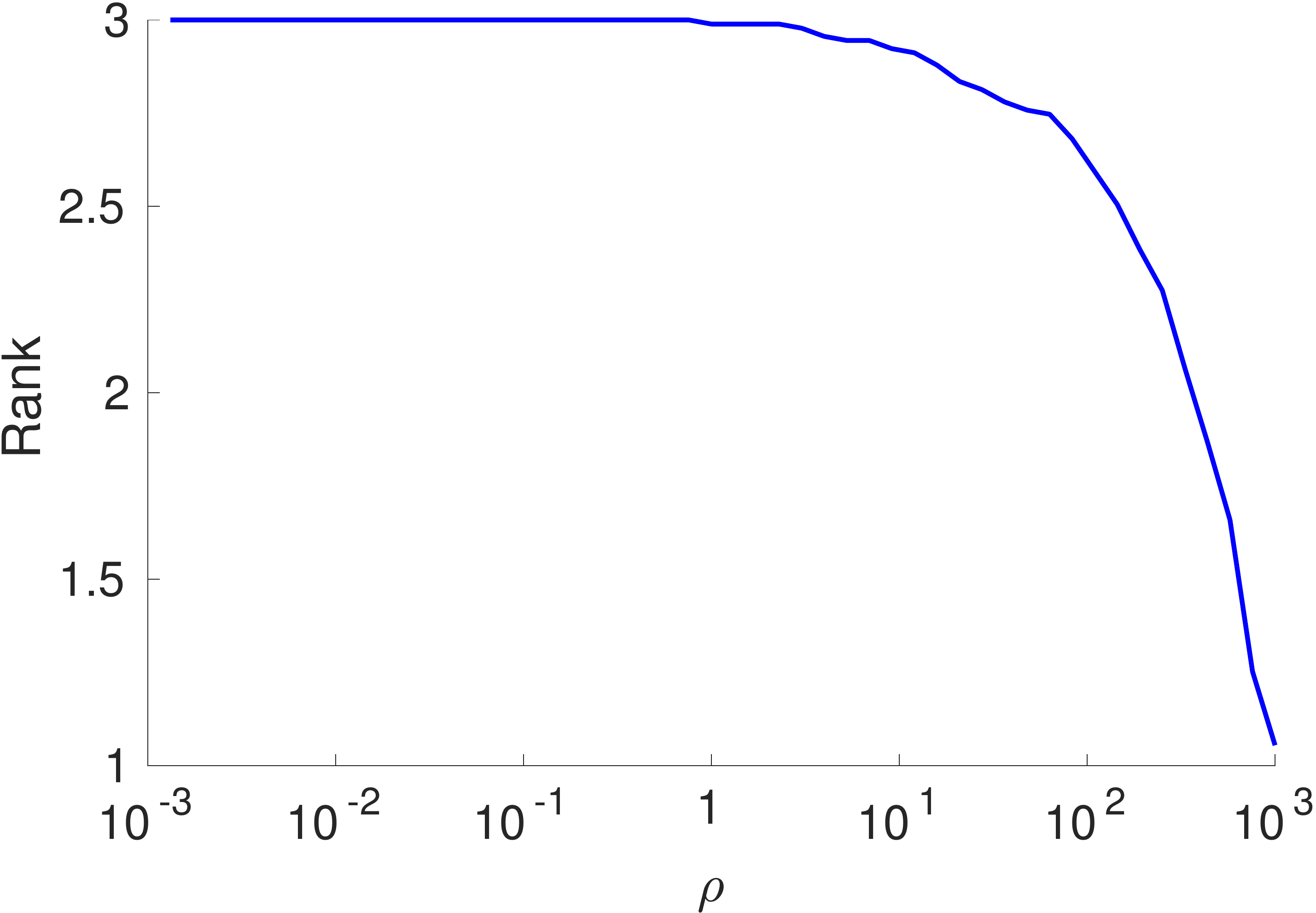}
	\begin{minipage}[b]{0.48\textwidth}
	\caption{Rank drop (for Dirichlet measurements, with $f_c=17$), with respect to $\rho$. Results are averaged over $100$ random trials of positive $3$-sparse initial measures; the minimal separation distance, \ie the minimal distance between two consecutive spikes, is larger than $1/(10f_c)$ in all the cases.}
	\label{rank-stability}
	\end{minipage}
	\postimagespace
\end{figure} 

The following proposition makes this statement more precise.

\begin{prop} Let $(\Rr_{\la,\rho_n})$ be a sequence of solution of $(\Pp^{(\ell)}_{\la,\rho_n}(y))$, with $\rho_n \rightarrow 0$ as $n \rightarrow \infty$. Then any accumulation point of $\left(\Rr_{\la,\rho_n} \right)$ is a solution of \eqref{eq:sdp-blasso}.
\label{prop:sensitivity}
\end{prop}
\begin{proof} One has
\eq{
	\frac{1}{\mell}\tr(\Rr_{\la,\rho_n}) \leq f_{\la,\rho_n}(\Rr_{\la,\rho_n}) \leq f_{\la,\rho_n}(0) = \frac{1}{2\la}\normH{y}^2
}
hence there exists a subsequence $(\Rr_{\la,\rho_s})$ that converges. Let $\Rr_\la^\star$ be its limit, and let $\Rr_\la^0$ be a solution of \eqref{eq:sdp-blasso}.
Since $\frac{1}{2\rho}\norm{\Rr_{\rho_s} - \projtoep(\Rr_{\la,\rho_s})}^2 \leq \frac{1}{2\la}\normH{y}^2$, one has $\norm{\Rr_{\la,\rho_s} - \projtoep(\Rr_{\la,\rho_s})} \rightarrow 0$ when $s \rightarrow \infty$, which ensures that $\Rr_\la^\star \in \Tomat$.  Furthermore, we have

\eq{
	f_\la(\Rr_{\la,\rho_s}) \leq f_{\la,\rho_s}(\Rr_{\la,\rho_s}) \leq f_{\la,\rho_s}(\Rr_\la^0) = f_\la(\Rr_\la^0).
}
%
%
Passing to the limit in these inequalities thus gives $f_\la(\Rr_\la^\star) \leq f_\la(\Rr_\la^0)$. Since $\Rr_\la^\star$ is semi-definite positive (as the SDP cone is closed) and belongs to $\Tomat$, it is a solution of \eqref{eq:sdp-blasso}.
\end{proof}

\subsection{Support recovery}
\label{subsec::approximate-rootfinding}
We observe numerical evidences of the robustness of the extraction procedure described in Section~\ref{sec:support-extraction} in this penalized setting: although the solutions of \eqref{eq:sdp-blasso-pen} do not exactly satisfy the generalized T\oe{}plitz property, and therefore are not moment matrices, Algorithm~\ref{algo:extraction} still yields a good estimation of the support of $\mu_\la$. In particular, in a satisfyingly large regime of values of parameter $\rho$, the eigenvalues of the multiplication matrices (see Section~\ref{sec:support-extraction}) remain very stable: in Fig.~\ref{fig:ev-trajectories}, we consider a $1$D setting, and we plot $\arg z_j(\rho)$, where $z_j(\rho)$, $1\leq j \leq r$ are the eigenvalues of the multiplication matrix of Algorithm~\ref{algo:extraction}, extracted from a solution of \eqref{eq:sdp-blasso-pen}. The width of the line is defined as $\abs{\log( \abs{1 - \abs{z_j}} )}$, so that the thicker the line is, the closer $z_j(\rho)$ is to the unit circle. We see that the extraction procedure is very stable, up to a certain point, which coincides with the first rank drop of $\Rr_{\la,\rho}$.

\begin{figure}
	\includegraphics[width=0.5\textwidth]{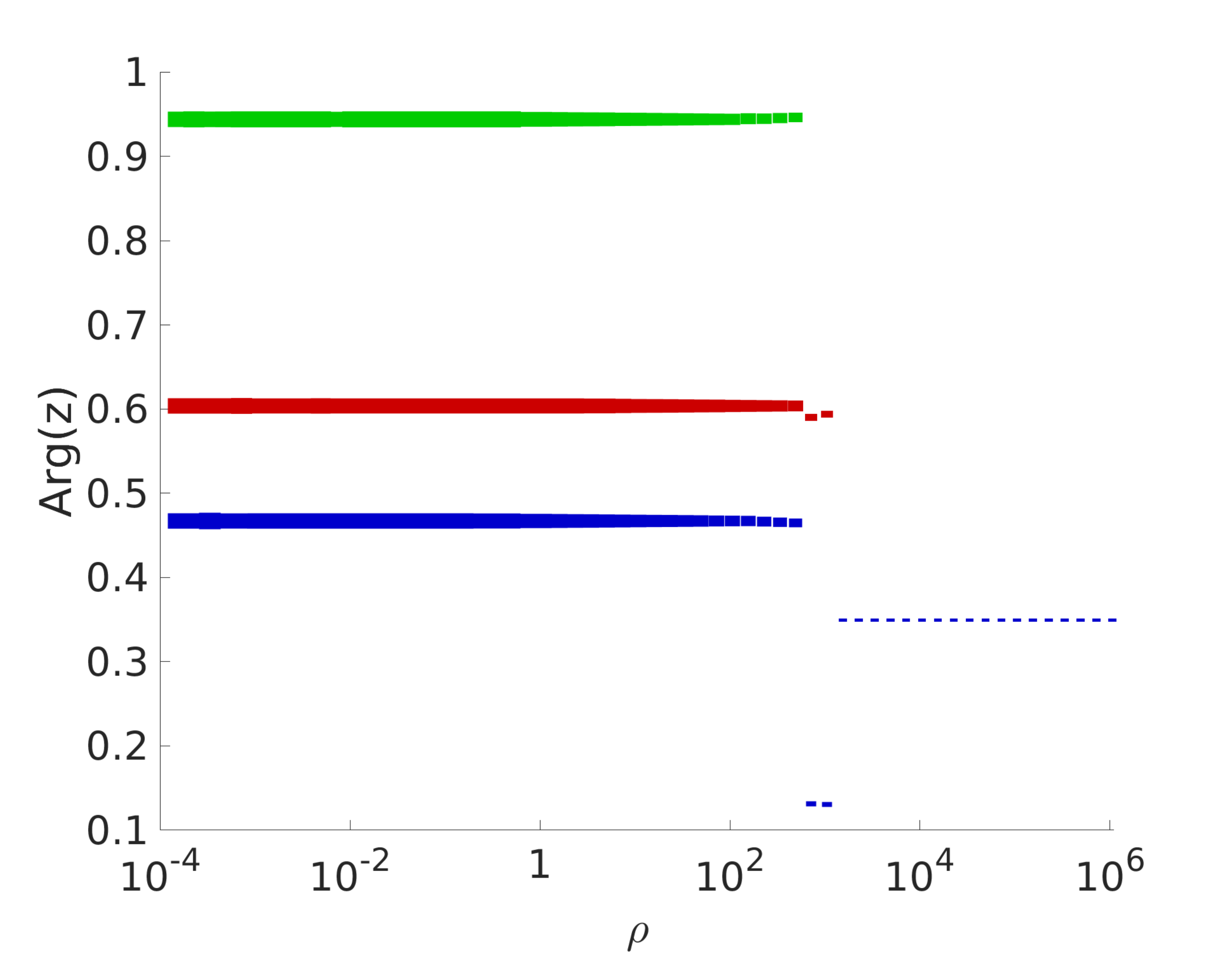}
	\begin{minipage}[b]{0.48\textwidth}
	\caption{Trajectories of the support $x_j = \arg(z_j)$, $1\leq j \leq 3$ (see Algorithm~\ref{algo:extraction}), with respect to $\rho$.}
	\label{fig:ev-trajectories}
	\end{minipage}
	\postimagespace
\end{figure} 
\section{FFT-based Frank-Wolfe (FFW)}
\label{sec:ffw}

In this section, we propose an efficient numerical scheme for solving \eqref{eq:sdp-blasso}, that takes advantage of the low-rank property of the solutions as well as of the convolutive structure of the T\oe{}plitz constraint. Given a matrix $\Rr$ of the form \eqref{eq:block-matrix}, we consider in the rest of the paper the normalized objective function
\eql{\label{eq:objective}	
	f(\Rr) \eqdef C_0 \left( \frac{1}{2}\left(\frac{\Tr(R)}{\mell} + \tau\right) + \frac{1}{2\la}\normH{y - \Aa z}^2 + \frac{1}{2\rho}\norm{R - \projtoep(R)}^2 \right),
}
where $C_0 = 2\la / \normH{y}^2$.

\subsection{Frank-Wolfe}
\label{subsection:fw}

The Frank-Wolfe algorithm \cite{frankwolfe56}, \textit{a.k.a.} conditional gradient, aims at minimizing a convex and continuously differentiable function $f$ over a compact convex subset $K$ of a vector space. The essence of the method is as follows: linearize $f$ at the current position $\Rr_t$, solve the auxiliary linear problem of minimizing $\Ss \mapsto \dotp{\nabla f(\Rr_t)}{\Ss}$ on $K$, and move towards the minimizer to obtain the next position. This scheme ensures the sparsity of its iterates, since the solution after $k$ iterations is a convex combination of at most $k$ atoms. We refer to \cite{jaggi13} for a detailed overview of the method. Since no Hilbertian structure is required, it is a good candidate to solve problems in Banach space \cite{bredies13}. Moreover, in many cases, the linear minimization oracle may be computed efficiently, as it amounts to extracting an extremal point of the set $K$. 

\paragraph{Over the semidefinite cone} In our case, $K$ is the positive semidefinite cone, which is linearly spanned by unit-rank matrices. It is not bounded (hence not compact), but one can restrict \eqref{eq:sdp-blasso-pen} over a bounded subset of the cone by noticing that for any solution $\Rr^\star$ of \eqref{eq:sdp-blasso-pen}, one has
\eq{
	\frac{1}{2}\left(\tau^\star + \frac{\Tr(R^\star)}{\mell} \right) \leq f(\Rr^\star) \leq f(0),
}
suggesting a subset of the form $\enscond{\Rr \succeq 0}{ \dotp{\Rr}{J_\mell} \leq D_0 \eqdef 2f(0)}$, where
\eq{
	J_\mell = \begin{bmatrix} \frac{1}{\mell}I_\mell & 0 \\ 0 & 1\end{bmatrix},
}
so that $\dotp{\Rr}{J_\mell} = \tau + \frac{1}{\mell}\Tr(R)$. 

The linear minimization then consists in computing a minor eigenvector of $\nabla f(\Rr)$.

\begin{lem}\label{lem:lmo} Let $M \in \Mm_\mell(\CC)$ be a Hermitian matrix, and let $\{\la_1, \ldots, \la_\mell \}$ be its eigenvalues, with $\la_1 \leq \ldots \leq \la_\mell$. Then, for $D_0 \geq 0$,
\eq{
	\uargmin{\substack{S\succeq0 \\ \dotp{S}{J_\mell} \leq D_0}} \dotp{M}{S} = 
		\left\{ \begin{aligned} 
			&D_0J_\mell^{-\frac{1}{2}}e_1e_1^h J_\mell^{-\frac{1}{2}}\qifq \la_1 < 0 \\ 
			&0 \quad \text{otherwise} 
		\end{aligned} \right.
}
where $e_1 \in \Ker(J_\mell^{-\frac{1}{2}}MJ_\mell^{-\frac{1}{2}} - \la_1 I)$ such that $\norm{e_1} = 1$.
\end{lem}

\begin{proof} By the change of variable $S' = J_\mell^{\frac{1}{2}} S J_\mell^{\frac{1}{2}}$, the linear program reformulates
\eq{
	\argmin \dotp{M'}{S'} \qstq \left\{\begin{aligned} &S' \succeq 0 \\ &\Tr(S') \leq D_0 \end{aligned} \right.
}
with $M' = J_\mell^{-\frac{1}{2}} M J_\mell^{-\frac{1}{2}}$.
Let $S' \in \enscond{X \succeq 0}{\Tr(X) \leq D_0}$, and write $S' = \sum \al_i v_iv_i^h$, with $\al_i \geq 0$ and $\norm{v_i} = 1$. Then
\eq{
		\dotp{M'}{S'} = \sum \al_i v_i^h M' v_i \geq \sum \al_i e_1^hM'e_1 = \Tr(S') \la_1
}
with equality if and only if $S' = (\sum \al_i)e_1e_1^h$, hence the desired result.
\end{proof}

\subsection{The FFW algorithm}
\label{sec:ffw-algo}

To solve $\eqref{eq:sdp-blasso}$, we apply a plain Frank-Wolfe scheme, to which we also add a non-convex corrective step similar to \cite{boyd15}, see Algorithm~\ref{algo:ffw}. This last step compensates the slow convergence of Frank-Wolfe.

\paragraph{Linear minimization oracle} As mentioned above, the linear minimization step of Frank-Wolfe in our case simply amounts to computing a minor eigenvector of the gradient of $f$ (or more specifically of $J_\mell^{-1/2}\nabla f J_\mell^{-1/2}$, but we omit this scaling in the following for simplicity) at the current iterate. In practice,  we perform this step efficiently with power iterations, see Section~\ref{sec:fft-tricks}. In order to determine the lowest singular value of $\nabla f$ (simultaneously with a corresponding singular vector), we often need to run the algorithm two times consecutively: if after the first run, a negative eigenvalue $\la_1 < 0$ is returned, this is indeed the lowest one (since power iterations retrieve the eigenvalue whose magnitude is the greatest); if not, we re-run the algorithm on $\nabla f - \la_1 I$, and add the resulting value to $\la_1$ to obtain the lowest singular value of $\nabla f$.

\paragraph{Low-rank storage} To take advantage of the low-rank structure of the solutions, we store our iterates as $\Rr = \Uu\Uu^*$, and work only with the factor $\Uu$. This is a cornerstone of our approach, since in practice the matrix $\Rr$ is too large to be stored entirely. Furthermore, this factorization allows an efficient implementation of several steps of FFW (see Section~\ref{sec:fft-tricks}) that considerably lowers the complexity of the algorithm.

Consequently, at each step of the algorithm, the update consists in adding a column (namely a leading eigenvector of $\nabla f$, as mentioned above) at the end of the matrix $\Uu$, thus increasing by one the rank of $\Rr$ each time. 

\paragraph{Non-convex corrective step} The non-convex step that we add after each Frank-Wolfe update consists, as in \cite{boyd15}, in a gradient descent on $F: \Uu \mapsto f(\Uu\Uu^h)$. The idea is to continuously move the eigenstructue of the iterate in the manifold of fixed rank matrices to improve the value of the functional. This is similar to the celebrated Burer-Monteiro non-convex method for low-rank minimization , which has proven to be vey efficient in practice \cite{boumal16}.
We use a limited-memory BFGS descent in our implementation.

\paragraph{Stopping criterion}
It is known \cite{jaggi13} that if $\Ss$ is a solution of the linear minimization oracle, then it satisfies the inequality $\dotp{\Rr - \Ss}{\nabla{f}(\Rr)} \geq f(\Rr) - f(\Rr_\star)$ for any $\Rr$. We use this property as a stopping criterion, ceasing the iterations if $\dotp{\Rr - \Ss}{\nabla{f}(\Rr)}$ goes below some tolerance $\epsilon$.

\begin{algorithm}
\caption{}
\begin{algorithmic}
\STATE{\textbf{set}: $\Uu_0 = \left[0\ldots 0\right]^\top$, $D_0 = 2f(0)$}
\WHILE{$\dotp{\Uu_r\Uu_r^h - v_rv_r^h}{\nabla{f}(v_rv_r^h)} \geq \epsilon f(x_0)$}
	\STATE{\begin{tabular}{ll}
				1. 	& linear minimization oracle\\
					& \quad $v_r = D_0 J_\mell^{-\frac{1}{2}}\left(\arg\min_{\norm{v} \leq 1}  v^\top \cdot \left( J_\mell^{-\frac{1}{2}}\nabla f(\Uu_r\Uu_r^*)J_\mell^{-\frac{1}{2}} \right) \cdot v \right) J_\mell^{-\frac{1}{2}}$\\
				2. 	& update\\
					& \quad $\hat{\Uu}_{r+1} = \left[\sqrt{\alpha_r} \Uu_r, \sqrt{\beta_r} v_r \right]$, where\\
					& \quad $\alpha_r, \beta_r = \arg\min_{\alpha \geq 0, \beta \geq 0, \alpha + \beta \leq 1} f(\alpha \Uu_r\Uu_r^* + \beta v_rv_r^*)$\\
				3. 	& corrective step \\
					& \quad $\Uu_{r+1} = \textbf{bfgs}\enscond{\Uu \mapsto f(\Uu\Uu^*)}{\Uu \in \CC^{(\mell+1)\times(r+1)}, \; \text{starting from}\;\, \hat{\Uu}_{r+1}}$
			\end{tabular}}
\ENDWHILE
\RETURN{$(\Uu_{i,j})_{1\leq i \leq \mell, ~1\leq j \leq r+1}$}
\end{algorithmic}
\label{algo:ffw}
\end{algorithm} 

\subsection{Fast-Fourier-Transform-based computations}
\label{sec:fft-tricks}

In what follows, we consider a variable $\Rr$ of the form \eqref{eq:block-matrix}, of rank $r$, such that $\Rr = \Uu\Uu^*$ where
\eq{
	\Uu = \begin{bmatrix} U_1 \\ \zeta \end{bmatrix},
}
with $U_1 \in \CC^{\mell \times r}$ and $u \in \CC^r$. In particular, this yields $R = U_1U_1^*$, $z = U_1^*\zeta$ and $\tau = \norm{\zeta}^2$. For clarity, we keep using $R$, $z$ and $\tau$ in our expressions, but in practice only $U_1$ and $\zeta$ are stored. Finally, we assume that the matrices are indexed following the colexicographic order, so that generalized T\oe{}plitz matrices may be written in the form \eqref{eq:gen-toep-colex}. In this section only, we use the same symbol $\Ff$ to refer to the discrete Fourier transform (instead of the discrete-time Fourier transform).

\paragraph{Power Iterations} Computing a minor eigenvector of $\nabla f$ can be done using power iterations, which consist in recursively applying $\nabla f$ to a vector. Given the form \eqref{eq:objective} of the objective, its gradient at $\Rr$ reads
\eql{\label{eq:gradient}
	\Grad f = C_0
		\begin{bmatrix} 
			\frac{1}{2\mell}I + \frac{1}{\rho}(R - \projtoep(R))  & \frac{1}{2\la}\Aa^*(\Aa z - y) \\
			\\
			\frac{1}{2\la}(\Aa^*(\Aa z - y))^* & \frac{1}{2}
		\end{bmatrix},
}
so that multiplying it with a vector $w = \begin{bmatrix} w_1 \\ \om \end{bmatrix}$ ($w_1 \in \CC^\mell$, $\om \in \CC$) yields
\eq{
	\left(\nabla f\right) w = \frac{C_0}{2\mell} \begin{bmatrix} w_1 \\ 0\end{bmatrix} + \frac{C_0}{\rho} \begin{bmatrix}Rw_1 \\ 0 \end{bmatrix} + \frac{C_0}{2\la} \begin{bmatrix} \om \Aa^*(\Aa z - y) \\ \dotp{\Aa^*(\Aa z - y)}{w_1} + \la \om\end{bmatrix} - \frac{C_0}{\rho} \begin{bmatrix} \projtoep(R) w_1 \\ 0\end{bmatrix}.
}
While the first three terms in the sum above are quite straightforward to compute (remember that $Rw_1$ is computed as $U_1(U_1^*w_1)$), evaluating $\projtoep(R)w_1$ on the other hand can be costly. The next two propositions show that it can actually be performed in $O(\mell \log \mell)$ operations using only fast Fourier transforms.


\begin{prop}
Let $U_1^{(1)}, \ldots, U_1^{(r)}$ be the columns of $U_1$. Then $\projtoep(U_1U_1^*) = \sum_{k\in \Om_{2\ell}} u_k \Theta_k$, where 
\eq{
	u_k = \frac{1}{\mathrm{card}\enscond{(s,t) \in \Om_\ell}{s-t=k}} \left[ \sum_{j=1}^r \Ff^{-1} \left(\left|\Ff \left( \tilde{U}_1^{(j)}\right)\right|^2\right) \right]_k,
}
where $\tilde{U}_1^{(j)} \in \CC^{\abs{\Om_{2\ell}}}$ is defined as
\eq{
	\forall s \in \Om_\ell, \quad \forall k \in \Om_{2\ell}, \quad
	(\tilde{U}_1^{(j)})_{s-k} = \left\{ 
		\begin{array}{cl}
			(U_1^{(j)})_{s-k} & \text{if} \quad s-k \in \Om_\ell \\
			0 & \text{otherwise}
		\end{array},
	\right.
}

\end{prop}
\begin{proof} 
Since $\normd{\projtoep(U_1U_1^*)}^2 = \sum_{s,t \in \Om_\ell}\abs{u_{s-t}}^2$, we have:
\eq{
	\begin{aligned}
		\normd{\projtoep(U_1U_1^*) - U_1U_1^*}^2 &= \sum_{s,t \in \Om_\ell} \abs{u_{s-t} - \sum_j (U_1^{(j)})_s (U_1^{(j)})_t}^2 \\
		&= \sum_{k \in \Om_{2\ell}} \sum_{s-t=k} \abs{u_k - \sum_j (U_1^{(j)})_s (U_1^{(j)})_{s-k}}^2
	\end{aligned}
}%
Minimizing this quantity with respect to $u$ leads to
\eq{
	\forall k \in \Om_{2\ell}, \quad u_k = \frac{1}{\text{card}\enscond{(s,t) \in \Om_\ell}{s-t=k}} \sum_{s \in \Om_\ell} \sum_{j=1}^r (\tilde{U}_1^{(j)})_s (\tilde{U}_1^{(j)})_{s-k}.
}%
Then
\eq{
		\sum_{s \in \Om_\ell} (\tilde{U}_1^{(j)})_s (\tilde{U}_1^{(j)})_{s-k} = \left[ \tilde{U}_1^{(j)} * \tilde{U}_1^{(j)-}\right]_k
			= \left[\Ff^{-1}\left(\Ff(\tilde{U}_1^{(j)}) \cdot {\Ff(\tilde{U}_1^{(j)})^*}\right) \right]_k,
}%
which yields the desired result.
\end{proof} 

\begin{rem}
The quantity $\mathrm{card}\enscond{(s,t) \in \Om_\ell}{s-t=k}$ can be obtained by computing $\mathds{1}_{\Om_\ell} * \mathds{1}_{\Om_\ell}$, using again fast Fourier transforms.
\end{rem} 

\begin{prop}
Let $T \in \Tomat$, and write $T = \sum_{k \in \Om_{2\ell}} u_k \Theta_k$. Let $w_1 \in \CC^\mell$. Then
\eq{
	\forall k \in \Om_\ell, \quad (T w_1)_k = \left( \Ff^{-1} \left( \left\langle \Ff(u), \Ff\left(\tilde{w}_1\right) \right\rangle \right) \right)_k,
}%
where $\tilde{w}_1 \in \CC^{\abs{\Om_{2\ell}}}$ is defined as
\eq{
	\forall k \in \Om_{2\ell}, \quad (\tilde{w}_1)_k = \left\{ 
		\begin{array}{ll}
			(w_1)_k & \text{if} \quad k \in \Om_\ell\\
			0 & \text{otherwise}
		\end{array}
		\right.
}
\end{prop}
\begin{proof} 
The product $T w_1$ is the (aperiodic) convolution of $u$ and $w_1$, and may be formulated as a periodic convolution between $u$ and a zero-padded version of $w_1$, hence the result. 
\end{proof} 

We conclude this section by giving the closed-form expression for the linesearch coefficients $\al_r$ and $\be_r$ in step 2 of Algorithm~\ref{algo:ffw}, assuming in our notations that $\Uu_r = \Uu$ and $v_r = w$.
\begin{prop}
With the same notations as before, let
\eq{ 			
	\left\{
	\begin{aligned}
		c_{11} &\eqdef C_0 \left( \frac{1}{2\la}\normH{\Aa z}^2 + \frac{1}{2\rho}(\norm{R}^2 - \norm{\projtoep(R)}^2) \right)\\
		c_{22} &\eqdef C_0 \left( \frac{\om^2}{2\la}\normH{\Aa w_1}^2 + \frac{1}{2\rho}(\norm{w_1w_1^*}^2 - \norm{\projtoep(w_1w_1^*)}^2) \right)\\
		c_{12} &\eqdef C_0 \left( \frac{\om}{\la}\Re{\dotp{z}{\Aa^*\Aa(w_1)}_\Hilbert} + \frac{1}{\rho}\Re{\dotp{R}{w_1w_1^*} - \dotp{\projtoep(R)}{ \projtoep(w_1w_1^*)}} \right)\\
		c_1 &\eqdef C_0 \left( \frac{1}{2}(\tau + \frac{\Tr R}{\mell}) - \frac{1}{\la}\Re\dotp{y}{\Aa z}_\Hilbert \right)\\
		c_2 &\eqdef C_0 \left( \frac{1}{2}(\om^2 + \frac{\norm{w_1}^2}{\mell}) - \frac{\om}{\la}\Re\dotp{y}{\Aa w_1}_\Hilbert \right)
	\end{aligned}
	\right.
}%
and let $\Delta = \enscond{\al,\be \in \segc{0}{1}^2}{ \al + \be \leq 1}$. Then, the solutions $\al_r, \be_r$ of the linesearch of step 2 in Algorithm~\ref{algo:ffw} are given by
\eq{
	(\al_r, \be_r) = 
	\left\{
		\begin{array}{cl}
			P_\Delta( 0, -\frac{c_2}{2c_{22}} ) & \text{if} \quad c_{11} = 0 \qandq c_{22} \neq 0\\
			P_\Delta( -\frac{c_1}{2c_{11}}, 0 ) & \text{if} \quad c_{22} = 0 \qandq c_{11} \neq 0\\
			P_\Delta(\frac{c_{12} c_2 - 2c_{22}c_1}{4c_{11}c_{22} - c_{12}^2}, \frac{c_{12} c_1 - 2c_{11}c_2}{4c_{11}c_{22} - c_{12}^2}) & \text{otherwise}
		\end{array}
	\right.
}
\end{prop}
\begin{proof}
This is a straightforward consequence of minimizing the quadratic form $f(\al \Rr + \be ww^*) = c_{11}\al^2 + c_{22}\be^2 + c_{12}\al\be + c_1\al + c_2\be + \frac{1}{2\la}\norm{y}_H^2$ over $\Delta$. Furthermore, it is realistic to consider only the three cases above (in particular, one never has $c_{11} = c_{22} = 0$ or $4c_{11}c_{22} - c_{12}^2 = 0$ in practice).
\end{proof} 


\subsection{Implementation of the approximation matrix $\Aa$}
\label{sec:fft-implemented-A}
When dealing with problems that are not deconvolution, such as subsampled convolution or foveation,the approximation matrix $\A$ is a full matrix of size $\abs{\Gg} \times (2f_c+1)^d$, $\Gg$ being the grid on which the measurements live, making it numerically difficult to handle in practice. We give an efficient way to implement multiplication with $\A$ in the case of subsampled convolution, when the grid is a regular lattice. In other cases (foveation in particular), one needs to store the full matrix.

We consider convolution measurements over the grid $\Gg \eqdef \frac{1}{L}\segi{0}{L-1}^d$. The approximation matrix has the form $\A = (c_{-k}(\tilde{\phi})e^{2i\pi\dotp{k}{t}})_{t \in \Ga, k \in \Om_c}$, see Section~\ref{sec:fourier-approximation}. Let $S^\uparrow_q$ and $S^\downarrow_q$ be respectively the upsampling and downsampling (by a factor $q$) operators. For $K \geq 2f_c+1$, let $\Pad_K: \CC^{\abs{\Om_c}} \rightarrow \CC^{K^d}$, be defined as
\eq{
	\Pad_K : (z_k) \mapsto (Z_k), \qwhereq \forall k \in \segi{0}{K-1}^d, \quad Z_k = \left\{
		\begin{array}{cl}
			z_k & \text{if} \quad k-f_c \in \Om_c\\
			0 & \text{otherwise}
		\end{array}
		\right.
}
and let $\mathrm{Restr}_{\Om_c}: \CC^{K^d} \rightarrow \CC^{\abs{\Om_c}}$ be its adjoint. Finally, let $q \in \NN$, such that $q \geq \lceil \frac{2f_c+1}{L} \rceil$, and let $e_c \eqdef \left(e^{-\frac{2i\pi}{L}\dotp{f_c}{n}}\right)_{n \in \segi{0}{Lq-1}^d}$.

The following proposition gives an efficient $O(L^d \log L)$ implementation of the multiplication by $\A$ or $\A^*$.
\begin{prop} \label{prop:A-fft-implementation}
Given $z \in \CC^{\abs{\Om_c}}$ and $y \in \CC^{\abs{\Gg}}$, one has
\eq{
	\begin{aligned}
		&\A z = (Lq) S^\downarrow_q \left( e_c \odot \mathrm{iDFT} (\Pad_{Lq}(D(\phi) z) ) \right),\\
		&\A^* y = D(\phi)^* \mathrm{Restr_{\Om_c}}( \mathrm{DFT}(e_c^* \odot S^\uparrow_q(y))),
	\end{aligned}
}
where $D(\phi) \eqdef \Diag(c_{-k}(\tilde{\phi}))$, and $\odot$ stands for element-wise multiplication.
\end{prop}
\begin{proof} For $t \in \Ga$, we write $ t = \frac{n}{L}$, with $0 \leq n \leq L-1$.  Then
\eq{
	\begin{aligned}
		\forall t \in \Ga, \quad (\A z)_t &= \sum_{k\in \Om_c} e^{2i\pi \dotp{k}{t}}(Dz)_k\\
			&= e^{-2i\pi \dotp{f_c}{t}} \sum_{k \in \segi{0}{2f_c}^d} e^{\frac{2i\pi}{L}\dotp{k}{n}} (Dz)_{k-f_c}\\
			&= e^{-2i\pi \dotp{f_c}{t}} \sum_{k \in \segi{0}{Lq-1}^d} e^{\frac{2i\pi}{Lq}\dotp{k}{qn}} (\Pad_{Lq}(Dz))_k\\
			&= (Lq) \left( e_c \odot \mathrm{iDFT}( \Pad_{Lq}(Dz))\right)_{qn}
	\end{aligned}
}
which yields the first equality. The second equality on the other hand is obtained by passing to the adjoint.
\end{proof}

\subsection{Complexity} 
Using the FFT implementations described in the two previous sections, we are able to decrease the computational cost of the two elementary operations in FFW: evaluating $f'(\Uu\Uu^*) w$ and evaluating $F'(\Uu)$, for $\Uu \in \CC^{(\mell+1)\times r}$ and $w \in \CC^{\mell+1}$ (and $\mell = (2\ell+1)^d$, $\ell \geq f_c$). Table~\ref{table:complexity} summarizes the costs of both these operations in the three settings we consider in this paper, \ie convolution (for which the approximation matrix $\A$ is diagonal, see Section~\ref{sec:fourier-approx:examples}), subsampled convolution (for which $\A$ is implemented using Prop.~\ref{prop:A-fft-implementation}) and foveation (for which $\A$ is a dense matrix).
\begin{table}[h]
\begin{tabular}{c|c|c}
	& $f'(\Uu\Uu)$ w & $F'(\Uu)$ \\
	\hline
	& & \\
	convolution & $O(r\ell^d \log \ell)$ & $O(r^2\ell^d + r\ell^d \log \ell)$ \\
	& & \\
	subsampled & $O(r\ell^d \log \ell + $ & $O(r^2\ell^d + r\ell^d \log \ell + $ \\
	convolution & $\max(L,f_c)^d\log \max(L,f_c))$ & $\max(L,f_c)^d\log \max(L,f_c))$ \\
	(with regular grid) & & \\
	& & \\
	foveation & $O(r\ell^d \log \ell + L^d f_c^d)$ & $O(r^2\ell^d + r\ell^d \log \ell + L^d f_c^d)$
\end{tabular}
\caption{Computational costs. For the subsampled convolution and foveation cases, where the measurement live on a grid $\Gg$, we assume $\abs{\Gg} = L^d$.}
\label{table:complexity}
\end{table}

\section{Numerics}
\label{sec:numerics}

We study in this section the behavior of FFW with respect to parameters such as sparsity, minimal separation distance, $\la$, $\rho$ or the number of BFGS iterations.

\paragraph{Scalings} As already mentioned, in all our tests, we consider the objective multiplied by $C_0 = 2\la/\norm{y}^2$ (and its gradient accordingly), so that $f(0) = 1$. If the observations lie on a grid $\Gg$, \ie $\Hilbert = \CC^{\abs{\Gg}}$, we choose $\normH{\cdot} = \frac{1}{\sqrt[d]{\abs{\Gg}}}\norm{\cdot}$. Finally, we scale the parameter $\la$ of the BLASSO with $\normi{\Phi^*y}$, \ie we set $\la = \la_0 \normi{\Phi^*y}$.

\paragraph{Power Iteration step} The tolerance for the power iteration step is set to $10^{-8}$, with a maximum of $2000$ iterations. Iterations are stopped when the angle between the eigenvectors returned by two consecutive steps goes below the tolerance.

\paragraph{BFGS step} We use Mark Schmidt's code for the BFGS solver \cite{schmidt05}. The tolerance for this step is set to $10^{-11}$ (in terms of functions or parameters changes), with typically $500$ as the maximum number of iterations. This step is crucial to ensure the finite convergence of the algorithm. When $\rho$ tends to zero, plain Frank-Wolfe steps become significantly insufficient, and the number of BFGS iterations necessary to converge at each step increases, see Fig.~\ref{fig:perf-cost}.

\paragraph{Stopping criterion} As explained in Section~\ref{sec:ffw-algo}, the linear minimization oracle in Frank-Wolfe gives access to a bound on the current duality gap, which can then be used to decide when to stop the algorithm. However, it is difficult to define a stopping criterion based on this property that remains stable from one kernel to another. Therefore, in our implementation, we rather use the objective decrease as stopping criterion, and stop the iterations when $\abs{f(\Rr_{t+1}) - f(\Rr_t)}$ goes below some tolerance $\eps$ (where $f$ is the normalized objective). In the tests presented in this section, $\eps$ is set to $10^{-8}$.

\paragraph{Support extraction} In the tests presented here, the extraction step (see Section~\ref{sec:support-extraction}) is perfomed following the approach of \cite{josz17}, using the implementation of C\'edric Josz, that he kindly let us use.

\subsection{Tests on synthetic data}
To generate our measurements, we follow \eqref{eq:obs}. The ground-truth measures $\mu_0$ are randomly generated: specifically, we draw random positions uniformly over $\Torus^d$, and random amplitudes uniformly over $\segc{-1}{1}$. Figure~\ref{fig:examples} gives instances of reconstruction in each setting described in Section~\ref{sec:fourier-approx:examples}, with reconstruction errors with respect to $\mu_0$.

\begin{figure}
	\centering
	\captionsetup[subfigure]{labelformat=empty}
	\subfloat[$f_c = 15$, $\frac{\norm{w}}{\norm{y_0}} = 10^{-4}$]{\includegraphics[width=0.25\textwidth]{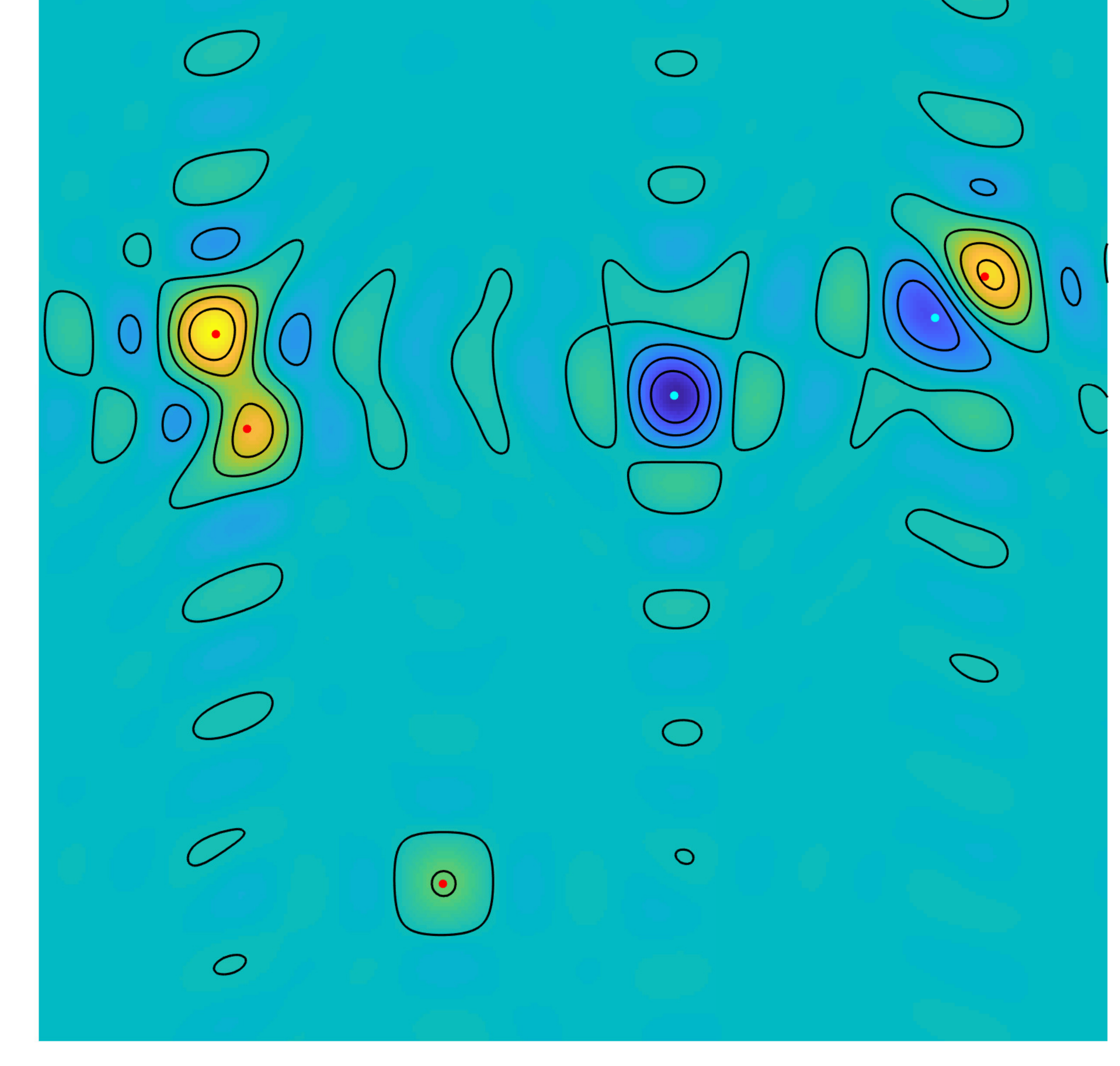}}
	\subfloat[$f_c = 30$, $\frac{\norm{w}}{\norm{y_0}} = 4.10^{-5}$]{\includegraphics[width=0.25\textwidth]{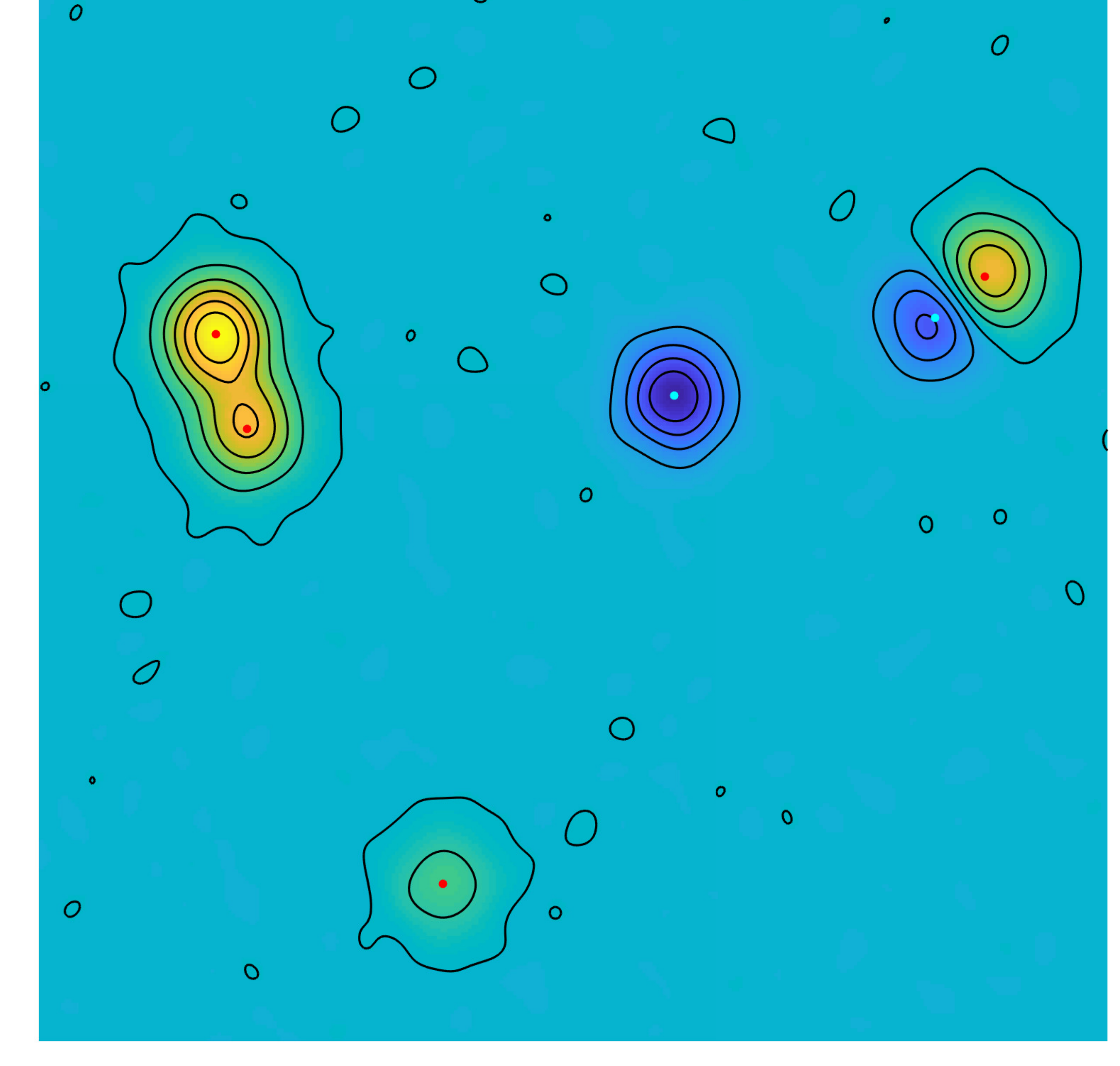}}
	\subfloat[$f_c=30$, $\frac{\norm{w}}{\norm{y_0}} = 10^{-2}$, $\Gg : 64\times64$]{\includegraphics[width=0.25\textwidth]{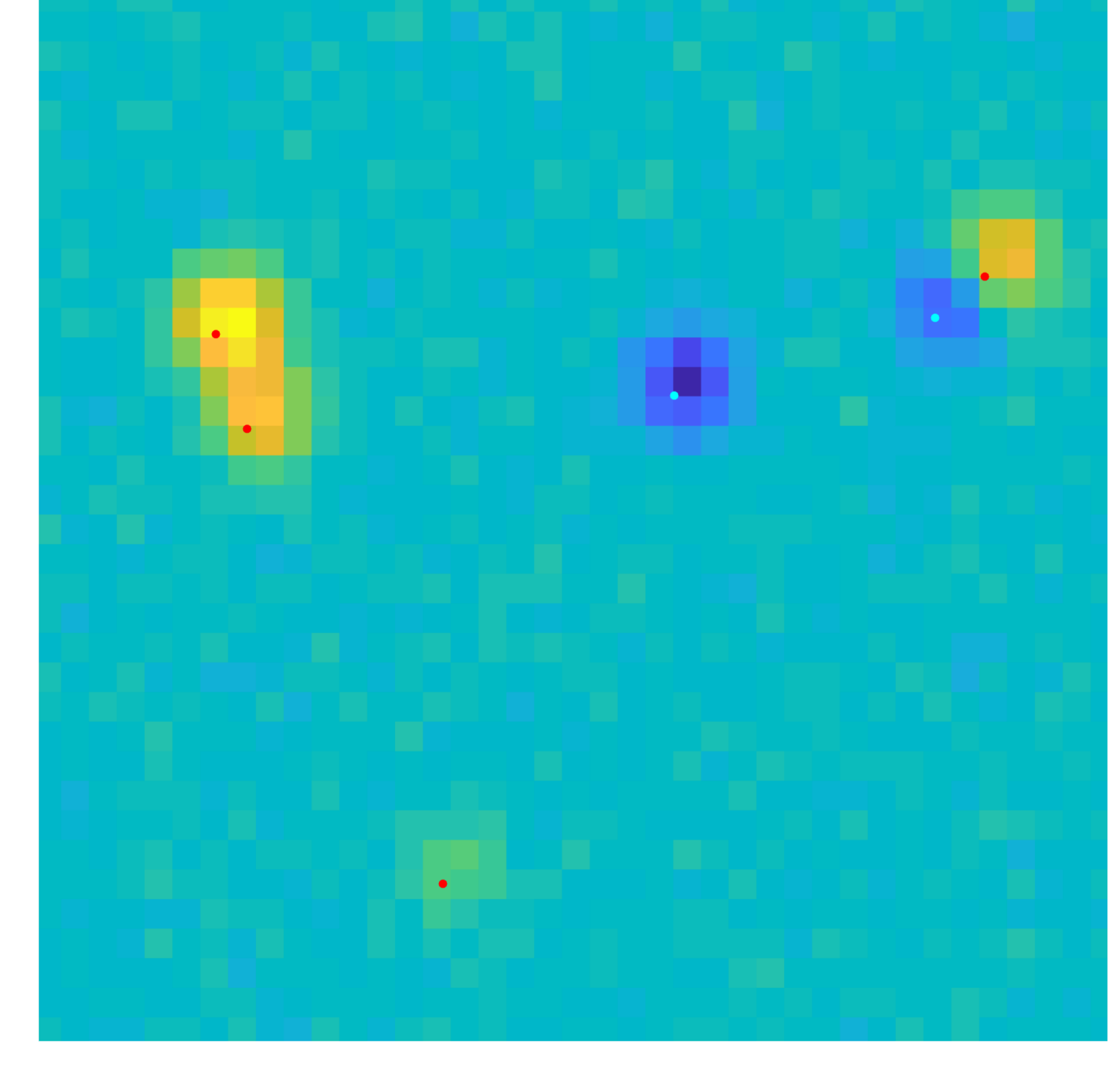}}
	\subfloat[$f_c=30$, $\frac{\norm{w}}{\norm{y_0}} = 10^{-3}$, $\Gg : 64\times64$]{\includegraphics[width=0.25\textwidth]{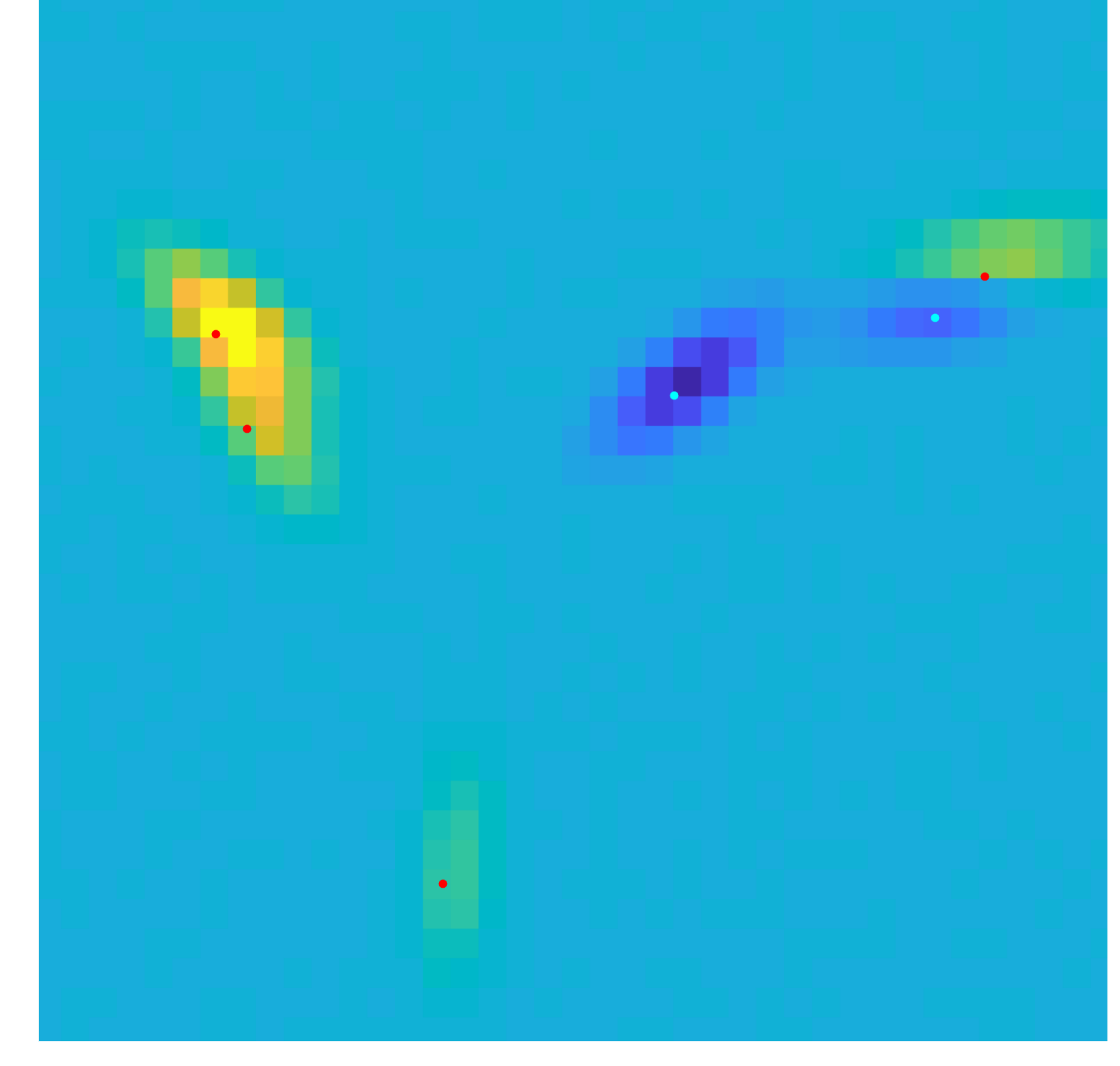}}\\
	\subfloat[$\la_0 = 2.10^{-3}$, $\rho=10^3$]{\includegraphics[width=0.25\textwidth]{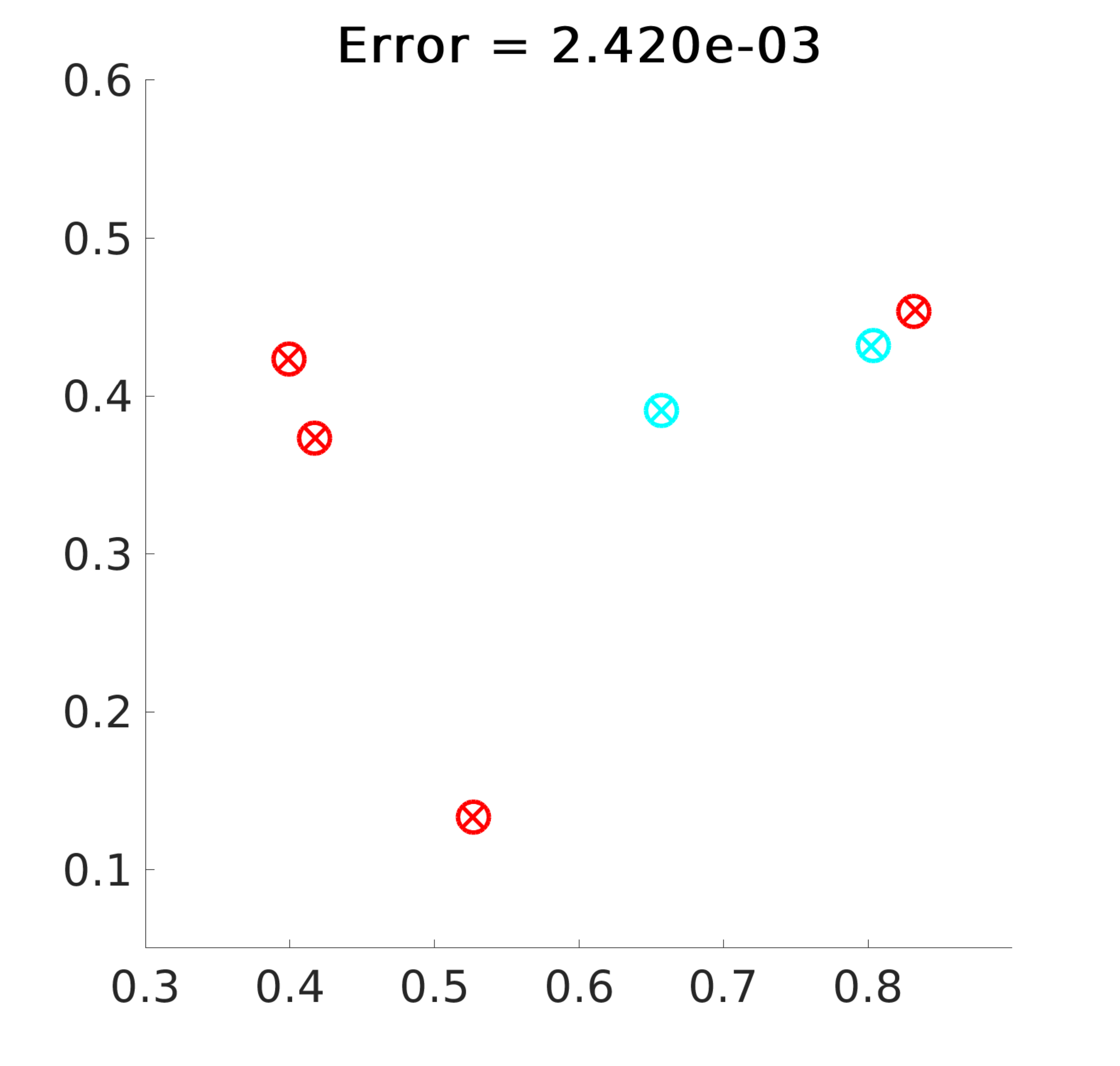}}
	\subfloat[$\la_0 = 2.10^{-3}$, $\rho=10^3$]{\includegraphics[width=0.25\textwidth]{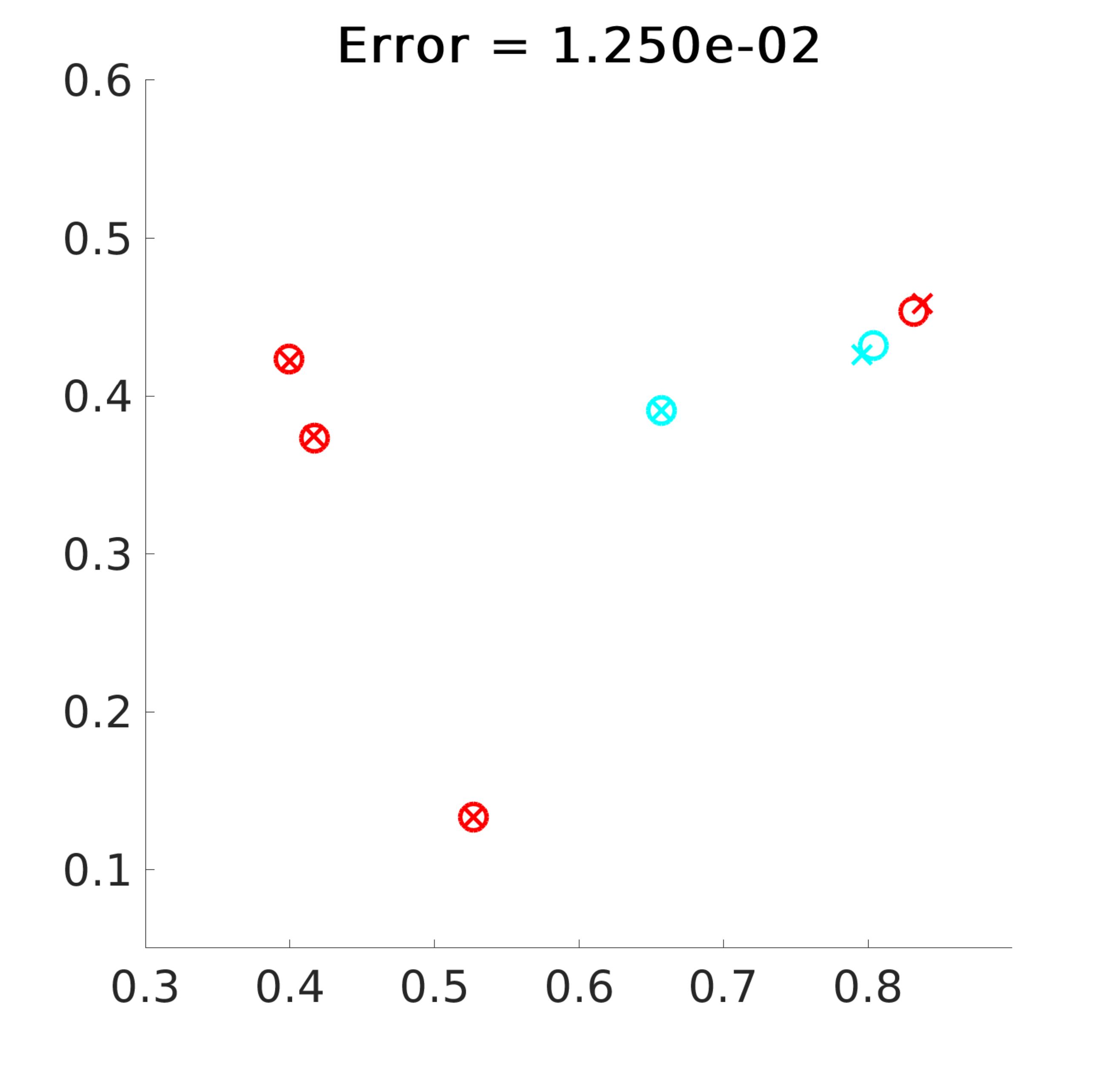}}
	\subfloat[$\la_0 = 10^{-3}$, $\rho=10^4$]{\includegraphics[width=0.25\textwidth]{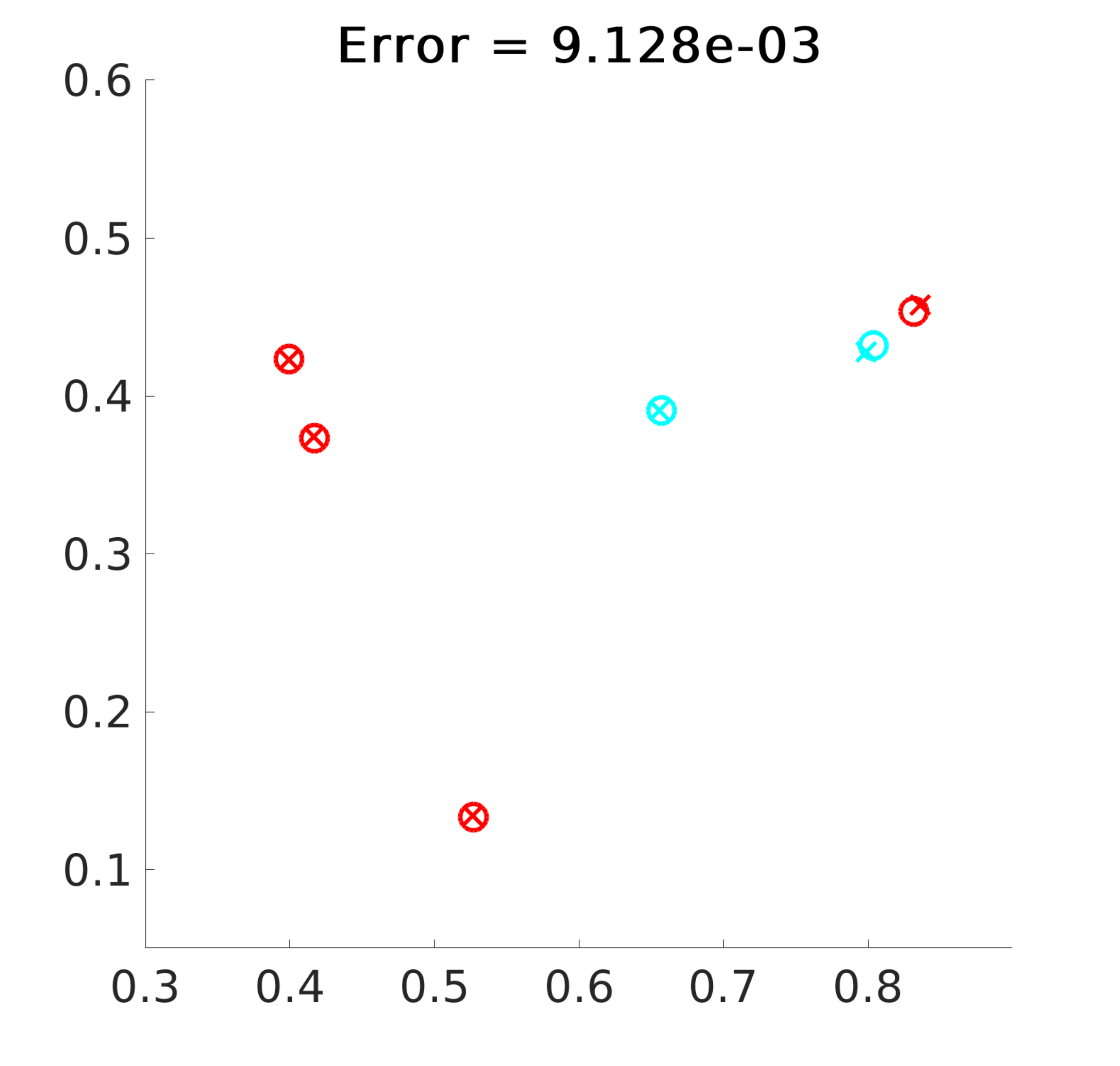}}
	\subfloat[$\la_0 = 10^{-3}$, $\rho=10^4$]{\includegraphics[width=0.25\textwidth]{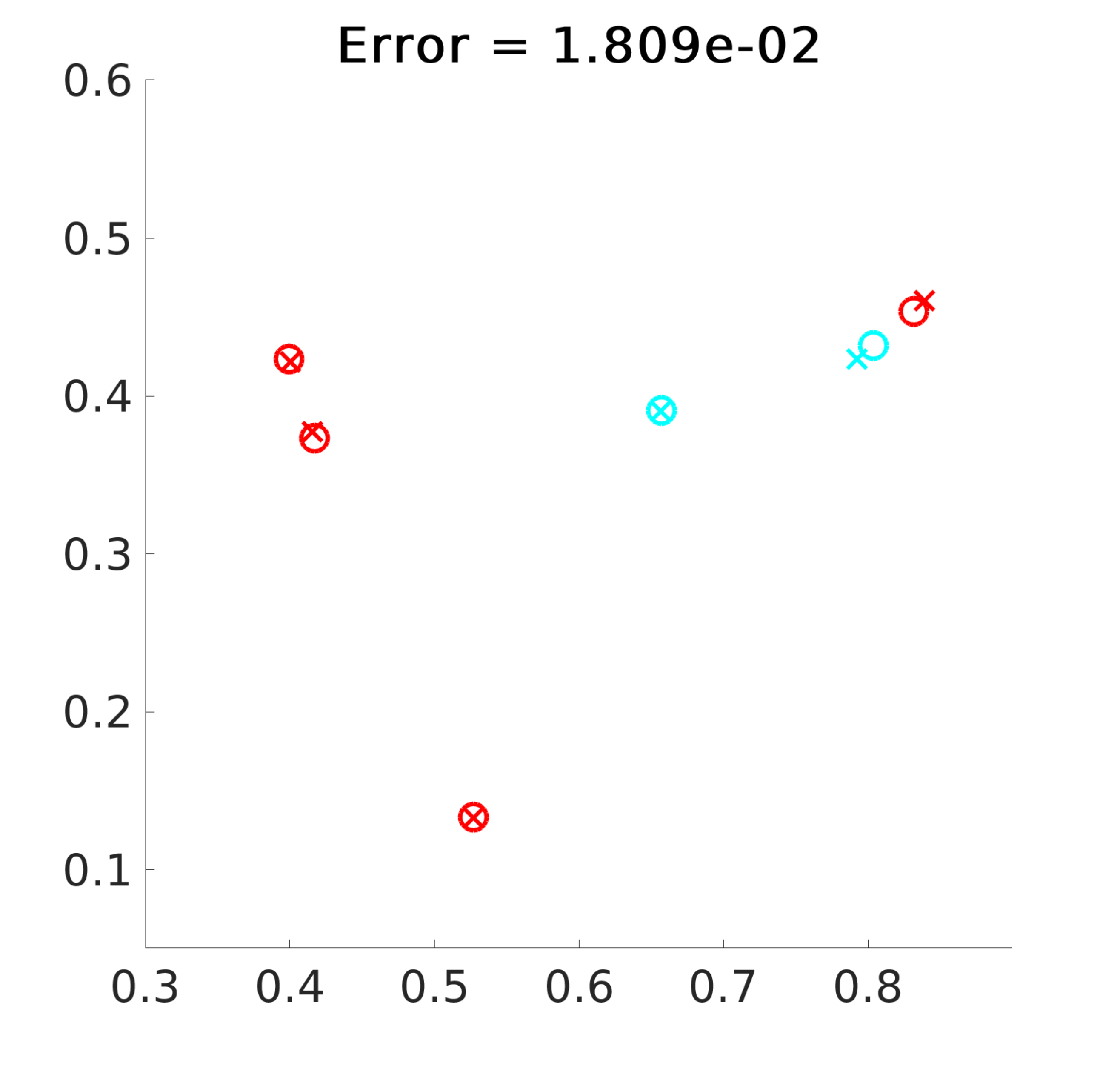}}
	\caption{From left to right: Measurements $y = \Phi\mu_0 + w$ (we plot $\Ff_c^* y$ in the first two figures) in the case of Dirichlet convolution, Gaussian convolution, Subsampled Gaussian convolution, and (Subsampled) Gaussian foveation. The support of $\mu_0$ is represented by red (positive spikes) and blue (negative spikes) dots. On the bottom line, the indicated errors are defined as $\norm{x_0 - x_r}/\norm{x_0}$, $x_0$ and $x_r$ being respectively the ground-truth and the reconstructed supports.}
	\label{fig:examples}
\end{figure}

\paragraph{Finite convergence} Remarkably, FFW converges in very few steps, usually as many as the number of spikes composing the solutions. Fig.~\ref{fig:finite-convergence} shows the number of FFW iterations with respect to the sparsity of the initial measure, averaged over $200$ random trials, \ie random positions and random amplitudes. The red curve correspond to a subset of these trials for which the minimal separation distance is greater than $1/f_c$. We see that in these simpler cases, FFW converges exactly in $r$-steps, $r$ being the number of spikes composing the solution.

\begin{figure}
	\centering
	\includegraphics[width=0.5\textwidth]{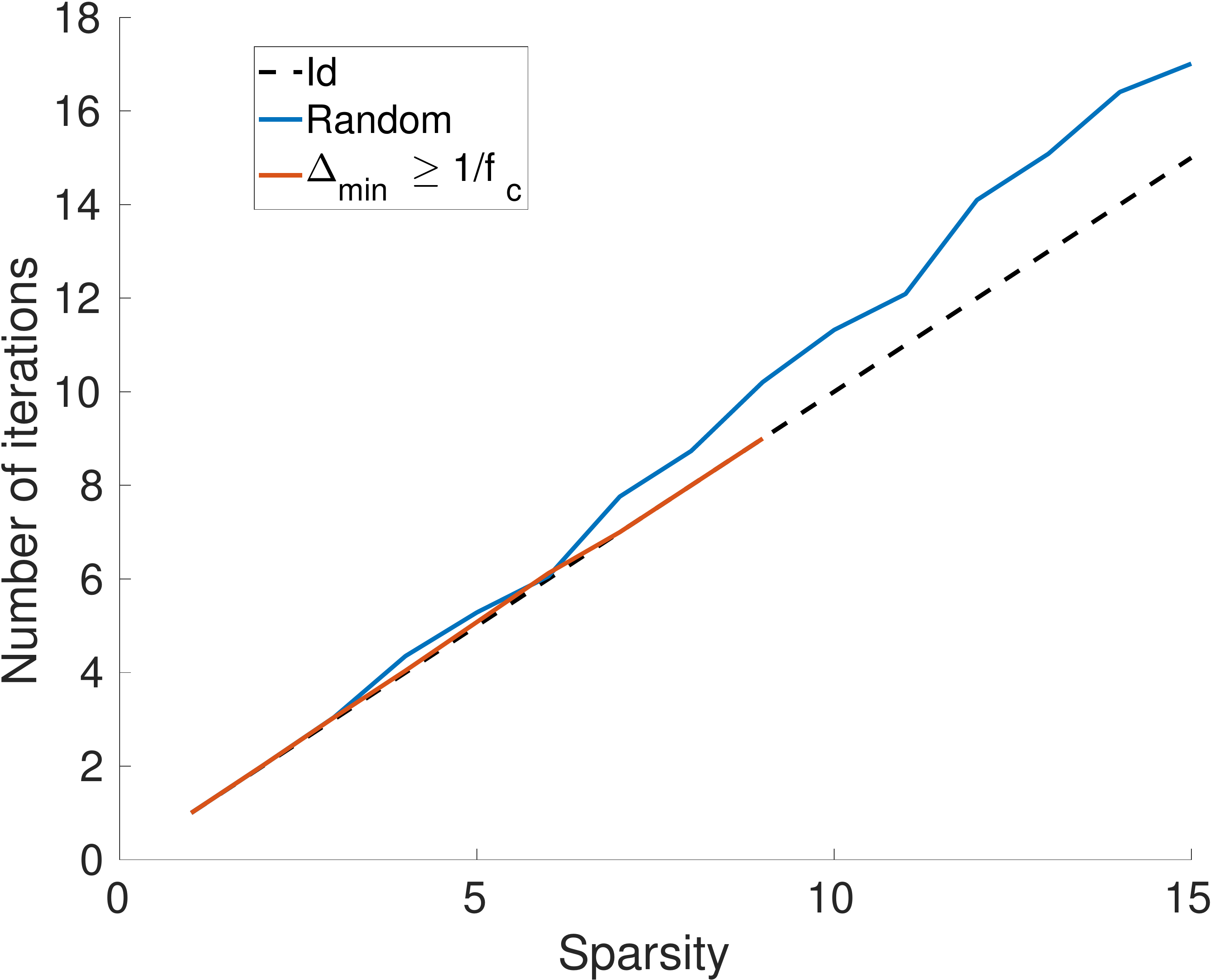}
	\caption{Number of FFW iterations with respect to sparsity of the initial measure.}
	\label{fig:finite-convergence}
	\postimagespace
\end{figure} 

\paragraph{Performance} Several metrics can be used to measure the recovery performance of FFW. We use two of them in our tests: the Jaccard index \cite{jaccard01}, and the flat norm \cite{Federer}, also called dual bounded Lipschitz norm. The Jaccard index measures the similarity between the initial (finite) support $S_0 \subset \Torus^d$ and the (finite) reconstructed support $S_r \subset \Torus^d$. It is defined as
\eq{
	J \eqdef \frac{\abs{S_0 \cap S_r}}{\abs{S_0 \cup S_r}} = \frac{\abs{S_0 \cap S_r}}{\abs{S_0}+\abs{S_r}-\abs{S_0 \cap S_r}}.
}
The value $\abs{S_0 \cap S_r}$ is determined with respect to some tolerance $\de$: given $x_0 \in S_0$, we consider that $x_r \in S_r$ belongs to $S_0 \cap S_r$ if $\norm{x_0 - x_r} \leq \de$ ($x_r$ is called a \emph{true positive}), and $x_0$ cannot be associated to more than one point in $S_r$. In our tests using the Jaccard index (see Section~\ref{sec:real-data-tests}), we set $\de = 10^{-2}$. The flat metric on the other hand is an optimal transport based metric, which we use to measure how close the reconstructed signal is from the source signal (with respect to both positions and amplitudes). It can be computed using linear porgramming.

The first plot in Fig.~\ref{fig:perf-cost} shows the flat distance between the measure $\mu_{\la,\rho}$ reconstructed by FFW, and a solution $\mu_\la$ of \eqref{eq:blasso} computed using MOSEK. The results are averaged over $680$ trials (random positions, random amplitudes and random sparsities in $\segi{2}{8}$), and sorted with respect to the minimal separation distance, either lower than $1/f_c$ (dashed line, 417 cases) or greater (solid line, 263 cases). As expected, the quality of the reconstruction decreases as the relaxation parameter $\rho$ increases. On the other hand, the computational cost, represented in the second figure in terms of total number of FFT performed, decreases as $\rho$ increases. For this experiment, the maximum number of BFGS iterations was set to 1000. This cost comes essentially from the BFGS iterations. In all our $1$D tests, setting $\rho$ between  $1$ and $10$ gave good performance. The sweet spot seems to strongly depend on the dimension $d$: for $d=2$, better performance is achieved with $\rho$ of the order of $10^3$ or $10^4$, see Fig.~\ref{fig:bench}.

\begin{figure}
	\centering
	\captionsetup[subfigure]{labelformat=empty}
	\subfloat[]{\includegraphics[width=0.48\textwidth]{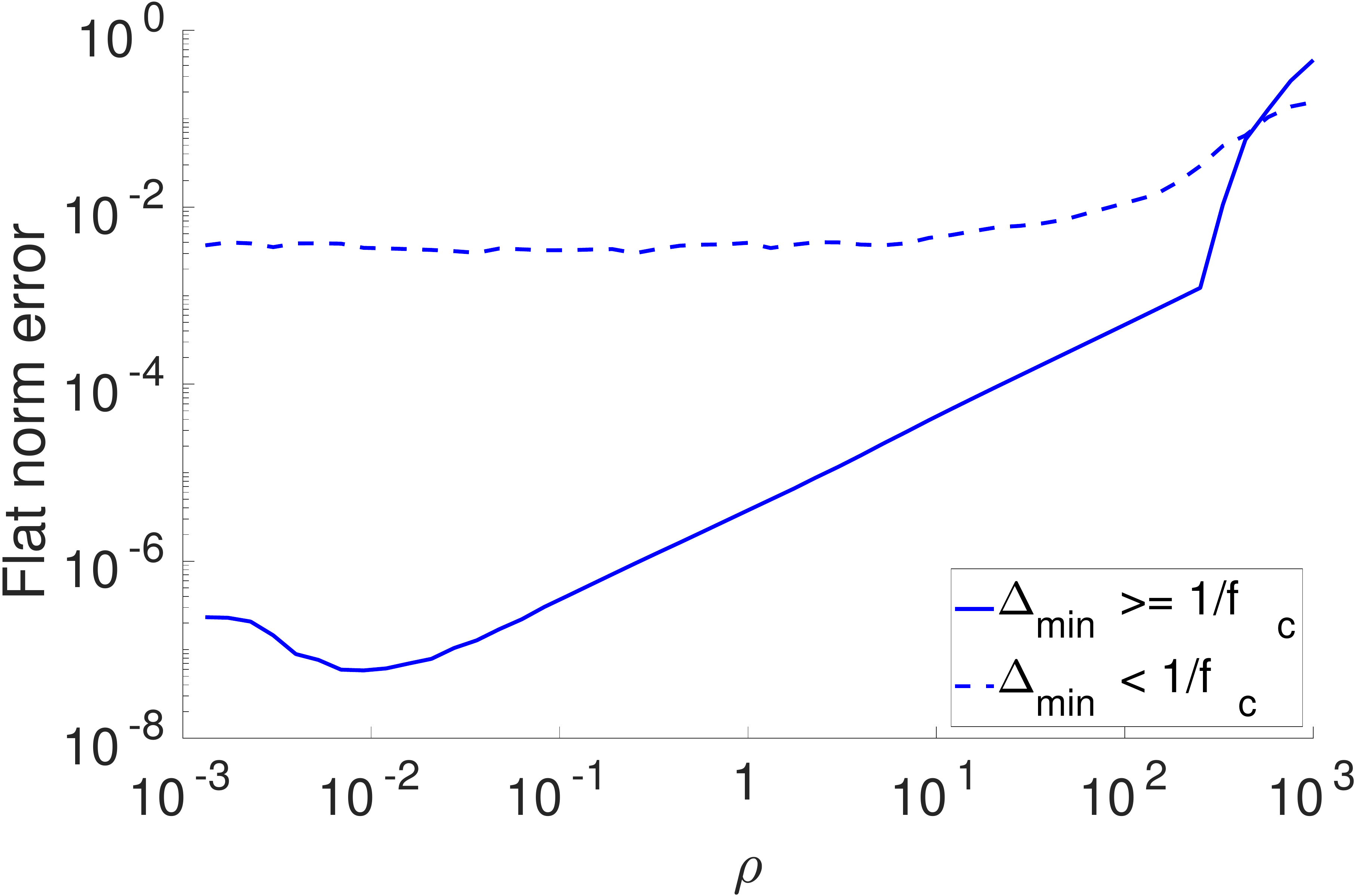}}
	\subfloat[]{\includegraphics[width=0.48\textwidth]{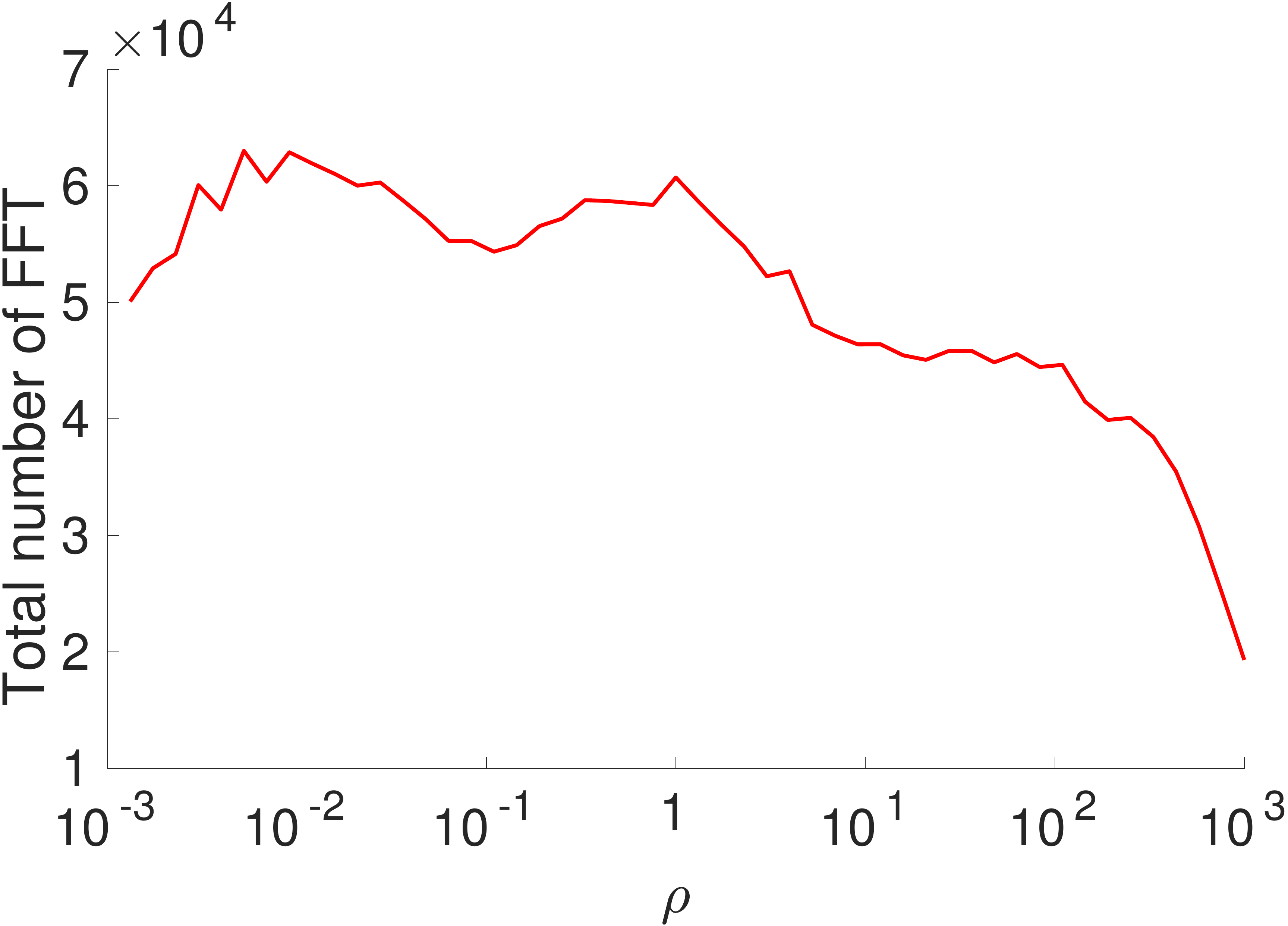}}
\caption{Performance (left) and computational cost (right).}
\label{fig:perf-cost}
\postimagespace
\end{figure} 


\subsection{Tests on SMLM data}
\label{sec:real-data-tests}
We present some results of FFW applied to data taken from the SMLM challenge \cite{smlm}. Fig.~\ref{fig:smlm-example} shows an example of reconstruction for one image of the challenge. On these data, the performance is measured by the Jaccard index, since the challenge gives information only about the locations (and not the amplitudes) of the Dirac masses.

\begin{figure}
	\centering
	\captionsetup[subfigure]{labelformat=empty}
	\subfloat[]{\includegraphics[width=0.4\textwidth]{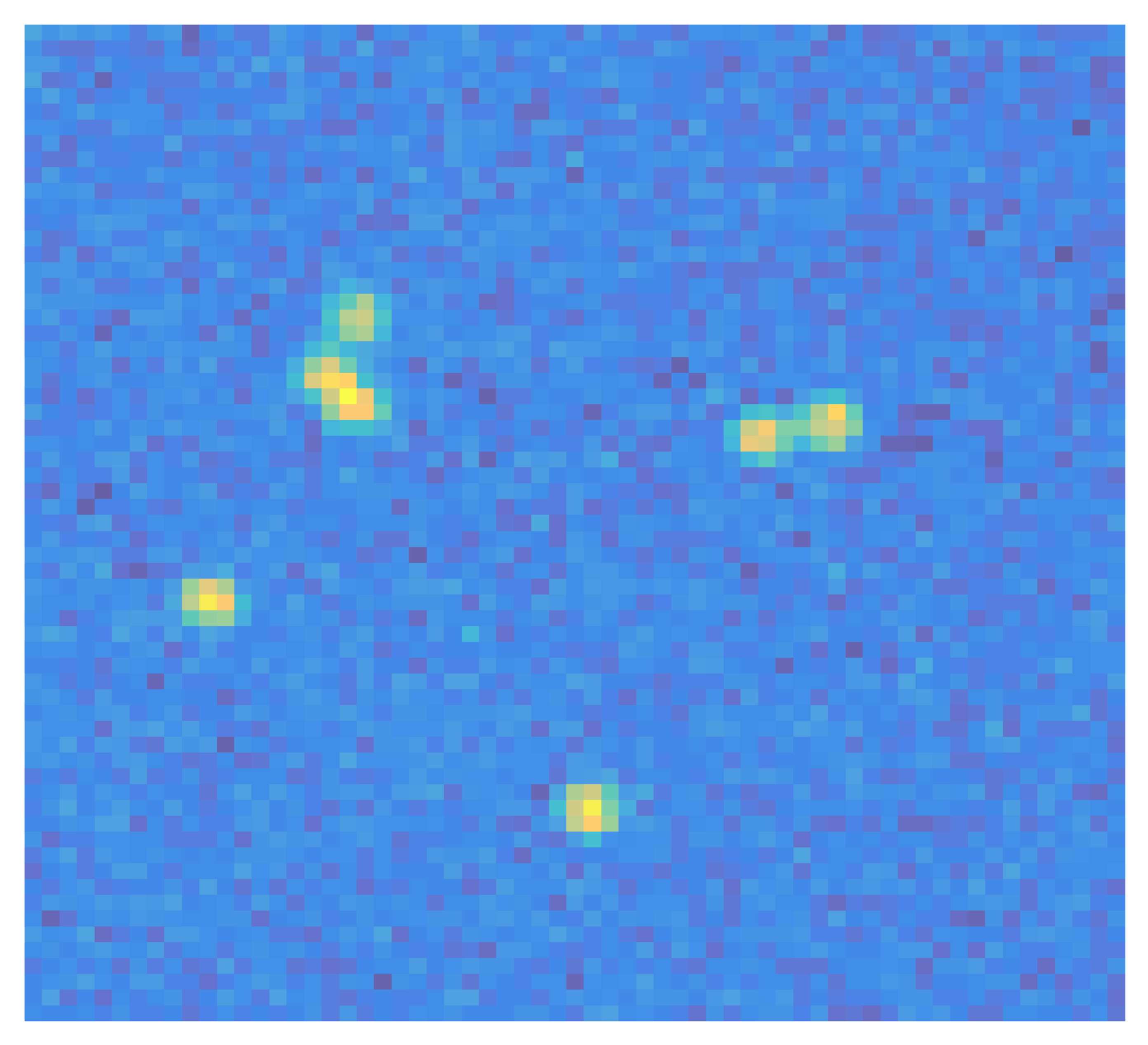}}
	\subfloat[]{\includegraphics[width=0.44\textwidth]{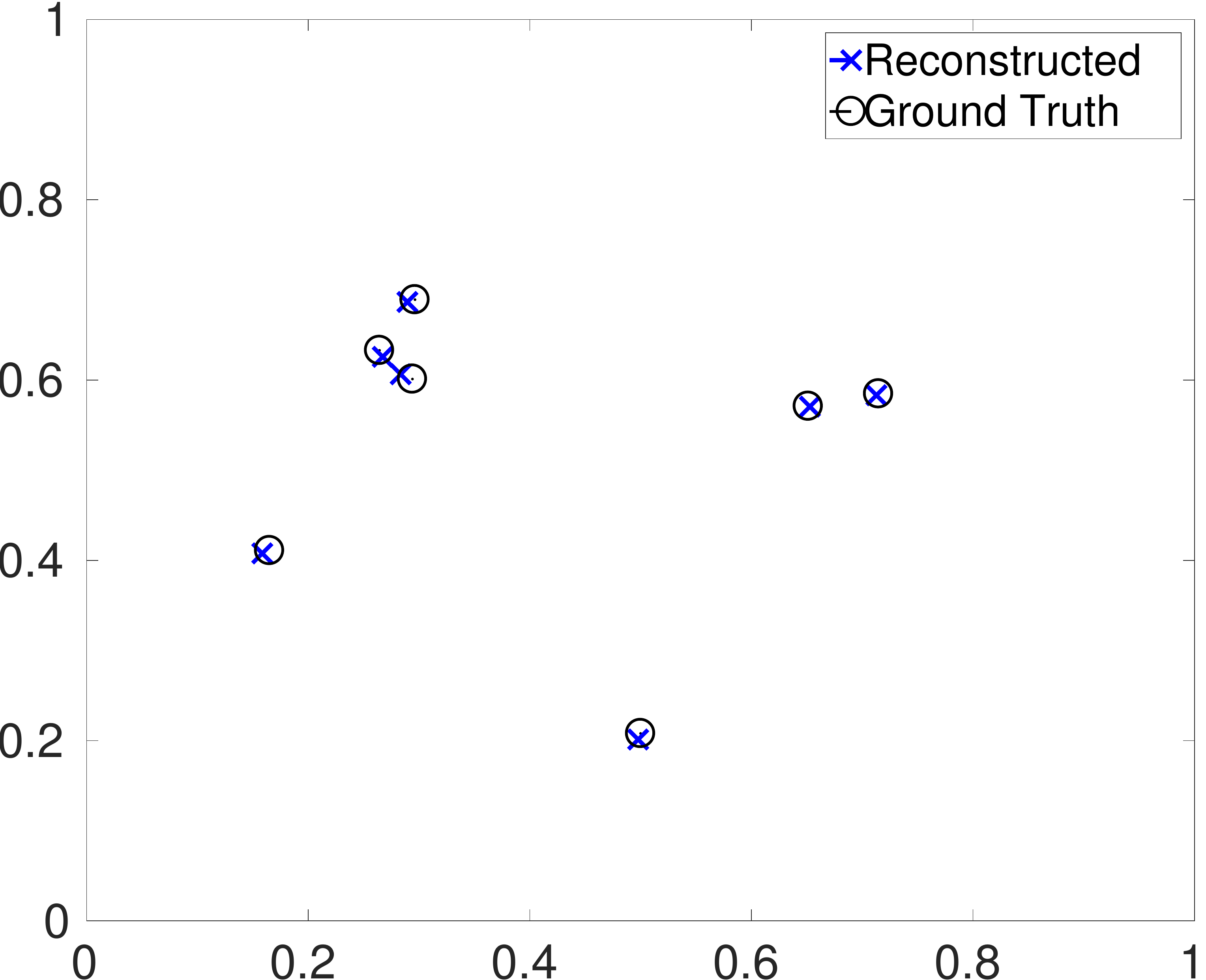}}
	\caption{Example of reconstruction on data from the smlm challenge. Relative error is $\norm{x_{\text{rec}}-x_0}/\norm{x_0} = 1.57 \times 10^{-2}$}
	\label{fig:smlm-example}
\end{figure}

In FFW, the most costly step is the BFGS step. Fig.~\ref{fig:bfgs-step} shows the impact of diminishing the maximum number of BFGS iterations  on the quality of the reconstruction (measured in terms of Jaccard index). The red solid line represents the time taken by FFW, in seconds, and the dashed line the time spent in the BFGS iterations. Results are averaged over $20$ random images taken from the challenge. We see that a low bound on the number of BFGS iterations strongly deteriorates the performance. On the other hand, we do not gain much by setting this bound higher than $250$.

\begin{figure}
	\centering
	\includegraphics[width=0.5\textwidth]{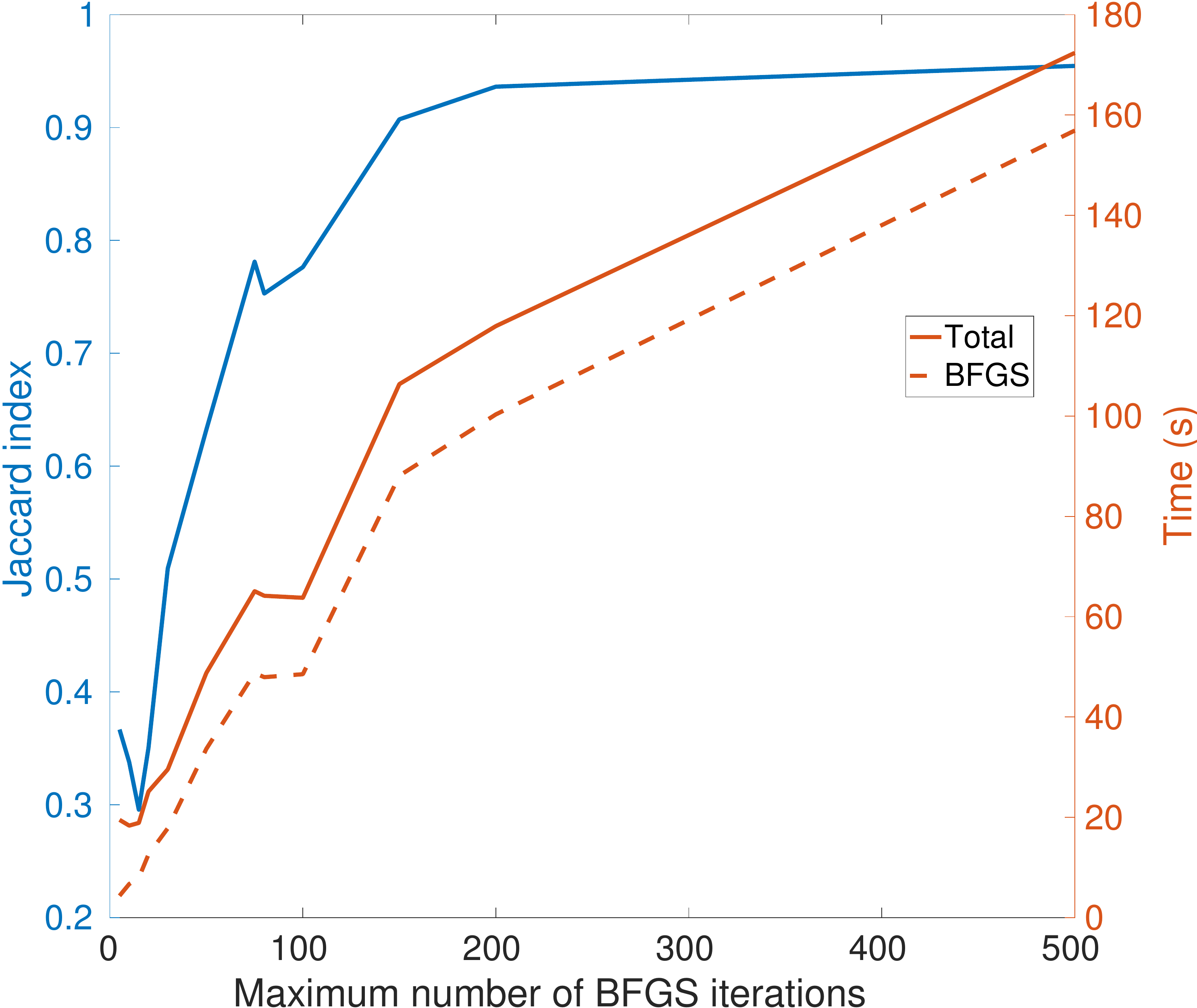}
	\caption{Perfomance versus maximum number of BFGS iterations}
	\label{fig:bfgs-step}
\end{figure}

Finally, Figure~\ref{fig:bench} shows the Jaccard index with respect to parameters $\rho$ and $\la_0$. Each pixel is obtain by averaging over 20 random images taken from the challenge. This gives an idea on the range of choices for $\rho$ and $\la_0$ in which FFW performs well. Although the choices for $\lambda$ does not change much following the different settings (kernel, dimension), the best values for $\rho$ seems to depend on $d$, as mentioned above.

\begin{figure}
	\centering
	\includegraphics[width=0.5\textwidth]{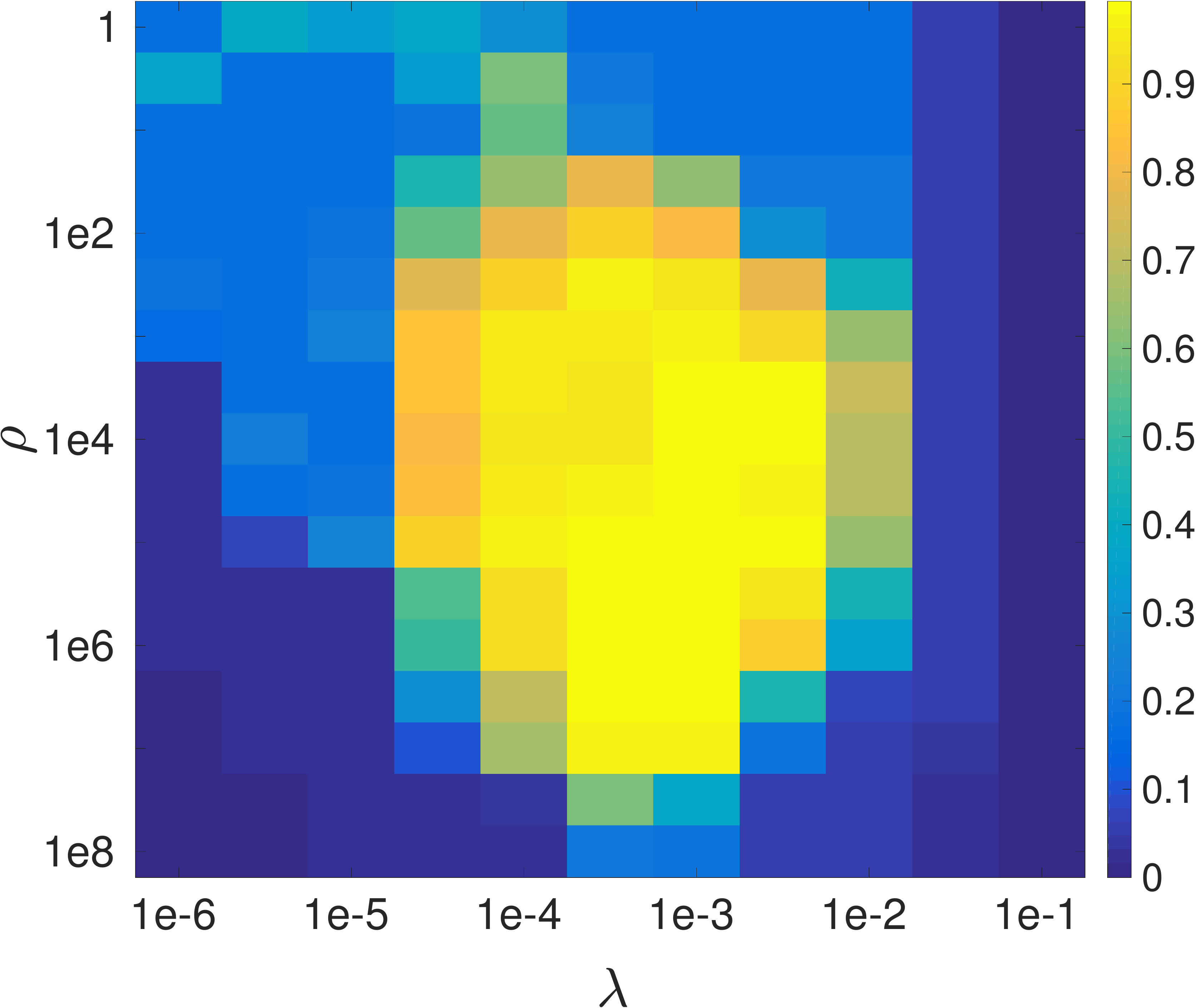}
	\caption{We measure the performance (in terms of Jaccard index) of FFW with respect to parameters $\la$ and $\rho$. Each pixel is obtained by averaging over 20 images.}
	\label{fig:bench}
\end{figure}

\section{Conclusion}
\label{sec:conclusion}

In this paper, we introduce a numerical solver for the sparse super-resolution problem in dimension greater than one and in various settings, including deconvolution. Our approach is based on the one hand on a spectral approximation of the forward operator, and on the other hand on a hierarchy of semidefinite relaxations of the initial problem, inspired by the Lasserre hierarchy \cite{Lasserre}. Low-rank solutions of the resulting SDP are then extracted using a FFT-based Frank-Wolfe method, which is scalable with the dimension. An interesting line of future works could be to extend the FFW algorithm to other problems that fit into the semidefinite hierarchies framework, such as optimal transport.

\section*{Acknowledgements}

This work was supported by grants from R\'egion Ile-de-France and the ERC project NORIA.

\appendix
\section{Proof of Theorem 1}
\label{app:flatness}

In this section, we prove Theorem~\ref{thm:flatness}, that is, we explain the connection between flat extensions (in the sense of Definition~\ref{def:flat-matrix}) and the existence of a representing measure $\nu$. The proof essentially relies on the ideas of Curto and Fialkow \cite{curto91,curto02,curto96} for the truncated $K$-moment problem (see also the variant \cite{laurent05} by Laurent) though our problem does not exactly fit into the framework of the aforementioned articles. While experts in this topic should have no difficulty in filling he gap, we present here some detail which might help the non-specialists.

In the monograph \cite{curto96}, the authors consider moments matrices with multivariate polynomials in $(z,\conj{z})$. Seeing the torus $\TT$ as the complex unit circle makes it possible to reformulate our problem to their setting, except that their moment matrices are indexed by monomials with total degree less than $\ell$,
\eql{
	\jm_1, \jm_2 \in \NN^d, \qandq \deg_1(\conj{\Z}^{\jm_1}\Z^{\jm_2}) \eqdef j_{1,1} + j_{2,1} + \ldots + j_{1,d} + j_{2,d} \leq \ell,
}
whereas the indices in our framework are naturally selected by their maximum degree,
\eql{
	\im \in \ZZ^d, \qandq \deg_{\infty}(\Z^\im) \eqdef \abs{\im}_{\infty} \eqdef \max(\abs{i_1}, \ldots, \abs{i_d}) \leq \ell.
}
On the contrary, \cite{laurent09} handles moment matrices indexed with more general sets of indices, but their analysis is given in the real case.

In the following, in order to derive Theorem~\ref{thm:flatness}, we combine the ideas of \cite{laurent09} and \cite{curto02}. We consider a positive Borel measure $\nu$ defined on the torus, and we see it alternatively as a measure in the complex plane or in the real plane, going through the following different moment matrices.

\paragraph{Trigonometric moment matrices} The moment matrices considered in this paper are mainly trigonometric moment matrices. Given a positive measure on $\Torus^d$, its (trigonometric) moment matrix $M_{\Torus}$ has entries
\eql{
	(M_{\Torus})_{\im,\jm} = \int_{\Torus^d} e^{-2i\pi\dotp{\im}{t}} e^{2i\pi\dotp{\jm}{t}} \diff\nu_{\Torus}(t) = \int_{\Torus^d}e^{2i\pi\dotp{\jm-\im}{t}} \diff\nu_{\Torus}(t),
}
for $\im, \jm \in \ZZ^d$, $\abs{\im}_{\infty}, \abs{\jm}_{\infty} \leq \ell$. Such matrices are \emph{generalized T\oe{}plitz} in the sense that for all admissible multi-indices,
\eql{\label{eq:gentoep-T}
	(M_{\Torus})_{\im+\sm,\jm} = (M_{\Torus})_{\im,\jm-\sm}.
}
For fixed $\jm$, as the column $(M_{\Torus})_{\cdot,\jm}$ contains the moments of the measure $e^{2i\pi\dotp{\jm}{t}}\diff\nu_{\Torus}(t)$, we say that it corresponds to the monomial $\Z^\jm$. The flatnesss property (Definition~\ref{def:flat-matrix}) is equivalent to the fact that for all $\jm \in \segi{-\ell}{\ell}^d \setminus \segi{-(\ell-1)}{\ell-1}^d$ the column $\Z^\jm$ is a linear combination of the set of columns $\enscond{Z^{\jm'}}{\jm' \in \segi{-(\ell-1)}{\ell-1}^d}$.

\paragraph{Moment matrices in the complex plane} Given a measure $\nu_{\CC}$ defined on $\CC^d$, one may consider its moments against the variable $\Z$ and its conjugate $\conj{\Z}$, namely
\eql{\label{eq:moment-C}
	(M_{\CC})_{(\im_1, \im_2), (\jm_1, \jm_2)} = \int_{\CC^d} (\conj{z}^{\im_2} z^{\im_1})(\conj{z}^{\jm_1} z^{\jm_2}) \diff \nu_{\CC}(z) = \int_{\CC^d} \conj{z}^{\jm_1+\im_2} z^{\im_1+\jm_2} \diff \nu_{\CC}(z),
}
for $\im_1, \im_2, \jm_1, \jm_2 \in \NN^d$ such that $\max(\im_1 + \im_2) \leq \ell$ and $\max(\jm_1+\jm_2) \leq \ell$. Such matrices have a structure property recalling that of Hankel matrices,
\eql{\label{eq:genhank-C}
	(M_{\CC})_{(\im_1+\mathbf{r}_1, \im_2+\mathbf{r}_2), (\jm_1, \jm_2)} = (M_{\CC})_{(\im_1, \im_2), (\jm_1+\mathbf{r}_1, \jm_2+\mathbf{r}_2)}.
}
Similarly as above, we note that the columns $(M_{\CC})_{\cdot, (\jm_1, \jm_2)}$ corresponds to the monomial $\conj{\Z}^{\jm_1}\Z^{\jm_2}$. We say that $M_{\CC}$ is \emph{flat} if the columns corresponding to $\conj{\Z}^{\jm_1}\Z^{\jm_2}$, where $\max (\jm_1+\jm_2) = \ell$, are linear combinations of the columns $\conj{\Z}^{\jm_1'}\Z^{\jm_2'}$ where $\max (\jm_1' + \jm_2') \leq \ell -1$.

\paragraph{Moment matrices in the real plane} Given a measure $\nu_{\RR^2}$ defined on $(\RR^2)^d$, we consider its moments against the variables $\Xm$ and $\Ym$, namely
\eql{\label{eq:moment-R}
	(M_{\RR^2})_{(\im_1, \im_2), (\jm_1, \jm_2)} = \int_{(\RR^d)^2} (x^{\im_1}y^{\im_2})(x^{\jm_1}y^{\jm_2}) \diff \nu_{\RR^2}(x,y) = \int_{(\RR^d)^2} (x^{\im_1+\jm_1}y^{\im_2+\jm_2}) \diff \nu_{\RR^2}(x,y)
}
for $\im_1, \im_2, \jm_1, \jm_2 \in \NN^d$ such that $\max (\im_1+\im_2) \leq \ell$ and $\max (\jm_1+\jm_2) \leq \ell$. Such moment matrices have the \emph{generalized Hankel} property, that is
\eql{\label{eq:genhank-R}
	(M_{\RR^2})_{(\im_1+\mathbf{r}_1, \im_2 + \mathbf{r}_2), (\jm_1, \jm_2)} = (M_{\RR^2})_{(\im_1, \im_2), (\jm_1+\mathbf{r}_1, \jm_2 + \mathbf{r}_2)}.
}
For fixed $\jm_1,\jm_2 \in \NN^d$, the column $(M_{\RR^2})_{\cdot, (\jm_1, \jm_2)}$ corresponds to the monomial $\Xm^{\jm_1}\Ym^{\jm_2}$. As above, we say that $M_{\RR^2}$ is \emph{flat} if the columns corresponding to $\Xm^{\jm_1}\Ym^{\jm_2}$, where $\max (\jm_1+\jm_2) = \ell$, are linear combinations of the columns $\Xm^{\jm_1'}\Ym^{\jm_2'}$ where $\max (\jm_1' + \jm_2') \leq \ell - 1$.

\subsection{From the torus to the complex plane}
Given a positive Borel measure $\nu \eqdef \nu_{\Torus}$ on $\Torus^d$, we may see it as a measure in $\CC^d$ by considering its image measure by $T$, $\nu_{\CC} \eqdef T_{\#}\nu_{\Torus}$, where
\eql{
	T : \Torus^d \rightarrow \CC^d, \quad (t_1, \ldots, t_d) \mapsto (e^{2i\pi t_1}, \ldots, e^{2i\pi t_d}).
}
The resulting measure has support in $(\SS^1)^d$, where $\SS^1 \eqdef \enscond{z \in \CC}{\abs{z}=1}$. We recall that the image measure $\nu_{\CC} = T_{\#}\nu_{\Torus}$ is characterized by $\nu_{\CC}(B) = \nu_{\Torus}(T^{-1}(B))$ for all Borel sets, so that for all $\nu_\CC$-summable function $\psi$,
\eql{\label{eq:push-fwd}
	\int_{\CC^d} \psi(z) \diff\nu_\CC(z) = \int_{\Torus^d} \psi(T(t)) \diff\nu_{\Torus}(t).
}

Obviously, the moment matrices $M_{\Torus}$ and $M_{\CC}$ have different sizes. However, the following is a first step in relating them.
\begin{lemma}
	Let $\ell \geq 2$, and let $M_\CC$ be a  moment matrix representing some positive Borel measure $\nu_\CC$ on $\CC^d$. Then $\Supp \nu_\CC \subset (\SS^1)^d$ if and only if the columns corresponding to $\conj{Z}^{\jm_1}Z^{\jm_2}$ and $\conj{Z}^{\jm_1'}Z^{\jm_2'}$ are equal for all multi-indices such that $\jm_1-\jm_2 = \jm_1'-\jm_2'$, where $\abs{\jm_1+\jm_2}_{\infty} \leq \ell$, $\abs{\jm_1'+\jm_2'}_{\infty} \leq \ell$.
\label{lem:redundancy}
\end{lemma}
\begin{proof}
The column $\conj{Z}^{\jm_1}Z^{\jm_2}$ contains elements of the form $\int_{\CC^d}(\conj{z}^{\im_1}z^{\im_2})\conj{z}^{\jm_1}z^{\jm_2} \diff\nu_\CC(z)$ for $\im_1, \im_2 \in \NN^d$. If $\Supp \nu_\CC \subset (\SS^1)^d$, then $\conj{z}^{j_{1,k}}z^{j_{2,k}} = \conj{z}^{j_{1,k}'}z^{j_{2,k}'}$ for all $z$ in the integration domain and all $k \in \segi{1}{d}$, hence
\eql{
	\int_{\CC^d}(\conj{z}^{\im_1}z^{\im_2})\conj{z}^{\jm_1}z^{\jm_2} \diff\nu_\CC(z) = \int_{\CC^d}(\conj{z}^{\im_1}z^{\im_2})\conj{z}^{\jm_1'}z^{\jm_2'} \diff\nu_\CC(z),
}
and the two columns are equal.

Conversely, if the above-mentioned columns are equal, let $k \in \segi{1}{d}$, and $\jm \in \NN^d$ such that $j_{k'} = 1$ for $k' = k$, $0$ otherwise. Then,
\eql{
	\begin{aligned}
		\int_{\CC^d} (1-\abs{z_k}^2)^2 \diff\nu_\CC(z) 
			&= \int_{\CC^d} 1 \diff\nu_\CC(z) + \int_{\CC^d} (\conj{z}^\jm z^\jm)(\conj{z}^\jm z^\jm)\diff\nu_\CC(z) - 2\int_{\CC^d} (\conj{z}^\jm z^\jm)\diff\nu_\CC(z)\\
			&= \int_{\CC^d} 1 \diff\nu_\CC(z) - \int_{\CC^d} (\conj{z}^\jm z^\jm)\diff\nu_\CC(z)\\
			&= 0,
	\end{aligned}
}
using the equality between the columns $1$ and $\conj{Z}^\jm Z^\jm$ and the corresponding relations between their entries. Since the integrand is nonnegative, we deduce that $\nu$ charges only points where $\conj{z}^\jm z^\jm = 1$, \ie $\conj{z_k} z_k = 1$. As a result, 
$$\Supp \nu \subset \bigcap_{k=1}^d \enscond{z \in \CC^d}{\conj{z_k} z_k = 1} = (\SS^1)^d.$$
\end{proof}

Using Lemma~\ref{lem:redundancy}, we see that for all $\jm \in \ZZ^d$ such that $\abs{\jm}_\infty \leq \ell$, all the columns $(M_\CC)_{\cdot, (\jm_1, \jm_2)}$ such that $\jm_2 - \jm_1 = \jm$ (and $\max (\jm_1 + \jm_2) \leq \ell$ are equal. In fact, from \eqref{eq:push-fwd} we see that those columns are obtained by "repeating" the column $\jm$ of $M_\Torus$ at all indices such that $\jm_2-\jm_1 = \jm$ (and similarly for the rows).

More precisely, given a maximal degree $\ell$, let us denote the set of Laurent polynomials by
\eql{
	\CC_{\ell}[\conj{\Z}^{\pm 1}]  \eqdef \Span \enscond{\Z^\jm}{\jm \in \ZZ^d, \abs{\jm}_{\infty} \leq \ell},
}
the set of polynomials by
\eql{
	\CC_{\ell}[\Z,\conj{\Z}] \eqdef \Span \enscond{\conj{\Z}^{\jm_1}\Z^{\jm_2}}{\jm_1, \jm_2 \in \NN^d, \max (\jm_1 + \jm_2) \leq \ell}
}
and let $J$ denote the matrix of the operator
\eql{
	\CC_{\ell}[\Z,\conj{\Z}] \rightarrow \CC_{\ell}[\conj{\Z}^{\pm 1}], \quad \conj{\Z}^{\jm_1}\Z^{\jm_2} \mapsto \frac{1}{c(\jm_2-\jm_1)} \Z^{\jm_2-\jm_1}, 
}
where $c(\jm) = \Card \enscond{(\jm_1',\jm_2')}{\jm_1', \jm_2' \in \NN^d, \max (\jm_1' + \jm_2') \leq \ell, \jm_2' - \jm_1' = \jm}$. Then, from \eqref{eq:push-fwd}, the following relation holds,
\eql{\label{eq:congruency-T-C}
	M_\CC = J^*M_\Torus J.
}
We immediately deduce:

\begin{lemma}
\label{lem:equiv-T-C}
The relation $\nu_\CC = T_\#\nu_\Torus$ defines a one-to-one correspondence between (positive Borel) measures $\nu_\Torus$ on $\Torus^d$ and measures on $\CC^d$ supported on $(\SS^1)^d$. That correspondence preserves the cardinality of the support and one has $T(\Supp \nu_\Torus) = \Supp \nu_\CC$. Moreover their moment matrices are related by \eqref{eq:congruency-T-C}.

Conversely, let $M_\Torus$ and $M_\CC$ be matrices (indexed by $\CC_{\ell}[\conj{\Z}^{\pm 1}]$ and $\CC_{\ell}[\Z,\conj{\Z}]$ respectively) such that \eqref{eq:congruency-T-C} holds. Then,
\begin{enumerate}
\item $M_\Torus$ satisfies \eqref{eq:gentoep-T} if and only if $M_\CC$ satisfies  \eqref{eq:genhank-C},
\item $M_\CC \succeq 0$ if and only if $M_\Torus \succeq 0$,
\item $\rank M_\Torus = \rank M_\CC$. Moreover $M_\Torus$ is flat if and only if $M_\CC$ is flat.
\end{enumerate}
\end{lemma}

\subsection{From the complex plane to the real plane}
Now, given a positive Borel measure $\nu_{\CC}$ on $\CC^d$, we see it as a measure on $(\RR^2)^d$ by considering its image measure by $S$, $\nu_{\RR^2} \eqdef S_{\#}\nu_{\CC}$, where
\eql{
	S : \CC^d \rightarrow (\RR^2)^d, \quad (z_1, \ldots, z_d) \mapsto \left( \frac{z_1+\conj{z_1}}{2}, \frac{z_1-\conj{z_1}}{2i}, \ldots, \frac{z_d+\conj{z_d}}{2}, \frac{z_d-\conj{z_d}}{2i} \right).
}
We consider the following subspace of polynomials in the variables $X_1, Y_1, \ldots, X_d, Y_d$, 
\eql{
	\CC_{\ell}[\Xm,\Ym] \eqdef \Span \enscond{\Xm^{\jm_1}\Ym^{\jm_2}}{\jm_1, \jm_2 \in \NN^d, \max (\jm_1 + \jm_2) \leq \ell}.
}

Obviously, $\CC_{\ell}[\Z,\conj{\Z}]$ and $\CC_{\ell}[\Xm, \Ym]$ are isomorphic as vector spaces. We are interested in the relations between the moment matrices $M_{\RR^2}$ and $M_\CC$ when changing variables with $S$, that is
\begin{align}
	&\forall k \in \{1, \ldots, d\}, \quad Z_k \eqdef X_k + iY_k, \quad \conj{Z_k} \eqdef X_k - iY_k, \quad \text{or conversely}, \label{eq:change1}\\
	&\forall k \in \{1, \ldots, d\}, \quad X_k \eqdef \frac{1}{2}(Z_k+\conj{Z_k}), \quad Y_k \eqdef \frac{1}{2i}(Z_k - \conj{Z_k}).
\end{align}

We note that for all $k \in \{1, \ldots, d\}$, given some indices $i_k, j_k$,
\eq{
	\conj{Z_k}^{j_{1,k}}Z_k^{j_{2,k}} = (X_k + iY_k)^{j_{1,k}} (X_k - iY_k)^{j_{2,k}} = \sum_{r_1, r_2} c_{r_1, r_2} X_k^{r_1}Y_k^{r_2},
}
where $c_{r_1, r_2} \in \CC$ and the sum is over all the indices $r_1, r_2 \in \NN$ such that $r_1 + r_2 = j_{1,k} + j_{2,k}$. As a result, given $\jm_1, \jm_2 \in \NN^d$,
\eql{
	\conj{\Z}^{\jm_1}\Z^{\jm_2} = \prod_{k=1}^d (X_k + iY_k)^{j_{1,k}} (X_k - iY_k)^{j_{2,k}} = \sum_{\mathbf{r}_1, \mathbf{r}_2} c_{\mathbf{r}_1, \mathbf{r}_2} \Xm^{\mathbf{r}_1}\Ym^{\mathbf{r}_2},
}
where the sum is over all the multi-indices $\mathbf{r}_1, \mathbf{r}_2 \in \NN^d$ such that, for all $k$, $r_{1,k} + r_{2,k} = j_{1,k} + j_{2,k}$.

As a result, the change of variable \eqref{eq:change1} induces a linear map $L : \CC_\ell[\Xm,\Ym] \rightarrow \CC_\ell[\Z,\conj{\Z}]$ which admits a block decomposition, mapping surjectively (hence bijectively), for each $\mathbf{t} \in \NN^d$ with $\max (\mathbf{t}) \leq \ell$, the space
\begin{align*}
		&\Span \enscond{\Xm^\im\Ym^\jm}{\im,\jm \in \NN^d, \quad \im + \jm = \mathbf{t}}\\
		\quad \text{onto} \quad &\Span \enscond{\conj{\Z}^\im\Z^\jm}{\im,\jm \in \NN^d, \quad \im + \jm = \mathbf{t}}
\end{align*}
Confronting \eqref{eq:moment-C} and \eqref{eq:moment-R}, in view of the change of variable formula
\eql{
	\int_{(\RR^d)^2} \psi(x,y) \diff\nu_{\RR^2}(x,y) = \int_{\CC^d} \psi(S(z)) \diff\nu_\CC(z), 
}
we deduce that
\eql{\label{eq:congruency-C-R}
	M_{\RR^2} = L^* M_\CC L,
}
where $L$ is invertible.

\begin{lemma} 
\label{lem:equiv-C-R}
The relation $\nu_{\RR^2} = S_\#\nu_\CC$ defines a one-to-one correspondence between (positive Borel) measures $\nu_\CC$ on $\CC^d$ which represent $M_\CC$ ad measures $\nu_{\RR^2}$ on $(\RR^d)^2$ which represent $M_{\RR^2}$. That correspondence preserves the cardinality of the support and $S(\Supp \nu_\CC) = \Supp \nu_{\RR^2}$.

Conversely, let $M_\CC$ and $M_{\RR^2}$ be matrices (indexed by $\CC_\ell[\Z,\conj{\Z}]$ and $\CC_\ell[\Xm,\Ym]$ respectively) such that \eqref{eq:congruency-C-R} holds. Then,
\begin{enumerate}
\item $M_\CC$ satisfies \eqref{eq:genhank-C} if and only if $M_{\RR^2}$ satisfies \eqref{eq:genhank-R}, 
\item $M_{\RR^2} \succeq 0$ if and only if $M_\CC \succeq 0$,
\item $\rank M_\CC = \rank M_{\RR^2}$. Moreover $M_\CC$ is flat if and only if $M_{\RR^2}$ is flat.
\end{enumerate} 
\end{lemma}

\subsection{Existence of a representing measure}
The rest of the proof of Theorem~\ref{thm:flatness} relies on the results of \cite{laurent09}. The matrix $R$ is a positive semi-definite matrix indexed by the set of monomials $\enscond{\Z^\jm}{\jm \in \ZZ^d, \abs{\jm}_{\infty} \leq \ell}$ which satisfies \eqref{eq:gentoep-T}.

Applying successively \eqref{eq:congruency-T-C} and \eqref{eq:congruency-C-R}, we obtain a matrix $N \eqdef L^*J^*RJL$, which satisfies \eqref{eq:genhank-R}, indexed by the set of monomials 
$$\Cc_\ell \eqdef \enscond{\Xm^{\jm_1}\Ym^{\jm_2}}{\jm_1, \jm_2 \in \NN^d, \max (\jm_1 + \jm_2) \leq \ell}.$$

We note that $\Cc_{\ell-1}$ is closed under taking divisors, hence connected to $1$ (see \cite[Sec. 1.3]{laurent09} for the definition). Moreover, its closure,
\begin{align}
	\Cc_{\ell-1}^+ &\eqdef \Cc_{\ell-1} \cup \left( \bigcup_{k=1}^d X_k \Cc_{\ell-1} \right) \cup \left( \bigcup_{k=1}^d Y_k\Cc_{\ell-1} \right)\\
			&= \enscond{m, X_1m, \ldots, X_dm, Y_1m, \ldots, Y_dm}{m \in \Cc_{\ell-1}}
\end{align}
is a subset of $\Cc_\ell$. The matrix $\restr{N}{\Cc_\ell}$ being a flat extension of $\restr{N}{\Cc_{\ell-1}}$ (by Lemma~\ref{lem:equiv-T-C} and \ref{lem:equiv-C-R}), this implies that $\restr{N}{\Cc_{\ell-1}^+}$ is a flat extension of $\restr{N}{\Cc_{\ell-1}}$.

The theorem \cite[Th. 3.2]{laurent09} by Laurent and Mourrain then ensures that there exists a representing measure $\nu_{\RR^2}$ for $\restr{N}{\Cc_{\ell-1}^+}$. As the theorem does not say anything about the rows and columns in $\Cc_\ell \setminus \CC_{\ell-1}^+$, we show below that $\nu_{\RR^2}$ actually represents all the entries of $N$.

\begin{lemma}
The measure $\nu_{\RR^2}$ represents $N$.
\end{lemma}
\begin{proof}
Let us write
\eql{
	R = \begin{blockarray}{ccc} & \Cc_{\ell-1} & \Cc_{\ell} \setminus \Cc_{\ell-1}\\
		\begin{block}{c(cc)}
			\Cc_{\ell-1} 		        			& A     & B \\
			\Cc_{\ell} \setminus \Cc_{\ell-1} 	& B^* & C \\
		\end{block} 
	\end{blockarray}
	\eqdef N.
}
A first step is to show that $B = AQ$ for some matrix $Q$ which only depends on $\restr{N}{\Cc_{\ell-1}^+}$, the restriction of $N$ to $\Cc_{\ell-1}^+ \times \Cc_{\ell-1}^+$. Then, since $N \succeq 0$ is flat, we deduce by \cite[Prop. 2.2]{curto96} that $C$ is uniquely determined from $A$ and $B$, as $C = Q^*AQ$.

In a second step, we consider the moment matrix $M_{\RR^2}$ of $\nu_{\RR^2}$ on $\Cc_\ell$. As it is flat, positive semi-definite and generalized Hankel, its block satisfy a similar property as those of $N$, involving some matrix $\tilde{Q}$ which depends on $\restr{M_{\RR^2}}{\Cc_{\ell-1}^+}$. The key point is that since $M_{\RR^2}$ and $N$ coincide in $\Cc_{\ell-1}^+ \times \Cc_{\ell-1}^+$, the matrices $Q$ and $\tilde{Q}$ are equal, hence the whole matrices $N$ and $M_{\RR^2}$ are equal.

The main point is therefore to prove that $B = AQ$ (the argument for $M_{\RR^2}$ being similar). To lighten the notation, we write $\tilde{\Xm}^\jm \eqdef \Xm^{\jm_1}\Ym^{\jm_2}$ with $\jm = (\jm_1,\jm_2)$. For $k \in \{1, \ldots, d\}$, we denote by $\e^{(k)} = (\e_1^{(k)}, \e_2^{(k)}) \in \NN^d \times \NN^d$ any multi-index such that
\eql{
	\e_{1,n}^{(k)} = \e_{2,n}^{(k)} = 0 \quad \text{for} \; n \neq k, \qandq \left(\e_1^{(k)}, \e_2^{(k)}\right) = (1,0) \; \text{or} \; (0,1).
}
The elements of $\Cc_{\ell-1}^+ \setminus \Cc_{\ell-1}$ are of the form $\tilde{\X}^{\jm+\e^{(k)}}$ with $\max (\jm_1 + \jm_2) = \ell - 1$, whereas the elements of $\Cc_{\ell} \setminus \Cc_{\ell-1}$ have the form $\tilde{\X}^{\jm+\e^{(k_1)}+\ldots+\e^{(k_n)}}$ where $k_1, \ldots, k_n$ are distinct elements of $\{1, \ldots, d\}$.

Let $\jm_1,\jm_2 \in \NN^d$ with $\max (\jm_1 + \jm_2) \leq \ell-1$. Since $\restr{N}{\Cc_{\ell-1}^+}$ is a flat extension of $\restr{N}{\Cc_{\ell-1}}$, any column of $\restr{N}{\Cc_{\ell-1}^+}$ of the form $\tilde{\Xm}^{\jm+\e^{(k_1)}}$ is a linear combination of the columns in $\CC_{\ell-1}$. Let $q_{\jm,\e^{(k_1)}}$ (resp. $q_{\jm,\e^{(k_1)}}(\Xm,\Ym)$) denote the coefficients of that combination (resp. the corresponding polynomial\footnote{In the following we identify polynomials with the vectors of their coefficients, so that we can ``apply a moment matrix to a polynomial'' with suitable degree}). In other words, the polynomial $f^{(1)} \eqdef (\tilde{\X})^{\jm+\e^{(k_1)}} -  q_{\jm,\e^{(k_1)}}(\Xm,\Ym)$ is in $\ker \restr{N}{\Cc_{\ell-1}^+}$. Since $N \succeq 0$, we deduce that in fact $f^{(1)} \in \ker N$. 

For $2 \leq n \leq d$, consider the polynomial 
$$f^{(n)} \eqdef (\tilde{\Xm})^{\jm+\e^{(k_1)}+\ldots+\e^{(k_n)}} - \tilde{\Xm})^{\e^{(k_2)} + \ldots + \e^{(k_n)}}q_{\jm,\e^{(k_1)}}(\Xm,\Ym) \in \Cc_\ell.$$ 
As the $k_i$'s are pairwise distinct, we see that for $\im \in \Cc_{\ell-1}$, we have 
$$(\tilde{\Xm})^{\jm+\e^{(k_1)}+\ldots+\e^{(k_n)}} \in \Cc_\ell.$$
Hence, proceeding as in \cite[Lem. 5.7]{laurent10}, we get
\eql{\label{eq:recursiveness}
	(Nf^{(n)})_\im = (Nf^{(1)})_{\im+\e^{(k_2)} + \ldots + \e^{(k_n)}} = 0,
}
since  $f^{(1)} \in \ker N$.

From~\eqref{eq:recursiveness} and the definition of $f^{(n)}$, we deduce that the column $(\tilde{\Xm})^{\jm+\e^{(k_1)}+\ldots+\e^{(k_n)}}$ of $B$ is a linear combination of columns of the form $(\tilde{\Xm})^{\jm'+\e^{(k_1')}+\ldots+\e^{(k_{n-1}')}}$ with $\max (\jm') \leq \ell-1$. By an easy induction we obtain that the column $(\tilde{\Xm})^{\jm+\e^{(k_1)}+\ldots+\e^{(k_n)}}$ of $N$ is thus a linear combination of columns of the form $(\tilde{\Xm})^{\jm'+\e^{(k_1')}}$. By flatness of $\Cc_{\ell-1}^+$, it is thus a combination of columns corresponding to $\Cc_{\ell-1}$.

As a result, there exists a matrix $Q$ such that $B = AQ$, and $Q$ is uniquely determined from th e polynomials $q_{\jm,\e^{(k_1)}}$ for $\jm \in \Cc_{\ell-1}$, $1 \leq k_1 \leq d$. Since the same polynomials $q_{\jm,\e^{(k_1)}}$ can be used for $M_{\RR^2}$, we obtain that $\tilde{Q} = Q$ hence $M_{\RR^2} = N$.
\end{proof}

Now, we may go back to the torus. We observe that for $1\leq k \leq d$, the polynomial $X_k^2 + Y_k^2 - 1$ is in the kernel of $N = M_{\RR^2}$, hence $\nu_{\RR^2}$ is supported in $\bigcap_{k=1}^d \enscond{(x,y) \in (\RR^d)^2}{ x_k^2 + y_k^2 = 1}$. Applying Lemma~\ref{lem:equiv-C-R} and \ref{lem:equiv-T-C} above, we obtain the existence of a measure $\nu_{\Torus}$ such that $R$ is the moment matrix of $\nu_\Torus$ on $\Torus^d$, which is the claimed result.

\section*{Acknowledgments}
We would like to thank C\'edric Josz for letting us use his implementation of the extraction procedure.

\bibliographystyle{siamplain}
\bibliography{LowRankFFW}
\end{document}